\documentclass[12pt]{article}
	
	\usepackage{amssymb,amsmath,amsthm,amsfonts,amscd,latexsym, dsfont, color, xcolor}
	\usepackage[colorlinks]{hyperref} 
	\usepackage{graphics}
	\usepackage{wrapfig}
	\usepackage{pst-node}
	\usepackage{tikz-cd} 
	\usepackage{enumitem}
	\usepackage{mathrsfs}
	\newcommand{\bbR}{\mathbb R}
	
	\newcommand{\bbZ}{\mathbb Z}
	
	\newcommand{\bbC}{\mathbb C}

	\newcommand{\del}{\partial}

	\usepackage{cleveref}

\DeclareFontFamily{U}{mathx}{\hyphenchar\font45}
\DeclareFontShape{U}{mathx}{m}{n}{<-> mathx10}{}
\DeclareSymbolFont{mathx}{U}{mathx}{m}{n}
\DeclareMathAccent{\widebar}{0}{mathx}{"73}
	\newcommand{\delbar}{\widebar{\partial}}

\crefformat{section}{\S#2#1#3} 
\crefformat{subsection}{\S#2#1#3}
\crefformat{subsubsection}{\S#2#1#3}
	\allowdisplaybreaks

\theoremstyle{definition}

\usepackage{thmtools, thm-restate}
\theoremstyle{plain}

\newtheorem{theorem}{Theorem}[section]
\newtheorem{proposition}[theorem]{Proposition}
\newtheorem{lemma}[theorem]{Lemma}
\newtheorem{corollary}[theorem]{Corollary}

\newtheorem{example}[theorem]{Example}

\newtheorem{definition}[theorem]{Definition}
\newtheorem*{remark}{Remark}
\newtheorem*{remarks}{Remarks}
\numberwithin{equation}{section}
\newtheorem*{convention}{Convention}
	
	\usepackage[margin=1in]{geometry}
\begin{document}

\title{Canonical Maps from Spaces of Higher Complex Structures to Hitchin Components}
\author{Alexander Nolte$^1$\footnote{This material is based upon work supported by the National Science Foundation under Grant No. 1842494 and Grant No. 2005551}}
\date{$^1$Rice University}

\maketitle

\begin{abstract}
    For $S$ a closed surface of genus $g\geq2$, we construct a canonical diffeomorphism from the degree $3$ Fock-Thomas space $\mathcal{T}^3(S)$ of higher complex structures to the $\text{SL}(3,\bbR)$ Hitchin component. Our construction is equivariant with respect to natural actions of the mapping class group $\text{Mod}(S)$. For all $n \geq 3$, we show that the Fock-Thomas space $\mathcal{T}^n(S)$ has a canonical vector bundle structure over Teichm\"uller space. We then construct a $\text{Mod}(S)$-equivariant bundle isomorphism from $\mathcal{T}^n(S)$ to a sub-bundle of the restriction of the tangent bundle of the $\text{PSL}(n, \mathbb{R})$ Hitchin component to the Fuchsian locus. As consequences, we prove that the higher degree moduli space of complex structures is a bundle over the moduli space of Riemann surfaces and that the action of $\text{Mod}(S)$ on $\mathcal{T}^n(S)$ is a proper action by holomorphic automorphisms with respect to a canonical complex structure. The core of our approach is a careful analysis of higher degree diffeomorphism groups.
\end{abstract}

    \section{Introduction}
    
    Let $S$ be a closed, oriented, smooth surface of genus $g \geq 2$. For $G$ an adjoint group of the split real form of a complex simple connected Lie group, the Hitchin component $\text{Hit}(S,G)$ of the character variety $\text{Rep}(\pi_1(S), G)$ is the connected component containing the Fuchsian representations. Here, Fuchsian representations $\pi_1(S) \to G$ are discrete and faithful representations with image contained in a principal $\text{SL}(2, \mathbb{R})$ (see \cite{hitchin1992lie}).
    
    In seminal work \cite{hitchin1992lie}, Hitchin used Higgs bundles \cite{hitchin1987self} to show that $\text{Hit}(S,G)$ is diffeomorphic to $\bbR^{-\chi(S) \dim(G)}$. A guiding question in the study of Hitchin components is to what extent and in what ways $\text{Hit}(S,G)$ admits geometric interpretations analogous to geometric interpretations of Teichm\"uller spaces $\mathcal{T}(S)$.
    
    One approach to finding geometric descriptions of $\text{Hit}(S, G)$ that has seen substantial progress is to generalize the description of Teichm\"uller spaces in terms of hyperbolic structures on $S$ by realizing Hitchin representations as holonomies of $(G, X)$ structures on manifolds associated to $S$. At the time \cite{hitchin1992lie} was published, it was known from work of Goldman \cite{goldman1990convex} that holonomies of marked properly convex projective structures on S were an open subset of $\text{Hit}(S,\text{SL}(3, \bbR))$, and closedness was proved soon afterwards by Choi and Goldman \cite{choi1993convex}. Guichard and Weinhard have since interpreted $\text{Hit}(S,\text{PSL}(4, \bbR))$ and $\text{Hit}(S,\text{PSp}(4, \bbR))$ in terms of $(G, X)$ structures on the unit tangent bundle $T^1(S)$ in \cite{guichard2008convex}, and found $(G, X)$ structures interpretations for $\text{Hit}(S,G)$ in general in \cite{guichard2012anosov}.
    
    In \cite{fock2021higher}, a conjectural approach to generalizing the description of Teichm\"uller spaces in terms of complex structures on a surface was proposed by Fock and Thomas. Their higher-degree analogues to complex structures are sections of bundles of punctual Hilbert schemes (see Section \ref{section-jets-and-symplectos}).
    
    We denote the collection of degree-$n$ complex structures on $S$ by $\mathbb{M}^n(S)$ and the collection of complex structures on $S$ by $\mathbb{M}(S)$. Symplectic diffeomorphisms of the cotangent bundle $T^*(S)$ that setwise fix the zero section act on $\mathbb{M}^n(S)$. Denote by $\text{Ham}^0_c(T^*S)$ the group of compactly supported Hamiltonian diffeomorphisms of $T^*(S)$ generated by compactly supported Hamiltonian flows setwise fixing the zero section. Fock and Thomas investigate ${\mathcal{T}}^n(S) = \mathbb{M}^n(S)/\text{Ham}_c^0(T^*S)$, which we shall call the \textit{degree-$n$ Fock-Thomas space of $S$}.
    
    The degree-$2$ Fock-Thomas space ${\mathcal{T}}^2(S)$ is canonically diffeomorphic to $\mathcal{T}(S)$, in the sense that a $\text{Mod}(S)$-equivariant diffeomorphism is obtained through a $\text{Diff}_0(S)$-equivariant construction that identifies $\mathbb{M}(S)$ with $\mathbb{M}^2(S)$.
    
   Fock and Thomas conjecture in \cite{fock2021higher} that ${\mathcal{T}}^n(S)$ is canonically diffeomorphic to the Hitchin component $\text{Hit}(S,\text{PSL}(n, \bbR))$ for $n \geq 2$. We shall prove this for $n =3$ and give an interpretation of ${\mathcal{T}}^n(S)$ in terms of the Hitchin component for all $n\geq 3$.
    
    \begin{theorem}\label{headline-result} Let $\mathcal{T}^n(S)$ denote the degree-$n$ Fock-Thomas space of $S$, and let $\mathcal{F}^n(S)$ denote the Fuchsian locus of {\rm{$\text{Hit}(S, \text{PSL}(n,\mathbb{R}))$}}.
    \begin{enumerate}[label=(\alph*).] \item There is a canonical {\rm{$\text{Mod}(S)$}}-equivariant diffeomorphism {\rm{${\mathcal{T}}^3(S) \to \text{Hit}(S,\text{SL}(3,\bbR))$}}. 
    \item For all $n \geq 3$, ${\mathcal{T}}^n(S)$ has a natural vector bundle structure over $\mathcal{T}(S).$ There is a sub-bundle {\rm{$\mathcal{L}^n(S) \subset (T \text{Hit}(S, \text{PSL}(n, \mathbb{R}))|_{\mathcal{F}^n(S)}$}} and a canonical {\rm{$\text{Mod}(S)$}}-equivariant bundle isomorphism $\mathcal{T}^n(S) \to \mathcal{L}^n(S)$.
    \end{enumerate}
  \end{theorem}
    
    The maps in Theorem \ref{headline-result} are canonical in the following sense, loosely described. We construct elements of a vector bundle $\mathcal{H}\mathbb{M}^n(S)$ over $\mathbb{M}(S)$ from elements of $\mathbb{M}^n(S)$. This construction depends on no choices of reference data and is equivariant with respect to actions of $\text{Diff}_0(S)$ on $\mathbb{M}^n(S)$ and $\mathcal{H}\mathbb{M}^n(S)$. The maps of statements (a) and (b) are then obtained through $\text{Diff}_0(S)$-invariant constructions of elements of $\text{Hit}(S, \text{SL}(3, \bbR))$ and $\mathcal{L}^n(S)$ from elements of $\mathcal{H} \mathbb{M}^n(S)$, and are $\text{Mod}(S)$-equivariant.
    
    Propositions \ref{free-from-labourie-a}-\ref{free-from-labourie-b} document relationships between the maps of Theorem \ref{headline-result} and the geometry of $\text{Hit}(S, \text{PSL}(n,\bbR))$, and we explain the construction of these maps in more depth below.

    To give a more detailed description of the construction of the maps in Theorem \ref{headline-result}, let $\mathcal{G}^n(S)$ be the quotient of $\text{Ham}^0_c(T^*S)$ by the stabilizer of its action on $\mathbb{M}^n(S)$. We find a decomposition of $\mathcal{G}^n(S)$ as an internal semidirect product $\text{Diff}_0(S) \ltimes \mathcal{N}(S)$, where $\mathcal{N}(S)$ is the subgroup of $\mathcal{G}^n(S)$ that fixes every element of $\mathbb{M}^2(S)$ (Theorem \ref{structure-of-higher-diffeos}). Viewing $\mathcal{T}^n(S)$ as a ``two step" quotient $ (\mathbb{M}^n(S)/ \mathcal{N}(S))/\text{Diff}_0(S)$, we use the structure of $\mathcal{N}(S)$ to inductively apply a Hodge decomposition theorem for appropriate tensors, producing distinguished representatives of $\mathcal{N}(S)$-orbits of $\mathbb{M}^n(S)$ (Theorem \ref{harmonic-representatives}).
    
    This construction identifies $\mathbb{M}^n(S)/\mathcal{N}(S)$ with a vector bundle $\mathcal{H} \mathbb{M}^n(S)$ over $\mathbb{M}(S)$. Here, the fiber of $\mathcal{H} \mathbb{M}^n(S)$ over $\Sigma \in \mathbb{M}(S)$ consists of tuples $(\mu_3, ..., \mu_n)$ with $\mu_k \in \Gamma(K^{1-k}_\Sigma \overline{K}_\Sigma)$ satisfying that $\overline{\mu}_k {g_\Sigma}^{k-1}$ is a holomorphic $k$-adic differential on $\Sigma$ for $3\leq k\leq n$, where $g_\Sigma$ is the hyperbolic metric in the conformal class of $\Sigma$ and $K_\Sigma$ is the canonical bundle of $\Sigma$. We call such tensors \textit{harmonic $k$-Beltrami differentials}, and elements of $\mathcal{H}\mathbb{M}^n(S)$ \textit{harmonic $n$-complex structures}. A feature of our construction is that the induced action of $\text{Diff}_0(S)$ on $\mathcal{H} \mathbb{M}^n(S)$ is by pullback, so the identification descends to a diffeomorphism of the quotients $\mathcal{T}^n(S)$ and $\mathcal{B}^n(S) = \mathcal{H} \mathbb{M}^n(S)/\text{Diff}_0(S)$ (Theorem \ref{main-bundle-thm}).

    Statement (a) is then obtained by using the data of the harmonic representatives of $\mathcal{N}(S)$-orbits to develop affine spheres in $\mathbb{R}^3$ equivariant under Hitchin representations following Labourie \cite{labourie2007flat} and Loftin \cite{loftin2001affine}. Statement (b) is obtained by viewing harmonic $n$-complex structures as infinitesimal deformations of connections with Fuchsian holomomy using Higgs bundles.

    One aspect of our approach worth remarking upon is that we only directly associate information about Hitchin components to harmonic $n$-complex structures (Definition \ref{def-harmonic-n-complex}). It would be interesting to understand to what extent general $n$-complex structures admit analogous interpretations.

    Another approach to the conjecture of Fock and Thomas has been initiated by Thomas in \cite{thomas2020higher} and \cite{thomas2021differential}. He seeks to create an analogue of the non-Abelian Hodge correspondence for the cotangent bundle $T^*\mathcal{T}^n(S)$. In this approach, from a covector $v \in T^*\mathcal{T}^n(S)$, a vector bundle $V$ and pair of commuting nilpotent $\text{End}(V)$-valued one-forms are constructed. It is conjectured by Thomas that there is a canonical parametrization of $\text{Hit}(S, \text{PSL}(n,\bbR))$ as monodromies of connections on deformations of $V$. We follow a different route, focusing on clarifying the intrinsic structure of $\mathcal{T}^n(S)$, then using the picture developed to guide the construction of our maps to Hitchin components.
    
    A number of subsidiary results about the structure of $\mathcal{T}^n(S)$ follow from our main results, which we explain in the following subsections.
    
    \subsection{Complex Structure, K\"ahler Metrics, and Compatibility with the Geometry of Hitchin Components}

    The relationship of the maps constructed in Theorem \ref{headline-result} to the map that appears in Labourie's conjecture \cite{labourie2004anosov} allows for some results obtained in the study of Hitchin components (\cite{labourie2017cyclic}, \cite{labourie2018variations}) to be applied to clarify the geometry of $\mathcal{T}^n(S)$ and to describe how the maps in Theorem \ref{headline-result} interact with the geometry of $\text{Hit}(S, \text{PSL}(n,\bbR))$. The relevant results used are discussed in Section \ref{section-formal-hitchin-link}. We begin with the intrinsic geometry of $\mathcal{T}^n(S)$.
    
    For this, we introduce some notation. Let $\mathcal{B}(k,\Sigma)$ denote subspace of the fiber of $\mathcal{B}^n(S)$ over $\Sigma \in \mathcal{T}(S)$ consisting of classes of harmonic $k$-Beltrami differentials. The Petersson $L^2$ pairing on harmonic $k$-Beltrami differentials $\mu, \nu \in \mathcal{B}(k, \Sigma)$ is given by $\langle \mu, \nu \rangle = \int_\Sigma \mu \overline{\nu} g_\Sigma^{k-1}$. Combining our result canonically identifying $\mathcal{T}^n(S)$ and $\mathcal{B}^n(S)$ (Theorem \ref{main-bundle-thm}) and work of Labourie \cite{labourie2018variations} and Kim-Zhang \cite{kim2017kahler}, we obtain:
    
    \begin{theorem}\label{free-from-labourie}
        The degree-$n$ Fock-Thomas space $\mathcal{T}^n(S)$ has a canonical complex structure and a real dimension $n-2$ family $\mathcal{K}^n(S)$ of {\rm{$\text{Mod}(S)$}}-invariant real-analytic K\"ahler metrics. Every $h \in \mathcal{K}^n(S)$ is compatible with this complex structure and realizes the zero section of $\mathcal{T}^n(S)$, identified with $\mathcal{B}^n(S)$, as a totally geodesic submanifold on which the induced metric coincides with the Weil-Petersson metric on $\mathcal{T}(S)$.
        
        After identifying $\mathcal{T}^n(S)$ with $\mathcal{B}^n(S)$, on the vertical subspaces $\mathcal{B}(k,\Sigma)$ of the tangent spaces to the fibers of $\mathcal{T}^n(S)$ every $h \in \mathcal{K}^n(S)$ restricts to a multiple of the Petersson $L^2$ pairing. For every $h \in \mathcal{K}^n(S)$, the subspaces $\mathcal{B}(k, \Sigma)$ and $\mathcal{B}(j,\Sigma)$ are orthogonal for $j \neq k$.
    \end{theorem}
    
    We remark that as the K\"ahler metrics in Theorem \ref{free-from-labourie} are real-analytic (see Appendix \ref{appendix-kahler-metrics-analytic}), it follows from work of Feix \cite{feix2001hyperkahler} and Kaledin  \cite{kaledin1997hyperkaehler} that $T^*\mathcal{T}^n(S)$ admits hyperk\"ahler structures on neighborhoods of its zero section, confirming a conjecture of Thomas (\cite{thomas2020higher}, Conjecture 18.1). In \cite{fock2021higher}, a formal argument for the existence of an almost complex structure on $\mathcal{T}^n(S)$ appears.
    
    We also remark that the existence of a natural inclusion of $\mathcal{T}^2(S) \to \mathcal{T}^n(S)$ was shown in \cite{thomas2020thesis}. The image of this inclusion corresponds to the zero section of $\mathcal{T}^n(S)$ under its identification with $\mathcal{B}^n(S)$. 

    In \cite{kim2017kahler} and \cite{labourie2017cyclic}, Kim-Zhang and Labourie construct a real $1$-parameter family of K\"ahler metrics on $\text{Hit}(S, \text{SL}(3, \bbR))$ which are all compatible with an appropriate complex structure. Let this family of metrics be denoted by $\{ h_t\}$. The construction of the maps in Theorem \ref{headline-result} (a) is compatible with these metrics.

    \begin{proposition}\label{free-from-labourie-a}
     Let $\Phi_3$ be the diffeomorphism of Theorem \ref{headline-result} (a). The map $\Phi_3$ is holomorphic with respect to the same complex structure on {\rm{$\text{Hit}(S, \text{SL}(3,\bbR))$}} that the K\"ahler metrics $\{h_t\}$ constructed by Labourie and Kim-Zhang are compatible with, and $\mathcal{K}^3(S) = \{\Phi_3^* h_t \}$.
        \end{proposition}

The pressure metric $P$, as considered in \cite{bridgeman2015pressure}, is a $\text{Mod}(S)$-invariant metric on the $\text{PSL}(n, \bbR)$ Hitchin component $\text{Hit}(S, \text{PSL}(n,\bbR))$. Little is currently known about the pressure metric in general, but an explicit formula along the Fuchsian locus has been found by Labourie and Wentworth in \cite{labourie2018variations}. The construction of the maps in Theorem \ref{headline-result} (b) and the results of \cite{labourie2018variations} imply the following compatibility between the K\"ahler metrics of Theorem \ref{free-from-labourie} and the pressure metric.

\begin{proposition}\label{free-from-labourie-b}
        Let $\Psi_n$ be a bundle isomorphism from Theorem \ref{headline-result} (b). Let $P$ be the pressure metric on {\rm{$\text{Hit}(S, \text{PSL}(n,\bbR))$}}, and let $||\cdot||_P$ denote the Finsler norm induced by the pullback of $\Psi_n^* P$. Then for every $h \in \mathcal{K}^n(S)$, there are constants $C(k,h)$ so that for any $\Sigma \in \mathcal{T}(S)$ and $\mu_k \in \mathcal{B}(k, \Sigma)$, we have $h(\mu_k, \mu_k) = C(h,k) ||\mu_k||_P^2$. Here, $\mathcal{B}(k,\Sigma)$ is viewed as a subspace of the restriction of $T \mathcal{T}^n(S)$ to the zero section.
        
        For harmonic Beltrami differentials $\mu_k \in \mathcal{B}(k, \Sigma)$ and $\mu_j \in \mathcal{B}(j,\Sigma)$ with $k \neq j$, we have $h(\mu_k, \mu_j) = (\Psi_n^* P)(\mu_k,\mu_j) = 0$.
    \end{proposition}
    \subsection{The Mapping Class Group Action on Fock-Thomas Spaces}

    One effect of Theorem \ref{headline-result} is to $\text{Mod}(S)$-equivariantly identify $\mathcal{T}^n(S)$ with a vector bundle over $\mathcal{T}(S)$ on which the mapping class group acts by bundle isomorphisms that restrict to the standard $\text{Mod}(S)$ action on $\mathcal{T}(S)$. This establishes an analogue of Labourie's conjecture \cite{labourie2004anosov} for $\mathcal{T}^n(S)$. As a consequence, we obtain results for $\mathcal{T}^n(S)$ analogous to what a positive resolution to Labourie's conjecture would imply for $\text{Hit}(S, \text{PSL}(n,\bbR))$. In particular, our identification $\mathcal{T}^n(S) \to \mathcal{B}^n(S)$ descends to an identification of the respective quotients by $\text{Mod}(S)$. Let $\mathcal{M}(S) = \mathcal{T}(S)/\text{Mod}(S)$ denote the moduli space of Riemann surfaces.
    
    \begin{theorem}\label{moduli-identification}
        The moduli space of $n$-complex structures {\rm{$\mathcal{M}^n(S) = \mathcal{T}^n(S)/\text{Mod}(S)$}} is canonically identified with the bundle {\rm{$\mathcal{B}^n(S)/\text{Mod}(S)$}} over $\mathcal{M}(S)$ of tuples of classes of harmonic Beltrami differentials.
    \end{theorem}
    
    In \cite{labourie2004anosov}, Labourie proves that Anosov representations are well-displacing and uses this to show that the action of $\text{Mod}(S)$ on $\text{Hit}(S, \text{PSL}(n,\bbR))$ is proper. As a consequence of Theorem \ref{moduli-identification} and the classical theorem of Fricke, an analogous theorem holds for the action of $\text{Mod}(S)$ on $\mathcal{T}^n(S)$.
    
    \begin{theorem}
        The action of {\rm{$\text{Mod}(S)$}} on $\mathcal{T}^n(S)$ is proper.
    \end{theorem}

    Another consequence of Theorem \ref{headline-result} is that if a natural $\text{Mod}(S)$-equivariant diffeomorphism were found between $\text{Hit}(S, \text{PSL}(n, \mathbb{R}))$ and $\mathcal{T}^n(S)$, one would obtain a $\text{Mod}(S)$-equivariant identification of $\text{Hit}(S, \text{PSL}(n, \mathbb{R}))$ and $\mathcal{B}^n(S)$. This would be a partial solution to Labourie's conjecture. Negative evidence on Labourie's conjecture in rank $3$ has recently been found by Markovi\'{c} \cite{markovic2022non}.
    
    \subsection{Auxiliary Results}
    From the identification of $\mathcal{T}^n(S)$ and $\mathcal{B}^n(S)$ (Theorem \ref{main-bundle-thm}) underlying Theorem \ref{headline-result} and a well-known corollary of the Riemann-Roch theorem, we may determine the diffeomorphism type of $\mathcal{T}^n(S)$.
    
        \begin{corollary}\label{diffeo-type} ${\mathcal{T} }^n(S)$ is diffeomorphic to {\rm{$\bbR^{-\chi(S) \text{dim}(\text{PSL}(n,\bbR))} = \bbR^{(2g-2)(n^2-1)}$}}. \end{corollary}
    
    \begin{remark}A mildly weaker statement appears in \cite{fock2021higher}: that $\mathcal{T}^n(S)$ is a contractible manifold of dimension {\rm{$-\chi(S) \dim(\text{PSL}(n,\bbR))$}}.
    
    In \cite{fock2021higher}, the proposed contraction of $\mathcal{T}^n(S)$ is defined as a family of maps induced on the quotient $\mathcal{T}^n(S) = \mathbb{M}^n(S)/\mathcal{G}^n(S)$ from a contraction of $\mathbb{M}^n(S)$. However it can be shown, at least for $n \geq 4$, that this contraction of $\mathbb{M}^n(S)$ does not map orbits of $\mathcal{G}^n(S)$ into orbits of $\mathcal{G}^n(S)$. The proof of this is carried out by using results from Section \ref{section-group-structure} and a Stokes' Theorem argument to show that if the contraction of $\mathbb{M}^n(S)$ in \cite{fock2021higher} were to map $\mathcal{G}^n(S)$ orbits to $\mathcal{G}^n(S)$ orbits for some $n \geq 4$, then there would be a Riemann surface $\Sigma$, type $(-3,1)$ tensors on $\Sigma$ of the form $\delbar w_3$ with $w_3$ a $(-3,0)$-tensor, and holomorphic quartic differentials $\varphi$ so that $\int_{\Sigma} (\delbar w_3)\varphi \neq 0$. So the argument in \cite{fock2021higher} demonstrates path-connectivity of $\mathcal{T}^n(S)$, rather than full contractibility.
    
    The aspect of \cite{fock2021higher} that is of consequence to our methods here is that the discussion in \cite{fock2021higher} of the manifold structure and dimension of $\mathcal{T}^n(S)$ is formal, based on infinitesimal analysis of the action of {\rm{$\text{Ham}^0_c(T^*S)$}}. In general, though, global behavior of group actions can cause quotients to become singular or lose dimension in a way not seen by infinitesimal analysis. In the present case, the action of $\mathcal{G}^n(S)$ on $\mathbb{M}^n(S)$ is neither proper nor free (see Section \ref{section-group-structure}).
    
    Much of this paper can be seen as working out a sufficiently clear picture of $\mathcal{G}^n(S)$ to fully address the dimension of $\mathcal{T}^n(S)$: a proof of our main results using a lower bound to the dimension of $\mathcal{T}^n(S)$, in place of results we prove in Section \ref{section-group-structure}, is sketched in Section \ref{quotient-section}.
    \end{remark}
    The underlying source of the vector bundle structure of $\mathcal{T}^n(S)$ is the appearance of linear behavior in $\mathcal{G}^n(S)$ analogous to features of some jet groups (see Example \ref{model-jet-group}).
    
    \begin{theorem}\label{thm-structure-summary}
        The degree-$n$ diffeomorphism group $\mathcal{G}^n(S)$ is an internal semi-direct product {\rm{$\text{Diff}_0(S) \ltimes \mathcal{N}(S)$}}, where {\rm{$\text{Diff}_0(S)$}} is included via taking lifts of diffeomorphisms to $T^*S$. The group $\mathcal{N}(S)$ is nilpotent and has a nested sequence of subgroups {\rm{$\mathcal{N}(S) = \mathcal{N}^2(S) \vartriangleright \mathcal{N}^3(S) ... \vartriangleright \mathcal{N}^n(S) = \{\text{Id}\}$}} so that $[\mathcal{N}(S), \mathcal{N}^k(S)] \subset \mathcal{N}^{k+1}(S)$ and $\mathcal{N}^k(S)/\mathcal{N}^{k+1}(S)$ is isomorphic to $\Gamma(\mathcal{S}^{k}(TS))$ with group operation addition. Here, $\mathcal{S}^j$ denotes $j$-fold symmetric product.
    \end{theorem}
    
    From the theorem of Earle and Eells \cite{earle1969fibre} that $\text{Diff}_0(S)$ is contractible for closed surfaces $S$ of genus at least $2$, we immediately find a generalization to the setting of higher degree diffeomorphism groups:
    
    \begin{corollary}\label{contractible}
    $\mathcal{G}^n(S)$ is contractible.
    \end{corollary}
    
    \begin{remarks}
    \begin{enumerate} 
    \item The algebraic structure of $\mathcal{N}(S)$ allows for an analysis of the quotient $\mathbb{M}^n(S)/\mathcal{N}(S)$ that requires less of the machinery of infinite-dimensional Lie groups than often occurs with similar quotients. For instance, we do not need to use inverse limits of approximations of $\mathcal{N}(S)$ by Banach-Lie groups to access implicit function theorems, in contrast to the analysis of an action of {\rm{$\text{Diff}_0(S)$}} in a similar setting in \cite{fischer1984purely}.

    \item In Section \ref{section-group-structure}, we endow $\mathcal{G}^n(S)$ with a regular Fr{\'e}chet-Lie group structure by parametrizing $\mathcal{N}(S)$ with spaces of sections, and show this coincides as a topological group with the description of $\mathcal{G}^n(S)$ as a quotient of {\rm{$\text{Ham}_c^0(T^*S)$}} (Propositions \ref{quotient-topology-is-product}, \ref{prop-frechet-lie}). A natural question is if the map induced by the quotient of {\rm{$\text{Ham}_c^0(T^*S)$}} yields isomorphisms of appropriate stronger structures as well.
    \end{enumerate} \end{remarks}
    
    \subsection{Outline of the Paper}
    
    Broadly speaking, our proof of Theorem \ref{headline-result} requires three main inputs. The first is a sufficiently concrete description of the action of symplectic diffeomorphisms of $T^*S$ on $\mathbb{M}^n(S)$ to act as foundation for later study of $\mathcal{G}^n(S)$. The second is a coordinate system on $\mathbb{M}^n(S)$ well-adapted to the action of maps on $T^*S$ induced by diffeomorphisms of $S$. The third and most complicated component, relying on the other two, is a detailed picture of the structure of the higher-degree diffeomorphism group $\mathcal{G}^n(S)$.
    
    In Section \ref{section-jets-and-symplectos}, we give a description of $\mathbb{M}^n(S)$ in terms of jets of functions $f : T^*S \to \mathbb{C}$. The benefit of this perspective is that it allows for our desired non-infinitesimal description of the action of $\text{Ham}(T^*S)$ on $\mathbb{M}^n(S)$, which is carried out in Section \ref{2-complex}.
    
    We then describe the mapping class group action on $\mathcal{T}^n(S)$ and the natural projection maps between Fock-Thomas spaces in Section \ref{section-mapping-class}. In Section \ref{section-first-var}, we recall the first variation formula of Fock and Thomas for the action of $\text{Ham}^0_c(T^*S)$ on $\mathbb{M}^n(S)$ and find a geometric interpretation of the first variation formula in terms of Maa\ss \, derivatives (see \cite{wolpert1986chern} for definitions and an exposition).
  
    In Section \ref{section-natural-coords}, we introduce coordinates on $\mathbb{M}^n(S)$ that do not depend on a choice of a reference complex structure $\Sigma$ on $S$. In these coordinates, maps of $T^*S$ induced by diffeomorphisms of $S$ act by pullback on appropriate objects. Our analysis of the finer structure of $\mathcal{G}^n(S)$ is then done in Section \ref{section-group-structure}.
    
   We then show that $\mathcal{T}^n(S)$ admits a vector bundle structure over $\mathcal{T}(S)$ in Section \ref{quotient-section}, through application of previous results and a Hodge decomposition theorem for $k$-Beltrami differentials. The link this establishes between $\text{Hit}(S, \text{PSL}(n,\bbR))$ and $\mathcal{T}^n(S)$ is described in Section \ref{section-formal-hitchin-link}. We give an explicit and geometric description of the diffeomorphism in the the $\mathcal{T}^3(S)$ case in Section \ref{section-3-complex}. \\
    
    \par \noindent \textbf{Acknowledgements.} It is the author's great pleasure to thank Mike Wolf for suggesting thinking about higher complex structures, and for his continued support and guidance. The author also thanks Leo Digiosia for helpful conversations about symplectic topology and Alexander Thomas for careful comments on a previous draft of this paper.
    
    \section{Higher Complex Structures via Jets}\label{section-jets-and-symplectos}

    An $n$-complex structure is a special section of a bundle over $S$ whose fiber at $x_0$ is the space of ideals of complex codimension $n$ in $\mathbb{C}[\partial_x, \partial_y]$, where $(x,y)$ is a coordinate chart about $x_0$. In this section, we give a description of $n$-complex structures in terms of jets of functions $f: T^*S \to \mathbb{C}$. The objects produced are the same as in \cite{fock2021higher}, but their concrete origin in our perspective is essential to our methods in later sections. One benefit is a non-infinitesimal description of the action of higher degree diffeomorphisms, which we carry out in Section \ref{2-complex}.

    Let $M$ be a closed smooth manifold, $T^*M$ the cotangent bundle of $M$, and $Z^*(M)$ the zero section of $T^*M$. The cotangent bundle $T^*(M)$ has a canonical symplectic form $\omega_{\text{can}}$. The canonical symplectic form is exact: if $\lambda_{\text{taut}}$ is the tautological one-form on $T^*(M)$, we have $\omega_{\text{can}} = - d \lambda_{\text{taut}}$. In a linear coordinate system $(x_i, p_i)$, the tautological one-form and canonical symplectic form are expressed $\lambda_{\text{taut}} = p_i dx_i$ and $\omega_{\text{can}} = dx_i \wedge dp_i$. The symplectic form $\omega_{\text{can}}$ induces a Poisson bracket on $C^{\infty}(T^*M, \bbR)$ by $\{f, g\} = \omega_{\text{can}}(X_f, X_g)$. Here, for any smooth function $H$, the vector field $X_H$ is defined by the requirement $\omega_{\text{can}}(X_H, \cdot) = dH$. In linear coordinates $(x_i, p_i)$ for $T^*M$, the Poisson bracket is given by $$\{f, g\} = \sum_{i=1}^{\dim M} \left( \frac{\partial f}{\partial x_i} \frac{\partial g}{\partial p_i} - \frac{\partial f}{\partial p_i} \frac{\partial g}{\partial x_i} \right).$$

    For a smooth function $f : T^*(M) \to \mathbb{R}$, we denote the $k$-jet of $f$ by $j_k(f)$ (see \cite{kolar2013natural} for definitions and basics on jets). We typically restrict to the case where $f(Z^*(M)) = \{0\}$. This condition is equivalent to $j_0(f)(x) = 0$ for all $x \in Z^*(M)$. We adopt the notation that for a function $f: T^*(M) \to \bbR$ vanishing on $Z^*M$, the restriction of $j_k(f)$ viewed as a polynomial function $T_\alpha T^*M \to \mathbb{R}$ for each $\alpha \in T^*S$ to the tangent spaces of the fibers of $T^*M$ along the zero section is $j_k^0(f)$. Formally, let $\pi: T^*M \to M$ be the projection and $\mathcal{V} = (T T^*M)|_{Z^*M}$. Denote by $\mathcal{K}$ the kernel of $(D\pi)|_\mathcal{V}$. Then $j_k^0(f)$ is defined for $x \in M$ to be $j_k^0(f)_x = j_k(f)_{(x,0)}|_\mathcal{K}$. Note that $j_k(f)|_{Z^*S}$ is completely determined by $j_k^0(f)$. For any integers $k$ and $l$ with $0 \leq l \leq k+1$, write as $$\mathcal{J}_k^{l, \bbR} = \{ j_k^0(f) \mid f : T^*M \to \mathbb{R} \text{ has }  j_m^0(f) \equiv 0 \text{ for } m < l \}$$ the space of restricted $k$-jets that vanish to order $l-1$ on $Z^*S$.
    
    The symplectic structure of $T^*(M)$ endows $\mathcal{J}_{k}^{1,\bbR}$ with the structure of a Poisson algebra with operations $$j_k^0(f) j_k^0(g) = j_k^0(fg), \qquad \qquad \{j_k^0(f), j_k^0(g) \} = j_k^0(\{f,g\}).$$ With this product, $\mathcal{J}_k^l$ has a decreasing filtration by degree of vanishing $\mathcal{J}_k^{l,\bbR} \supset \mathcal{J}_k^{l+1,\bbR} \supset \cdots \mathcal{J}_k^{k+1,\bbR} = \{0\}$ so that $$\mathcal{J}_k^{m,\bbR} \cdot \mathcal{J}_k^{n,\bbR} \subset \mathcal{J}_k^{m+n,\bbR},  \qquad  \{\mathcal{J}_k^{m,\bbR} , \mathcal{J}_k^{n,\bbR} \} \subset \mathcal{J}_k^{m+n-1,\bbR}.$$ In particular, when $m < n$ we have that $\mathcal{J}_k^{n,\bbR}$ is an ideal in $\mathcal{J}_k^{m,\bbR}$. Beneath the formalism, this corresponds to the well-known fact from calculus that the product of functions vanishing to degree $m$ and $n$ vanishes to degree $m+n$.
    
    It is often useful to record the data of a jet in terms of the Taylor expansion of a representative function. Fix an auxiliary metric $g$ on $M$, which induces a metric on $T^*M$. Let $f: T^*(M) \to \bbR$ be a smooth function vanishing on $Z^*(M)$. Then $f$ has a degree $k$ Taylor expansion on fibers of $T^*M$ along $Z^*M$ of the form $$f(x, p) = a_1(x,p) + \cdots + a_k(x, p) + R_k(x, p) \qquad (x \in M, p \in T_x^*M)$$ where $a_j(x, p)$ is a homogeneous degree $j$ polynomial in $p$ for any fixed $x$. The auxiliary metric is used here to make $a_1, ..., a_k$ well-defined by demanding that for any sequence $(x_n, p_n)$ with $p_n \neq 0$ converging to $(x, 0)$, the remainder $R_k(x_n,p_n)$ is bounded by $\lim\limits_{n \to \infty} R_k(x_n,p_n) ||p_n||^{-k-1} \leq C_k(f, g)$. Compactness of $M$ ensures the sections $a_k \in \Gamma(\mathcal{S}^k(TM))$ and the existence of $C_k$ are independent of the auxiliary metric.
    
    By identifying the tangent spaces $T_{(x,0)} T_x^*M$ to $T_x^*M$ with $T_x^*M$ for $x \in M$, the data of $j_k^0(f)$ is equivalent to the functions $ a_1, ..., a_k$. This identifies elements of $\mathcal{J}_{k}^{1,\bbR}$ with sums of sections of the symmetric product bundles $\mathcal{S}^j(TM)$ for $1 \leq j \leq k$.
    
    One special subspace of $\mathcal{J}_n^{1,\bbR}$ is the collection $\mathcal{X}_n^1$ of linear $n$-jets, which consists of the purely linear $j_n^0(f)$. Elements of $\mathcal{X}_n^1$ are those jets $j_n^0(f)$ whose coordinate expansion in terms of homogeneous polynomials $a_0, a_1, .., a_k$ have $a_0 \equiv 0$ and $a_2, ..., a_n \equiv 0$. The linear $n$-jet space $\mathcal{X}_n^1$ is closed under Poisson bracket but not multiplication, and corresponds canonically to the space of vector fields on $M$.
        
    For complex-valued functions $f: T^*M \to \mathbb{C}$ on the real cotangent bundle, we obtain analogous definitions by breaking $f$ up into real and imaginary parts and extending the product and Poisson bracket bilinearly. Denote by $\mathcal{J}^{l}_k$ the analogue of $\mathcal{J}^{l,\bbR}_{k}$ for complex-valued $f : T^*M \to \mathbb{C}$, and $\mathcal{S}^{j,\bbC}(TM)= \mathcal{S}^j(TM) \oplus i \mathcal{S}^j(TM)$.
    
    Now restrict to the case of a closed, oriented surface $S$. For any point $x \in S$, the fiber of $\bigoplus_{j=0}^{k} \mathcal{S}^{j,\bbC}(TS)$ over $x$ is identified with the algebra of of degree $k$ complex-valued polynomials on the real vector space $T_x^*S$ modulo degree $k+1$ polynomials. Elements of $\left(\bigoplus_{j=0}^{k} \mathcal{S}^{j,\bbC}(TS)\right)_x$ extend uniquely to complex polynomials on the complexification $T^{*\bbC}_xS$ naturally with respect to multiplication, identifying the fiber of $\bigoplus_{j=0}^{k} \mathcal{S}^{j,\bbC}(TS)$ over $x$ with $\mathcal{S}^j(T_x^{\bbC}S)$.
    
    As fibers of $\bigoplus_{j=0}^{k} \mathcal{S}^{j,\bbC}(TS)$ are algebras, we can consider their multiplicative ideals. In \cite{fock2021higher}, Fock and Thomas show that the collection $$\mathcal{I}^{k}_x  = \left\{\text{Ideals } I \subset \bigoplus_{j=0}^{k-1} \mathcal{S}^j(T_x^{\bbC}S) \, \bigg| \, I \text{ supported at }0, \text{codim}(I) = k,  I + \overline{I} \text {  maximal}\right\}$$ of ideals is a smooth manifold by giving an explicit parametrization in coordinates. In our notation, the maximal ideal supported at $0$ is $\bigoplus_{j=1}^{k-1} \mathcal{S}^j(T^\bbC S).$ Taking these parametrizations in coordinate patches yields trivializations of the bundle $\mathcal{I}^{k}(S)$ with fiber over $x$ given by $\mathcal{I}^{k}_x$.
       
   \begin{definition}[Fock-Thomas \cite{fock2021higher}] An {\rm{$n$-complex structure}} is a section $I \in \Gamma(\mathcal{I}^{n}(S))$. We denote by $\mathbb{M}^n(S)$ the space of all $n$-complex structures modulo conjugation.
   \end{definition}
   
      Now, let $\Sigma$ be a Riemann surface structure compatible with the orientation of $S$. In a complex coordinate $z = x + iy$, a basis determined by $z$ for pointwise degree-1 polynomials on $T^{*\bbC}M$ is given by $p = \frac{1}{2}(p_x - i p_y)$ and $\overline{p} = \frac{1}{2}(p_x + i p_y)$, where $p_x : T^{*\bbC}S \to \mathbb{C}$ is defined by $p_x(dx) = 1$ and $p_x(dy) = 0$, and $p_y$ is defined analogously. The basis vectors $p$ and $\overline{p}$ correspond under the identification of $\mathcal{X}_n^1$ and vector fields to $\del_z$ and $\del_{\bar{z}}$, respectively.
   
     Fock and Thomas show that picking a Riemann surface structure $\Sigma$ on $S$ determines global coordinates on $\mathbb{M}^n(S)$. We call these coordinates on $\mathbb{M}^n(S)$ \textit{centered coordinates based at $\Sigma$}.
   
    \begin{proposition}{(Fock-Thomas \cite{fock2021higher})}\label{fock-thomas-global-coordinates}
    Let $\Sigma$ be a Riemann surface compatible with the orientation of $S$. If $I$ is an $n$-complex structure, then either $I$ or $\bar{I}$ admits in local coordinate patches a unique expression of the form \begin{align}\label{higher-complex-structure-normalized-form} \langle -\overline{p} + \mu_2(z) p + \mu_3(z) p^2 + \cdots + \mu_n(z) p^{n-1} \rangle, \qquad |\mu_2(z)| < 1 \end{align} with respect to a complex coordinate $z$.
    
    Globally, $\mu_k$ is a smooth $(1-k,1)$-tensor for $k = 2, ..., n$. A map assigning to every $x \in S$ an ideal in $\bigoplus_{j=0}^{n-1} \mathcal{S}^j(T_x^\bbC S)$ is an $n$-complex structure if such local expressions exist, and for any smooth $(-1, 1), ..., (1-n, 1)$-tensors  $\mu_2, ..., \mu_n$ with $|\mu_2(z)| < 1$ for all $z$, then there is an $n$-complex structure determined by the local expressions $\langle -\overline{p} + \mu_2(z) p + \mu_3(z) p^2 + \cdots + \mu_n(z) p^{n-1} \rangle$.
    \end{proposition}
    
    We note that our realization of higher complex structures as ideals inside jet spaces of fixed degree removes the presence of a term in Fock-Thomas' description that has the effect of working modulo terms of degree at least $n$ in the above.
    
    We say that an $n$-complex structure $I$ is \textit{compatible with the orientation of $S$} if $I$ has the form of \eqref{higher-complex-structure-normalized-form} with respect to a Riemann surface $\Sigma$ compatible with $S$. For every $I \in \mathbb{M}^n(S)$, exactly one of $I$ and $\overline{I}$ is compatible with the orientation of $S$. In the following, we identify $\mathbb{M}^n(S)$ with the space of $n$-complex structures compatible with the orientation of $S$. We adopt the notation that the $k$-Beltrami differential associated to a $n$-complex structure $I \in \mathbb{M}^n(S)$ is $\mu_k(I)$.
   
\section{First Features of Higher Diffeomorphisms}\label{2-complex}

We now explain the action of symplectic diffeomorphisms of $T^*S$ that setwise fix the zero section on $\mathbb{M}^n(S)$, and establish some first computations and lemmata about this action. We begin with an action on sections of arbitrary ideals, and show that it restricts to an action on $\mathbb{M}^n(S)$. We explain how linear lifts of diffeomorphisms of $S$ act on $\mathbb{M}^n(S)$, and then as an example show how in our framework $\mathcal{T}^2(S)$ can be canonically identified with $\mathcal{T}(S)$. 

Let $\text{Symp}(T^*S)$ denote the group of symplectic diffeomorphisms of $T^*(S)$ with respect to the canonical symplectic form. Consider the subgroup $\text{Symp}^0(T^*S)$ of $\text{Symp}(T^*S)$ consisting of $\varphi$ so that $\varphi|_{Z^*(S)}$ is an orientation-preserving diffeomorphism of $Z^*(S)$, and denote by $\text{Symp}_c^0(T^*M)$ the subgroup of $\text{Symp}^0(T^*S)$ consisting of compactly supported symplectomorphisms. The group $\text{Symp}^0(T^*S)$ acts on $\mathcal{J}_k^{1, \bbR}$ by $j_{k}^0(f) \varphi = j_k^0(f \circ \varphi)$. This is well-defined as $j_k^0(f)$ completely determines $j_k(f)|_{Z^*S}$, and is an action by multiplicative algebra isomorphisms on fibers that respect the filtration of $\mathcal{J}_k^{1,\bbR}$ by degree.

    The symplectomorphisms $\varphi \in \text{Symp}^0(T^*S)$ arising from compactly supported Hamiltonian flows $\varphi_t$ so that $\varphi_t|_{Z^*S}$ is a diffeomorphism of $Z^*S$ for all times $t$ form a subgroup of $\text{Symp}^0_c(T^*S)$, which we denote by $\text{Ham}^0_c(T^*S)$. The subgroup of $\text{Ham}_c^0(T^*S)$ consisting of elements of $\text{Ham}_c^0(T^*S)$ that pointwise fix $Z^*S$ is denoted $\text{Ham}^P_c(T^*S)$. The condition in these definitions that $\varphi_t$ restrict to diffeomorphisms of $Z^*S$ for all times is essential for the following. See Appendix \ref{appendix-community-service} for discussion of how this restriction avoids some potential pathologies.
    
    A smooth family of compactly supported smooth functions $H_t : T^*S \to \mathbb{R}$ that vanish on $Z^*S$ for all $t$ defines a Hamiltonian flow $\varphi_t$. For such an $H_t$, $$ dH_t = \frac{\partial H_t}{\partial x_i} dx_i + \frac{\partial H_t}{\partial p_i} dp_i = \frac{\partial H}{\partial p_i} dp_i, \qquad X_{H_t} = -\frac{\partial H}{\partial p_i} \partial_{x_i}.$$ So $X_{H_t}$ is tangent to the zero section for all $t$, hence $\varphi_t \in \text{Symp}_c^0(T^*S)$ for all $t$. Conversely, if $H_t$ is not constant on $Z^*S$, then $X_{H_t}$ is not tangent to $Z^*S$. For any function $f \in C^\infty(S, \bbR),$ we have that $\frac{d}{dt} f\circ \varphi_t = \{f\circ \varphi_t, H_t\}$.

Any section $\sigma \in \bigoplus_{j=0}^k(\mathcal{S}^{j,\bbC}(TS))$ is identified with a $k$-jet $j^0_k(f)$. For $\varphi \in \text{Symp}^0(T^*S),$ define as with real sections $j^0_k(f) \varphi =  j^0_k(f \circ \varphi)$. A symplectomorphism $\varphi \in \text{Symp}^0(T^*S)$ acts on $\sigma \in \bigoplus_{j=0}^k\Gamma( \mathcal{S}^j(T^{\bbC}S))$ by defining $\sigma \varphi$ to be the unique complex-polynomial extension to $T^{*\bbC}S$ of $j_k^0(f)\varphi$, where $j_k^0(f)$ is the jet associated to $(\sigma|_{T^*S})$. Here, the real cotangent bundle $T^*S$ is identified with the fixed-point set of conjugation in $T^{*\bbC}S$. This action is once again by multiplicative algebra isomorphisms on fibers that respect the filtration by degree.

\begin{remark} It is crucial for well-definition here that we restrict sections to the real cotangent bundle, act, then re-extend to the complexified cotangent bundle: there is no natural action of {\rm{$\text{Symp}^0(T^*S)$}} on $T^{*\bbC}S$. \end{remark}

Denote by $\mathbb{N}^n(S)$ the collection of functions $f$ from $S$ to the collection of pairs $(x,I)$ with $x \in S$ and $I$ an ideal in $\bigoplus_{j=0}^{n-1} \mathcal{S}^j(T^\bbC_xS)$ so that for all $x \in S$ the $S$-coordinate of $f(x)$ is $x$. For our purposes it is sufficient to view $\mathbb{N}^n(S)$ as a set containing $\mathbb{M}^n(S)$. Since $\varphi$ acts by algebra isomorphisms, it induces an action on $\mathbb{N}^n(S)$. We shall show $\varphi$ maps $\mathbb{M}^{n}(S)$ to itself.

A first observation about the action of $\text{Symp}^0(T^*S)$ is that for any ideal $I \in \mathbb{N}^n(S)$, if $\varphi = \varphi'$ on a neighborhood of $Z^*S$, then $I \varphi = I \varphi'$. This makes restricting attention to flows supported on a fixed compact set $K$ containing $Z^*S$ in its interior still yield the same possible actions of Hamiltonian diffeomorphisms generated by these flows on $\mathbb{N}^n(S)$ as considering the full group $\text{Ham}_c^0(T^*S)$. To see this, if $\varphi$ is the time-$1$ flow of $H_t$ and $\eta: T^*S \to \bbR$ is a compactly supported function that is $1$ on a neighborhood of $Z^*S$, then the time-1 flow $\widetilde{\varphi}$ of $\eta H_t$ agrees with $\varphi$ on a neighborhood of $Z^*S$ as the flow of $H_t$ fixes $Z^*S$ setwise for all times $t$.

Diffeomorphisms $f: S \to S$ lift to exact linear symplectomorphisms of the cotangent bundle $f^\#: T^*S \to T^*S$ by $f^\#\alpha = \alpha \circ (Df)^{-1}$. Lifts of diffeomorphisms provide the simplest class of actions of elements of $\text{Symp}^0(T^*S)$ to understand, and are a first stepping-stone towards obtaining a picture of the action. 
 
 Now, let $\varphi \in \text{Symp}^0(T^*S)$ and write $f_\varphi = \varphi|_{Z^*S}$. In linear coordinates on $T^*S$, the condition that $\varphi \in \text{Symp}^0(T^*S)$ implies that on the zero section, $D\varphi$ has the form \begin{align*}
    \begin{bmatrix} Df_{\varphi} & A \\ 0 & (Df_{\varphi})^{-1} \end{bmatrix},\qquad  Df_\varphi A^T = A (Df_\varphi)^T.
\end{align*} Here, $\varphi$ agrees with the lift of $f_\varphi$ to first order along $Z^*S$ if and only if $A = 0$ everywhere. 

\begin{lemma}\label{lemma-2-action}
    Fix a Riemann surface structure $\Sigma$ on $S$ compatible with the orientation of $S$. Let $g : T^*S \to \mathbb{C}$ vanish on $Z^*S$ and have local expression $j_1^0(g) = -\overline{p} + \mu_2(z) p$ with $|\mu_2| < 1$, and {\rm{$\varphi \in \text{Symp}^0(T^*S)$}}. Then in $\mathbb{N}^{2}(S) ,$ \begin{align}\label{diffeo-pullback-formula} \langle j_1^0(g \circ \varphi) \rangle =  \left \langle -\overline{p} + \left( \frac{\delbar f_\varphi + (\mu_2 \circ f_\varphi) \delbar \bar{f}_\varphi }{\del f_\varphi + (\mu_2 \circ f_\varphi) \del \bar{f}_\varphi } \right) p \right \rangle. \end{align}
\end{lemma}

\begin{proof}
    First consider the case where $f_\varphi = \text{Id}$. Then in local linear coordinates, $D\varphi$ has the form $\begin{bmatrix} \text{Id} & A \\ 0 & \text{Id} \end{bmatrix}.$ As $g$ vanishes on $Z^*S$, we see that $j_1^0(g) = j_1^0(g \circ \varphi)$. So for any $\varphi\in \text{Symp}^0(T^*S)$, we have $j_1^0(g \circ \varphi) = j_1^0(g \circ f_\varphi^\#)$. Thus it suffices to compute the action of a lift of an orientation-preserving diffeomorphism of $S$.
    
    Now let $f: S \to S$ be an orientation-preserving diffeomorphism. A standard computation shows that for any $\zeta$, with respect to the bases of appropriate complexified tangent spaces given by $\del$ and $\delbar$, \begin{align*} Df_\zeta &= \begin{bmatrix}
\del f & \delbar f  \\ \del \bar{f} & \delbar \bar{f}
\end{bmatrix}, \qquad (Df_{\zeta})^{-1} = \frac{1}{|\del f|^2 - |\delbar f|^2} \begin{bmatrix} \delbar \bar{f} & - \delbar f \\ -\del f & \del f
\end{bmatrix}.\end{align*}

We then obtain
\begin{align*}
j_1^0(g\circ f^\#)
	&=  -\left( \frac{(-\delbar f) p + (\del f) \overline{p}}{|\partial f|^2 - |\delbar f|^2} \right) + (\mu_2 \circ f)\left( \frac{(\delbar \bar{f})p - (\del \bar{f})\overline{p} }{|\del f|^2 - |\delbar f|^2} \right),
\end{align*} so that \begin{align*}
 \langle j_1^0(g \circ f^\#) \rangle & = \langle (\delbar f) p - (\del f) \overline{p} + (\mu_2 \circ f) (\delbar \bar{f})p - (\mu_2 \circ f) (\del \bar{f} )\overline{p} \rangle \\
	&= \left\langle -\overline{p} + \left( \frac{\delbar f + (\mu_2 \circ f) \delbar \bar{f} }{\del f + (\mu_2 \circ f) \del \bar{f} } \right) p \right \rangle .
\end{align*} \end{proof}

A subtlety of the action of higher diffeomorphisms appears in the previous proof: $\varphi$ fixes every $2$-complex structure if and only if $f_{\varphi} = \text{Id}$, which does not fully specify $D\varphi$ along $Z^*(S)$. So two symplectic diffeomorphisms $\varphi, \psi \in \text{Symp}^0(T^*S)$ agreeing in their action on $\mathbb{N}^2(S)$ does not imply equality of the $1$-jets of $\varphi$ and $\psi$ along $Z^*(S)$. This persists in $\mathbb{N}^n(S)$ for $n > 2$, and shapes our approach to finding effective coordinates to analyze the structure of higher degree diffeomorphism groups in Section \ref{section-group-structure}.

Proposition \ref{lemma-2-action} can be used to show the action of $\text{Symp}^0(T^*S)$ on $\mathbb{N}^n(S)$ restricts to an action on the space $\mathbb{M}^n(S)$ of $n$-complex structures.

\begin{proposition}\label{prop-actual-action}
    If $I \in \mathbb{M}^n(S)$ and {\rm{$\varphi \in \text{Symp}^0(T^*S),$}} then $I \varphi \in \mathbb{M}^n(S)$.
\end{proposition}

\begin{proof} Let $j_{n-1}^0(g) = -\overline{p} + \mu_2 p + ... + \mu_n p^{n-1}$ generate $I$ locally and $\varphi \in \text{Symp}^0(T^*S)$. We show that $I \varphi \in \mathbb{M}^n(S)$ by demonstrating that it admits a local expression of the form of formula \eqref{higher-complex-structure-normalized-form}.

Since the action of $\varphi$ on $j_1^0(g)$ depends only on $D \varphi$, following Lemma \ref{lemma-2-action} we see that \begin{align}\label{action-expression} \langle j_{n-1}^0(g \circ \varphi) \rangle = \left\langle -\overline{p} + \left( \frac{\delbar f_\varphi + (\mu_2 \circ f_\varphi) \delbar \bar{f}_\varphi }{\del f_\varphi + (\mu_2 \circ f_\varphi) \del \bar{f}_\varphi } \right) p + P(z, p, \overline{p}) \right \rangle . \end{align} where $P(z, p, \overline{p})$ is a polynomial in $p$ and $\overline{p}$ with coefficients functions of $z$ and no linear or constant terms in $p$ or $\overline{p}$. Note that as $|\mu_2| < 1$ and $f_\varphi$ is orentation-preserving, $$\left| \frac{\delbar f_\varphi + (\mu_2 \circ f_\varphi) \delbar \bar{f}_\varphi }{\del f_\varphi + (\mu_2 \circ f_\varphi) \del \bar{f}_\varphi }  \right| < 1. $$ So the condition $|\mu_2| < 1$ is preserved by the action of $\varphi$. Repeated application of the relation specified by formula \eqref{action-expression} yields a representative of $\langle j_{n-1}^0(g \circ \varphi) \rangle$ of the form of formula \eqref{higher-complex-structure-normalized-form}.
\end{proof}

Proposition \ref{prop-actual-action} ensures that $\text{Symp}^0(T^*S)$ acts on $\mathbb{M}^n(S)$, which justifies the following definitions.

    \begin{definition} {\rm{${\mathcal{T}}^n(S) = \mathbb{M}^n(S)/\text{Ham}^0_c(T^*S)$}} is the {\rm{degree-$n$ Fock-Thomas space}} of $S$.
    \end{definition}

    The space $\mathcal{T}^n(S)$ can be treated as just a topological space for now. A formal argument for $\mathcal{T}^n(S)$ being a manifold appears in \cite{fock2021higher}. Our later analysis independently yields a manifold structure on $\mathcal{T}^n(S)$.

\begin{definition}
The {\rm {degree $n$-diffeomorphism group}} $\mathcal{G}^n(S)$ is the quotient of {\rm{$\text{Ham}^{0}_{c}(T^*S)$}} by the kernel of its action on $\mathbb{M}^n(S)$.

The {\rm {$2$-stationary diffeomorphism group }}$\mathcal{N}(S)$ consists of all $h \in \mathcal{G}^n(S)$ that fix all $2$-complex structures. For $2 \leq k \leq n$, we define the {\rm{$k$-stationary diffeomorphism group}} $\mathcal{N}^k(S)$ to be the subgroup of $\mathcal{N}(S)$ that fixes all $k$-complex structures. An element of $\mathcal{N}^k(S)$ or in the pre-image of $\mathcal{N}^k(S)$ in {\rm{$\text{Symp}^0(T^*S)$}} is said to be {\rm{$k$-stationary}}, and a Hamiltonian $H_t$ is said to be {\rm{$k$-stationary}} if $H_t$ generates a flow by $k$-stationary higher diffeomorphisms.
\end{definition}
    
\begin{remark}
    Matters of topology require care for groups, such as $\mathcal{G}^n(S)$, arising from groups of diffeomorphisms with compact supports. One source of complications is that the compactly supported diffeomorphism groups {\rm{$\text{Diff}_c(T^*S)$}} of connected, noncompact, $\sigma$-compact manifolds are never topological groups with respect to the standard topology---their group operations are not continuous \cite{tatsuuma1998group}.
    
    A consequence of this is that even though the groups $\mathcal{G}^n(S)$ and $\mathcal{N}(S)$ inherit topologies from their definition, that they are even topological groups with these topologies is not immediately clear.
    
    Our structural results for $\mathcal{G}^n(S)$ in Section \ref{section-group-structure} will allow us to resolve some foundational matters about groups of higher diffeomorphisms, which we carry out in Section \ref{frechet-lie-subsection} and Appendix \ref{appendix-community-service}.
\end{remark}
    
    In the sequel we will make use of elements of $\text{Ham}_c^0(T^*S)$ agreeing with lifts of arbitrary diffeomorphisms $f \in \text{Diff}_0(T^*S)$ in a neighborhood of $Z^*S$. As lifts of nontrivial diffeomorphisms $f: S \to S$ are not compactly supported, we must verify that these exist.

\begin{lemma}\label{lifts-get-realized} For a diffeomorphism {\rm{$f \in \text{Diff}_0(T^*S)$}}, the lift $f^\#$ of $f$ agrees with an element of {\rm{$\text{Ham}_c^0(T^*S)$}} in its action on $\mathbb{M}^n(S)$. For an isotopy {\rm{$f_t \in \text{Diff}_0(S)$}}, there is an isotopy {\rm{$\varphi_t \in \text{Ham}_c^0(T^*S)$}} so $\varphi_t =f_t^\#$ on a neighborhood of $Z^*S$ for $0\leq t \leq 1$.
\end{lemma}

\begin{proof} 
For an isotopy $f_t \in \text{Diff}_0(S)$, every $f_t^\#$ is exact: $(f_t^\#)^* \lambda_{\text{taut}} = \lambda_{\text{taut}}$. It follows from standard symplectic topology (e.g. \cite{mcduff2017introduction} 9.19-9.20) that $f_t^\#$ is Hamiltonian on $T^*S$. Let $H_t$ be a generating function. Multiplying $H_t$ by a cutoff function gives a Hamiltonian that generates the desired isotopy in $\text{Ham}_c^0(T^*S)$. 
\end{proof}

\subsection{An Example: 2-Complex Structures and Teichm\"uller Space}\label{subsection-2-complex}
We now describe the identification between $\mathcal{T}^2(S)$ and $\mathcal{T}(S)$ in our framework, and prove a lemma used in our later analysis of the finer structure of $\mathcal{G}^n(S)$. The identification has been established in the machinery of \cite{fock2021higher} there; the utility of doing this for us is an explicit description within our framework.

We begin by recalling the formula for the action of diffeomorphisms on $\mathbb{M}(S)$ in terms of smooth Beltrami differentials on a fixed Riemann surface $\Sigma$.

Given a Beltrami differential $\mu$, an atlas of isothermal coordinates $\varphi^i$ (so $\delbar \varphi^i = \mu \del \varphi^i$) can be found, and these define a complex structure. Diffeomorphisms act on complex structures by pullback, and an explicit formula for how diffeomorphisms act on Beltrami differentials can be found as follows. Given a complex structure $\Sigma$ defined by an atlas of charts $\{\varphi^i\}$ and an orientation-preserving diffeomorphism $f \in \text{Diff}^+(S)$, another complex structure $f^* \Sigma$ is determined by the atlas $\{\varphi^i \circ f\}$. If $\varphi^i$ are isothermal coordinates for $\mu$, then
\begin{align*}
\delbar(\varphi^i \circ f) &= (\delbar f) [(\del \varphi^i) \circ f] + (\delbar \bar{f})[(\delbar \varphi^i ) \circ f] = [(\del \varphi^i) \circ f] ( \delbar f + (\mu\circ f) \delbar \bar{f} ), \\
\del (\varphi^i \circ f) &= (\del f) [(\del \varphi^i) \circ f] + (\del \bar{f})[(\delbar \varphi^i)\circ f] = [(\del \varphi^i)\circ f] (\del f + (\mu\circ f) \del \bar{f}).
\end{align*} This shows that $f^* \Sigma$ is the complex structure determined by the Beltrami differential \begin{align}\label{classical-pullback} f^* \mu = \frac{\delbar f + (\mu\circ f) \delbar \bar{f}}{\del f + (\mu \circ f) \del \bar{f} }.\end{align}

The observation that the same expression appears in formulas \eqref{classical-pullback} and \eqref{diffeo-pullback-formula} yields the following useful description of $\mathcal{G}^2(S).$

\begin{lemma}\label{2-diffeos-just-diffeos}
    The map {\rm{$\text{Ham}_c^0(T^*S) \to \text{Diff}_0(S)$}} given by $ \varphi \mapsto f_\varphi$ induces an isomorphism {\rm{$\mathcal{G}^2(S) \to \text{Diff}_0(S) $}}, equivariant with respect to the actions of $\mathcal{G}^2(S)$ on $\mathbb{M}^2(S)$ and {\rm{$\text{Diff}_0(S)$}} on $\mathbb{M}(S)$. 
\end{lemma}

As the map $\mu_2 \mapsto \langle - \overline{p} + \mu_2 p \rangle$ from $\mathbb{M}(S) \to \mathbb{M}^2(S)$ is equivariant with respect to the actions of $\text{Diff}_0(S)$ and $\mathcal{G}^2(S)$, an identification of $\mathcal{T}^2(S)$ and $\mathcal{T}(S)$ is obtained.

    \section{Projections and the Higher Degree Mapping Class Group}\label{section-mapping-class}
    Two further basic structural features of $\mathcal{T}^n(S)$ play roles in the following: there are natural projections $\pi_k: \mathcal{T}^n(S) \to \mathcal{T}^k(S)$ and a mapping class group action on $\mathcal{T}^n(S)$. The existence of projections was proved in \cite{fock2021higher} and the existence of a mapping class group action is remarked upon there. We explain a way to see the mapping class group action in terms of permuting markings of $n$-complex structures, which we use in our proofs of $\text{Mod}(S)$-equivariance of mappings later.
    
    Following \cite{fock2021higher}, there are projections $p_{k,j} : \mathcal{J}_k^1 \to \mathcal{J}_j^1 $ for $0 < j \leq k$ given by identifying $\mathcal{J}^1_k / \mathcal{J}^{j+1}_k$ with $\mathcal{J}^1_j$. In terms of homogeneous functions, $p_{k,j}: (a_1, ...,a_j, ... a_k) \mapsto (a_1, ..., a_j)$. As $\text{Symp}^0(T^*S)$ respects the filtration of $\mathcal{J}_k^1$ by degree, the projections $p_{k,j}$ are equivariant with respect to the action of higher degree diffeomorphisms. Furthermore, the induced maps $\mathbb{N}^k(S) \to \mathbb{N}^j(S) $ restrict to maps $\mathbb{M}^k(S) \to \mathbb{M}^j(S)$. So for a fixed $n$, the maps $p_{n,j}$ induce projections $\pi_{j}: {\mathcal{T}}^n(S) \to {\mathcal{T}}^j(S)$.

For our desired description of the $\text{Mod}(S)$ action on $\mathcal{T}^n(S)$, we shall need an analogue of the description of $\mathcal{T}(S)$ as the space of marked complex structures considered up to orientation-preserving diffeomorphism, which we now describe.

To begin, note that by Lemma \ref{2-diffeos-just-diffeos} for any $I \in \mathbb{M}^n(S)$ the projection $p_{n,2}(I) \in \mathbb{M}^2(S)$ is identified with a complex structure $\Sigma(I)$ on $S$. This identification is equivariant with respect to the actions of $\mathcal{G}^2(S)$ on $\mathbb{M}^2(S)$ and $\text{Diff}_0(S)$ on $\mathbb{M}(S)$. We remark that another description of $\Sigma(I)$ appears in Section \ref{section-natural-coords}.

The markings in our setting now carry over from classical Teichm\"uller theory essentially without modification. Specifically, we define a \textit{marked $n$-complex structure} to be a pair $(I, \phi)$ where $\phi: S \to \Sigma(I)$ is an orientation-preserving diffeomorphism. Markings on the same $n$-complex structure are subject to the equivalence relation defined by identifying $(I, \phi_1)$ and $(I, \phi_2)$ if $\phi_1$ and $\phi_2$ are isotopic, which is to say that $\phi_2^{-1} \circ \phi_1$ is isotopic to $\text{Id}$.

Denote by $\dot{\mathbb{M}}^n(S)$ the collection of marked $n$-complex structures on $S$. The group $\text{Symp}^0(T^*S)$ of symplectic diffeomorphisms acts on $\dot{\mathbb{M}}^n(S)$ by $(I, \phi) \varphi = (I\varphi, \phi \circ f_\varphi)$, where as before $f_\varphi = \varphi|_{Z^*S}$. Note that $\varphi \in \text{Ham}_c^0(T^*S)$ acts trivially on $\dot{\mathbb{M}}^n(S)$ if and only if it acts trivially on $\mathbb{M}^n(S)$.

The other component of our description of $\mathcal{T}^n(S)$ as a quotient of $\dot{\mathbb{M}}^n(S)$, and the component that requires care in its definition, is an appropriate analogue to the orientation-preserving diffeomorphism group $\text{Diff}^+(S)$ in Teichm\"uller theory. Such a group should enlarge the degree-$n$ diffeomorphism group $\mathcal{G}^n(S)$, which plays a role analogous to $\text{Diff}_0(S)$ in classical Teichm\"uller theory. A note here is that we must remove the restriction to compactly supported symplectic diffeomorphisms in some form, in order to include actions on $\dot{\mathbb{M}}^n(S)$ of diffeomorphisms of the form $f^\#$ for $f\in \text{Diff}^+(S)$ not isotopic to the identity.

The following are the groups of diffeomorphisms we work with, and have the benefit of including all lifts $f^\#$ of diffeomorphisms $f \in \text{Diff}^+(S)$, while still being sufficiently restricted to behave similarly to the group $\text{Symp}_c(T^*S)$ of compactly supported symplectic diffeomorphisms.

\begin{definition}\label{awful-symplectos-defs}
    A symplectic diffeomorphism {\rm{$\varphi \in \text{Symp}^0(T^*S)$}} so that $\varphi|_{Z^*S}$ is an orientation-preserving diffeomorphism of $Z^*S$ is {\rm{tame}} if there is a compact set $K \subset T^*S$ and {\rm{$f \in \text{Diff}^+(S)$}} so that $\varphi(\alpha) = f^\#(\alpha)$ for all $\alpha \notin K$. The group of tame symplectomorphisms of $T^*S$ is written as {\rm{$\text{Symp}_T^0(T^*S)$}}.
    
    A tame symplectic diffeomorphism {\rm{$\varphi \in \text{Symp}^0_T(T^*S)$}} is said to be {\rm{standard}} if there is a diffeomorphism {\rm{$f \in \text{Diff}^+(S)$}} so that {\rm{$\varphi \circ f^\# \in \text{Ham}_c^0(T^*S)$}}, and {\rm{$\varphi \in \text{Symp}_T^0(T^*S)$}} is otherwise said to be {\rm{exotic}}. The group of standard tame symplectic diffeomorphisms is denoted by {\rm{$\text{Symp}_T^{0,S}(T^*S)$}}.
    
    Denote by {\rm{$\text{Ham}_T^0(T^*S)$}} the group of tame symplectic diffeomorphisms $\varphi$ so that there exists {\rm{$f \in \text{Diff}_0(T^*S)$}} so that {\rm{$\varphi \circ f^\# \in \text{Ham}_c^0(T^*S)$}}.
    
    The quotient of {\rm{$\text{Symp}_T^0(T^*S)$}} by the kernel of its action on the collection of marked $n$-complex structures $\dot{\mathbb{M}}^n(S)$ is the {\rm{full degree-$n$ diffeomorphism group}} $\mathcal{H}^n(S)$. The quotient $\mathcal{H}^{n,S}(S)$ of {\rm{$\text{Symp}_T^{0,S}(T^*S)$}} by the kernel of its action on $\dot{\mathbb{M}}^n(S)$ is called the {\rm{$n$-standard diffeomorphism group}}.
    \end{definition}

A few observations and remarks are in order. The point of these definitions is that $\mathcal{T}^n(S)$ is identified with $\dot{\mathbb{M}}^n(S)/\text{Symp}_T^{0,S}(T^*S)$, which is explained in the coming paragraphs. As a consequence, we obtain a definition of $\mathcal{T}^n(S)$ in terms of marked $n$-complex structures that allows for a concrete description of the mapping class group action. Any $\varphi \in \text{Symp}_T^0(T^*S)$ agrees with the lift of a unique diffeomorphism except on a compact set. We denote this diffeomorphism by $f^\varphi$.

We begin with a remark on the on algebraic and topological structure of the groups of Definition \ref{awful-symplectos-defs}. Let $\mathcal{D}^+(S)$ denote the subgroup of $\text{Symp}_T^0(T^*S)$ consisting of lifts $f^\#$ of diffeomorphisms $f \in \text{Diff}^+(S)$. Then $\text{Symp}_c^0(T^*S)$ is a normal subgroup of $\text{Symp}_T^0(T^*S)$, the intersection $\text{Symp}_c^0(T^*S) \cap \mathcal{D}^+(S) = \{\text{Id}\}$ is trivial, and $\text{Symp}_T^0(T^*S) = \text{Symp}_c^0(T^*S) \mathcal{D}^+(S)$. So $\text{Symp}_T^0(T^*S)$ is the internal algebraic semidirect product $\text{Symp}_c^0(T^*S) \rtimes \mathcal{D}^+(S)$. We give $\text{Symp}_T^0(T^*S)$ the topology of the product $\text{Symp}_c^0(T^*S) \times \text{Diff}^+(S)$ in the obvious fashion. Due to analogous reasoning, we have a semidirect product decomposition $\text{Symp}_T^{0,S}(T^*S) = \text{Ham}_c^0(T^*S) \rtimes \mathcal{D}^+(S)$. One final algebraic remark used in the sequel is that $\text{Ham}_T^0(T^*S)$ is normal in $\text{Symp}_T^{0,S}(T^*S)$.

Another useful observation is that the quotient of $\text{Ham}_T^0(T^*S)$ by the stabilizer of its action on $\dot{\mathbb{M}}^n(S)$ is identified with $\mathcal{G}^n(S)$. To see this, for any $\varphi \in \text{Ham}_T^0(T^*S)$, let $\psi \in \text{Ham}_c^0(T^*S)$ agree with $(f^\varphi)^\#$ on a neighborhood of $Z^*S$. Such a diffeomorphism $\psi$ exists by Lemma \ref{lifts-get-realized}. Then $\psi \circ ((f^\varphi)^\#)^{-1} \circ \varphi \in \text{Ham}_c^0(T^*S)$ agrees with $\varphi$ in a neighborhood of $Z^*S$, and so has the same action on $\mathbb{M}^n(S)$.

 One verifies that the maps \begin{align*} \mathbb{M}^n(S)/\text{Ham}_c^0(T^*S) &\to \dot{\mathbb{M}}^n(S)/\text{Symp}_{T}^{0,S}(T^*S) \\ [I]&\mapsto [(I, \text{Id})]  \end{align*} and \begin{align*} \dot{\mathbb{M}}^n(S)/\text{Symp}_{T}^{0,S}(T^*S) &\to \mathbb{M}^n(S)/\text{Ham}_c^0(T^*S) \\ [(I, \phi)] &\mapsto [I(\phi^{-1})^\#] \end{align*}are well-defined and inverse to each other. The verification that the second map is well-defined uses the normality of $\text{Ham}_T^0(T^*S)$ in $\text{Symp}_T^{0,S}(T^*S)$. This gives an identification between $\mathcal{T}^n(S)$ and $\dot{\mathbb{M}}^n(S)/\text{Symp}_T^{0,S}(T^*S)$.

We are now ready to define the mapping class group action on $\mathcal{T}^n(S)$ and to give an equivalent description in terms of marked $n$-complex structures.

 \begin{definition}
 The action of the mapping class group on $\mathcal{T}^n(S)$ is defined by $[I][f] = [I f^\#]$. In the description of $\mathcal{T}^n(S)$ as {\rm{$\dot{\mathbb{M}}^n(S)/\text{Symp}_{T}^{0,S}(T^*S)$}}, this action is given by $[(I, \phi)][f] = [(I, f^{-1} \circ \phi)]$.
 \end{definition}
 
 One notes here that our identification $\Phi: \mathcal{T}^n(S) \to \dot{\mathbb{M}}^n(S)/\text{Symp}_{T}^{0,S}(T^*S)$ is equivariant with respect to the mapping class group action: $$\Phi([I][f]) = \Phi([If^\#]) = [(If^\#, \text{Id})] = [(I, f^{-1})] = [(I, \text{Id})][f] = \Phi([I])[f].$$
 
Inspired by a talk of Thomas, we conclude this section with a brief discussion of a peculiarity of the Fock-Thomas spaces: there is a group containing $\text{Mod}(S)$ that acts naturally on $\mathcal{T}^n(S)$, but which is not clear is only $\text{Mod}(S)$.

\begin{definition} The {\rm{degree-$n$ mapping class group}} {\rm{$\text{HMod}^n(S)$}} is the quotient $\mathcal{H}^n(S)/\mathcal{G}^n(S)$.
\end{definition}

We prove in Appendix \ref{appendix-community-service} that $\mathcal{H}^n(S)$ is a topological group, that $\mathcal{G}^n(S)$ is the identity component of $\mathcal{H}^n(S)$, and that $\text{HMod}^n(S)$ is discrete. The proof is somewhat delicate: it uses substantial theorems about infinite-dimensional Lie groups and our structural results from Section \ref{section-group-structure}. The proof also makes crucial use of the facts that the compactly supported De Rahm cohomology group $H^1_c(T^*S)$ is $0$ and that $\omega_\text{can}$ is exact. 
 
 The degree-$n$ mapping class group $\text{HMod}^n(S)$ has a natural projection to $\text{Mod}(S)$ by taking mapping classes of restrictions to the zero section: $[\varphi] \mapsto [f_\varphi]$ is a surjective homomorphism $\text{HMod}^n(S) \to \text{Mod}(S)$. If we denote the kernel of this map by $K$, there is a split short exact sequence $$\begin{tikzcd} 1 \arrow[r] & K \arrow[r] & \text{HMod}^n(S) \arrow[r] & \text{Mod}(S) \arrow[r] & 1, \end{tikzcd}$$ with splitting given by $i: [f] \mapsto [f^\#]$. We call elements of $\text{HMod}^n(S)$ \textit{degree-$n$ mapping classes} and adopt the terminology that a mapping class $[\varphi] \in \text{HMod}^n(S)$ is \textit{$n$-standard} if it is in $i(\text{Mod}(S))$ and \textit{exotic} otherwise. A question of Thomas is if exotic degree-$n$ mapping classes exist for $n > 2$. In any matter, $\text{HMod}^n(S)$ acts on $\mathcal{T}^n(S)$ viewed as $\mathbb{M}^n(S)/\mathcal{G}^n(S)$ by $[I] [\varphi] = [I\varphi]$.
 

\section{The First Variation Formula}\label{section-first-var}

In this section, we discuss the infinitesimal action of $\text{Ham}_c^0(T^*(S))$ on the space of $n$-complex structures. The main computational tool we apply in our later analysis of $\mathcal{G}^n(S)$ is the first variation formula of \cite{fock2021higher}, which we recall as Proposition \ref{first-variation-formula}. We also give a new geometric interpretation of the first variation formula in Section \ref{subsection-maass}.

Let $\Sigma$ be a Riemann surface compatible with the orientation of $S$. As before, in a complex coordinate $z = x + iy$, a basis determined by $z$ for pointwise degree-1 polynomials in $T^{*\bbC}M$ is given by $p = \frac{1}{2}(p_x - i p_y)$, $\overline{p} = \frac{1}{2}(p_x + i p_y)$, corresponding to $\del_z$ and $\del_{\bar{z}}$ respectively.

The basic formula underlying infinitesimal analysis of $\mathcal{G}^n(S)$ is an expression of the Poission bracket in complex coordinates. Considering the restriction of polynomials to the real subspace of the complexified contangent bundle, a computation shows that the Poisson bracket of $w(z)p^k \bar p^l$ and $a(z)p^m \bar p ^n$ is given by \begin{align*} \{w p^k \bar p^l, a p^m \bar p^n \} &= \del_p (a p^m \bar p^n)\del_z (w p^k \bar p^l) + \del_{\bar p} (a p^m \bar p^n) \del_{\bar z} (w p^k \bar p^l) - \del_z (a p^m \bar p^n)\del_p (w p^k \bar p^l) \\ &\qquad - \del_{\bar z} (a p^m \bar p^n)\del_{\bar p} (w p^k \bar p^l). 
\end{align*}
We emphasize here that as in Section \ref{2-complex}, even though an expression is given in complex coordinates, the Hamiltonian flow is only acting directly on the restriction of complex polynomials to the real cotangent bundle. The complex coordinate expression arises from taking unique complex polynomial extensions to $T^{*\bbC}S$.

 A useful simplification to infinitesimal analysis of higher diffeomorphisms is that the first variation of $I \in \mathbb{M}^n(S)$ under the flow generated by a Hamiltonian $H$ only depends on $j_{n-1}^0(H) \pmod I$.

\begin{lemma}[Fock-Thomas \cite{fock2021higher}]\label{lemma-work-mod-I}
Let $\Sigma \in \mathcal{T}(S)$ be fixed. Let $H$ be a Hamiltonian, and identify $j_{n-1}^0(H)$ with a function on $T^*S$. Then the first variation of $I \in \mathbb{M}^n(S)$ locally generated by $f$ under the Hamiltonian flow generated by $H$ in centered coordinates associated to $\Sigma$ is given by $\{ f,j_{n-1}^0(H)\} \pmod I = \{f, j_{n-1}^0(H) \pmod I\} \pmod I$.
\end{lemma}

Modulo an ideal $I = \langle -\overline{p} + \mu_2 p + ... + \mu_n p^{n-1} \rangle$ in local coordinates, the $(n-1)$-jet of any Hamiltonian $H$ can be written uniquely in the local form $w_{1} p + w_2 p^2 + ... + w_{n-1}p^{n-1}$, where $w_k \in \Gamma(K^{-k}_\Sigma)$ globally. We call this the \textit{normalized form of $H$ in centered coordinates based at $\Sigma$}.
We have the following first variation formula.

\begin{proposition}[Fock-Thomas \cite{fock2021higher}]\label{first-variation-formula} The first variation of an $n$-complex structure $I$ under a Hamiltonian flow generated by $H \equiv w_k \pmod I$ with $w_k \in \Gamma(K^{-k}_\Sigma)$ is given in centered coordinates based at $\Sigma$ by $$ \dot{\mu}_l = \begin{cases}
 -\delbar w_k + \mu_2 \del w_k - k w_k \del \mu_2  & \text{if $l = k+1$}, \\
 (l-k)\mu_{l-k+1} \del w_k - k w_k \del \mu_{l-k+1}  & \text{if } l > k+1, \\
 0 & \text{if } l < k+1.
 \end{cases} $$
\end{proposition}

We remark that our convention that higher diffeomorphisms act on the right results in a negation of the first variation formula from what appears in \cite{fock2021higher}.

A first indication of linearity phenomena in $\mathcal{G}^n(S)$ is that when $\mu_2 = 0$, the $l=k+1$ term of the Fock-Thomas first variation formula is linear and independent of $\mu_3,..., \mu_{k+1}$ for Hamiltonians of the form $w_k \pmod I$ with $k \geq 2$. 

\subsection{Geometric Interpretation}\label{subsection-maass}

The first variation formula of Fock and Thomas plays a central role in our analysis, and admits a concrete description in terms of Maa\ss \,  derivatives, which we explain. The $l=k+1$ term has the most complicated description, and most of the subsection is spent investigating it. Roughly, the formula for $\dot{\mu}_{k+1}$ (Proposition \ref{geometric-first-var}, formula \eqref{geometrized-k+1-term}) says that, viewed correctly, $\dot{\mu}_{k+1}$ is ``seeing" the first variation formula centered at the Riemann surface represented by $\mu_2$ from the perspective of the Riemann surface used to center coordinates.

See \cite{wolpert1986chern} for definitions and basic features of Maa\ss \, operators. Some similar computations appear in \cite{d-Hoker2015higher} for variations of period matricies, covariant derivatives, and Green's functions.

In order to get the desired geometric interpretation, we slightly modify the normalization of standard forms of $n$-complex structures. The normalization used by Fock and Thomas is that an $n$-complex structure locally has a single generator of the form $-\overline{p} + \mu_2 p + ... + \mu_n p^{n-1}$, which precisely matches the standard setup of Beltrami differentials in the $n =2$ case. An alternative but equivalent normalization is that a $n$-complex structure locally has a single generator of the form $\overline{p} + \mu_2 p + ... + \mu_n p^{n-1}$. An analogous characterization of $n$-complex structures to Proposition \ref{fock-thomas-global-coordinates} holds in this normalization. We call this normalization \textit{negative normalization}, since the tensor coordinates of an element of $\mathbb{M}^n(S)$ negate when changing normalization.

Negative normalization seems well-suited for computations involving generators of higher complex structures, and appears as the preferred normalization for the coordinate system introduced in Section \ref{section-natural-coords}. Here, the reason for adopting negative normalization is for a slight change it introduces to the first variation formula. In negative normalization, the corresponding formula to Proposition \ref{first-variation-formula} is $$ \dot{\mu}_l = \begin{cases}
 \delbar w_k + \mu_2 \del w_k - k w_k \del \mu_2  & \text{if $l = k+1$}, \\
 (l-k)\mu_{l-k+1} \del w_k -  k w_k \del \mu_{l-k+1} & \text{if } l > k+1, \\
 0 & \text{if } l < k+1.
 \end{cases} $$ Note the change in the $l=k+1$ term of $-\delbar w_k$ to $\delbar w_k.$ The proof is identical to that of the first variation formula in \cite{fock2021higher}.

Now let $\Sigma$ be a Riemann surface with local coordinate $z$. Let $\Sigma'$ be the Riemann surface determined by the Beltrami differential $\mu_2$ with local coordinate $w$. Many basic identities we use in the following can be found in \cite{d-Hoker2015higher}. One computes that tensors on $\Sigma'$ given in $w$-coordinates can be represented in $z$-coordinates by
\begin{align*}
    \eta(w) d\overline{w}^n &= \sum_{k=0}^n { n \choose k} \eta(w) (\partial_{\bar z}\overline{w})^n (-\overline {\mu}_2)^k (d\bar z)^{n-k}  (dz)^k,\\
    \eta(w) \partial_{w}^n &= \sum_{k=0}^n {n \choose k} \eta(w) (\partial_w z)^n (\overline{\mu}_2)^k \del_z^{n-k}  \del_{\overline{z}}^k.
\end{align*} For a general tensor $\omega(w)$ of type $(k, l)$ with $k, l \in \bbZ$, we define $\iota_{z \leftarrow w}$ to be the $(k,l)$-part of $\omega(w)$ represented in $z$-coordinates. For example, the above shows that $\iota_{z \leftarrow w} \eta(w) d\overline{w}^n = \eta(w) (\partial_{\bar z} \overline{w})^n d\bar{z}^n$. Composition of these operators differs from the identity by a Jacobian term: $$\iota_{z \leftarrow w} \iota_{w \leftarrow z} (\varphi(z) d\bar{z}^n) = \varphi(z) (\partial_{\bar z} \overline{w})^n (\partial_{\overline{w}} \bar z)^n d\bar z ^n = \frac{\varphi(z)}{{(1-|\mu_2|^2)^n} }d\bar z^n. $$ Similar Jacobian terms appear when applied to more general tensors.

On tensors of type $(-n, 0)$ on a Riemann surface $\Sigma$, the Maa\ss \,$\delbar$ operator $\delbar^{\text{Maa\ss}}_\Sigma : \Gamma(K^{-n}_\Sigma) \to \Gamma(K^{-n}_\Sigma \overline{K}_\Sigma)$ has coordinate expression $\varphi(z) d/dz^n \mapsto -\delbar \varphi(z)(d\bar z/dz^n)$.
 When a Riemann surface $\Sigma'$ is represented by a Beltrami differential $\mu_2$, we sometimes denote the corresponding Maa\ss \, derivative by $\delbar_{\mu_2}^\text{Maa\ss}$. The following computation locates the $l=k+1$ term of the first variation formula in terms of these operations on tensors.
 
\begin{lemma} For $(-n, 0)$ tensors on $\Sigma$, with $z$ a local holomorphic coordinate for $\Sigma$ and $w$ a local holomorphic coordinate for the Riemann surface represented by $\mu_2$ on $\Sigma$,
    $$\displaystyle{(1 - |\mu_2|^2)^{n+1} \iota_{z \leftarrow w}\left(\delbar_{\mu_2}^{{\text{\rm{Maa\ss}}}}\left(\iota_{w \leftarrow z} \left(\varphi(z) \frac{d}{dz^n}\right)\right) \right) } = (-\del_{\overline{z}} \varphi(z) - \mu_2 \partial_z \varphi (z) +n \varphi(z) \partial_z \mu_2) \frac{d \bar{z}}{d z^n}.$$
\end{lemma}

\begin{proof}
We have $\iota_{w \leftarrow z} \varphi(z) \partial_z^n = \varphi(z) (\partial_z w)^n \partial_w^n,$ and compute
\begin{align*}
    \partial_{\overline{w}} (\varphi(z) (\partial_z w)^n) &= {(\partial_{\overline{w}} \overline{z})}(\partial_{\bar z} + \mu_2 \partial_z) (\varphi(z) (\partial_z w)^n) \\
    &= (\partial_{\overline{w}} \overline{z}) \bigg[(\partial_{\bar z}\varphi(z))(\partial_z w)^n + n \varphi(z) (\partial_z w)^{n-1} (\partial_{\bar z} \partial_z w) + \mu_2 (\partial_z \varphi(z)) (\partial_z w)^n \\
    & \qquad \qquad \quad + n \mu_2 \varphi(z) (\partial_z w)^{n-1} (\partial_z \partial_z w)\bigg] \\
    &= (\partial_{\overline{w}} \overline{z}) \bigg[((\partial_{\bar z}\varphi(z) + \mu_2 \partial_z \varphi(z))(\partial_z w)^n + n \varphi(z) (\partial_z w)^{n-1} [\partial_z( \partial_{\bar z}w + \mu_2 \partial_z w)] +  \\
    & \qquad \qquad \quad - n\varphi(z) (\partial_z w)^{n} (\partial_z \mu)\bigg] \\
    &= (\partial_{\overline{w}} \overline{z}) (\partial_z w)^n (\partial_{\bar z} \varphi(z) + \mu_2 \partial_z \varphi(z) - n \varphi(z) \partial_z \mu) \\ & \qquad \qquad \quad  + n \varphi(z) (\partial_{\overline{w}}\overline z) (\partial_z w)^{n-1} \del_z \left[ \frac{\partial_{\overline{w}} w}{\partial_{\overline{w}} \overline{z} }  \right] \\
    &= (\partial_{\overline{w}} \overline{z}) (\partial_z w)^n (\partial_{\bar z} \varphi(z) + \mu_2 \partial_z \varphi(z) - n \varphi(z) \partial_z \mu).
\end{align*} Expanding this $(-n,1)$ tensor in $w$-coordinates with respect to $z$-coordinates,
\begin{align*}
    \delbar_w^{\text{Maa\ss}}\left(\iota_{w \leftarrow z} \varphi(z) \partial_z^n \right) &= -\left( \sum_{k=0}^n {n \choose k} (\partial_{\bar z} \varphi + \mu_2 \partial_z \varphi - n \varphi \partial_z \mu)(\partial_z w)^n (\partial_w z)^n (-\overline{\mu}_2)^k \partial_z^{n-k} \partial_{\bar z}^k  \right) \\ &\qquad \qquad \cdot ((\partial_{\overline{w}} \bar z ) (\partial_{\bar z} \overline{w} )  (d\bar z - \overline{\mu}_2 dz)).
\end{align*} Inspecting terms yields 
\begin{align*}
\iota_{z \leftarrow w}\left(\delbar_w^{\text{Maa\ss}}\left(\iota_{w \leftarrow z} \left(\varphi \partial_z^n \right)\right) \right) &= (\partial_z w)^n (\partial_w z)^n (\partial_{\overline{w}} \bar z) (\partial_{\bar z} \overline{w}) (-\delbar_z \varphi - \mu_2 \partial_z \varphi + n \varphi \partial_z \mu_2 ) \frac{d \bar{z}}{d z^n} \\
&= \frac{-\delbar_z \varphi - \mu_2 \partial_z \varphi + n \varphi \partial_z \mu_2 }{(1- |\mu_2|^2)^{n+1}} \frac{d \bar{z}}{d z^n}\end{align*} \end{proof} 

The $l > k+1$ terms of the first variation formula can be described in terms of Maa\ss\, derivatives on the base Riemann surface $\Sigma$ as follows. In a coordinate $z$, express the hyperbolic metric $g_\Sigma$ in the conformal class of $\Sigma$ as $g_\Sigma = \sigma^2 |dz|^2$ with $\sigma(z) > 0$. For any $(k, j)$-tensor $w \in \Gamma(K_\Sigma^{k}\overline{K}_\Sigma^j)$ with $k, j\in \mathbb{Z}$, the Maa\ss\, $\del$-derivative of $w$ is expressed $$\del_\Sigma^\text{Maa\ss} (w(z) dz^k d\overline{z}^j) = (\del w(z) - ik \sigma(z) w(z)) dz^{k+1} d\overline{z}^j.$$ For a uniformized Riemann surface $\Sigma = \mathbb{H}^2/\Gamma$, the description of the Maa\ss\, $\del$-derivative in terms of functions on $\mathbb{H}^2$ is $$\del_{\Sigma}^{\text{Maa\ss}} w(z) = \del w(z) + \frac{2k}{(z-\overline{z})} w(z).$$ The description above of the Maa\ss\, $\del$-derivative on an arbitrary Riemann surface can be obtained by pulling back this expression under a biholomorphism to a uniformized Riemann surface.

In any matter, in the setting of the first variation formula, for $l > k+1$ we have that $\mu_{l-k+1} \in \Gamma(K^{k-l}_\Sigma \overline{K}_\Sigma)$. So in a local coordinate $z$, \begin{align*}
    (l-k) &\left(\mu_{l-k+1}(z) \frac{d\overline{z}}{dz^{l-k}} \right) \del^{\text{Maa\ss}}_\Sigma \left( w_k(z)\frac{d}{dz^k} \right) - k \left( w_k(z) \frac{d}{dz^k} \right) \del^{\text{Maa\ss}}_\Sigma \left( \mu_{l-k+1}(z) \frac{d\overline{z}}{dz^{l-k}} \right)  \\ &\quad= \left[(l-k) \mu_{l-k+1}(z) \del w_k(z) + ik (l-k)\sigma(z) w_k(z) \mu_{l-k+1}(z) \right] \frac{d\overline{z}}{dz^l} \\ & \qquad \quad - \left[ k w_k(z) \del \mu_{l-k+1}(z) + i k(l-k) \sigma(z) w_k(z) \mu_{l-k+1}(z) \right] \frac{d\overline{z}}{dz^l} \\
    &\quad = [(l-k) \mu_{l-k+1}(z) \del w_k(z) - k w_k(z) \del \mu_{l-k+1}(z)] \frac{d\overline{z}}{dz^l}, 
\end{align*} which describes the $\dot{\mu}_l$-term of the first variation formula in terms of Maa\ss\, derivatives. In summary:

\begin{proposition}\label{geometric-first-var}
    For a $n$-complex structure $I = (\mu_2, ..., \mu_n)$, the first variation of $\mu_{k+1}(I)$ under a Hamiltonian flow generated by $H \equiv w_k \pmod I$ with $w_k \in \Gamma(K^{-k}_\Sigma)$ is given in negative normalization centered coordinates based at $\Sigma$ by \begin{align}\label{geometrized-k+1-term} \dot{\mu}_{k+1} = -(1 - |\mu_2|^2)^{k+1} \iota_{z \leftarrow w}\left(\delbar_{\mu_2}^{\text{{\rm{Maa\ss}}}}\left(\iota_{w \leftarrow z} w_k\right) \right).\end{align} 
    For $l > k+1$, the first variation of $\mu_l(I)$ under the Hamiltonian flow generated by $H$ is given in negative normalization centered coordinates based at $\Sigma$ by \begin{align}\label{geometrized-higher-order-terms} \dot{\mu}_l = (l-k)\mu_{l-k+1} \del_\Sigma^{\text{{\rm{Maa\ss}}}} w_k - k w_k \del_{\Sigma}^{\text{{\rm{Maa\ss}}}} \mu_{l-k+1}.\end{align}
\end{proposition}

Formula \eqref{geometrized-k+1-term} makes precise the idea that $\dot{\mu}_{k+1}$ in Proposition \ref{first-variation-formula} is ``seeing" the first variation in centered coordinates based at the Riemann surface represented by $\mu_2$ from the perspective of the reference Riemann surface. We remark that one auxiliary effect of Proposition \ref{geometric-first-var} is to give a verification that, as expected, all terms in the first variation formula are globally tensors of appropriate types on the base Riemann surface $\Sigma$.

If $S$ has genus at least 2, there are no nonzero holomorphic sections $w \in \Gamma(K^{-k}_{\Sigma'})$ for $k \geq 1$ and any complex structure $\Sigma'$ on $S$. We will have use of the following corollary of the preceding proposition.

\begin{corollary}\label{no-surprise-stationaries}
    Let $S$ have genus $g \geq 2$. If $H \not \equiv 0 \pmod I$ has first nonzero term in normalized form $w_k \pmod I$ with $w_k \in \Gamma(K_\Sigma^{-k})$, then $\mu_{k+1}(I)$ has nonzero first variation under the flow of $H$. 
\end{corollary}

\section{Natural Coordinates}\label{section-natural-coords}
Our eventual goal is to analyze the quotient $\mathcal{T}^n(S) = \mathbb{M}^n(S)/\mathcal{G}^n(S)$ by viewing it as a quotient taken in two steps---first finding distinguished representatives of the action by $\mathcal{N}(S)$, then taking the quotient of the space of these representatives by $\text{Diff}_0(S)$. Making this work requires that the action of $\text{Diff}_0(S)$ on $\mathbb{M}^n(S)$ by lifts of diffeomorphisms to $T^*S$ restrict to an action on our distinguished elements of $\mathcal{N}(S)$-orbits. The representatives we produce in Section \ref{quotient-section} do not satisfy this in centered coordinates, for reasons arising from centered coordinates' dependence on a reference Riemann surface.

In this section, we produce coordinates on $\mathbb{M}^n(S)$ in which the collection of harmonic $n$-complex structures (Definition \ref{def-harmonic-n-complex}) is invariant under $\text{Diff}_0(S)$ (Proposition \ref{harmonic-natural}). Our coordinates are obtained without a choice of a reference Riemann surface through a modification of the procedure used to produce centered coordinates in \cite{fock2021higher}.

We begin by reviewing the construction of the associated complex structure to an $n$-complex structure (see \cite{fock2021higher}) from the perspective we make use of in the construction of natural coordinates. Let $I \in \mathbb{M}^n(S)$ and $x \in S$. Then $I_x$ is a codimension-$n$ ideal in $\bigoplus_{j=0}^{n-1} \mathcal{S}^j (T_x^{\bbC}S)$ with $I_x + \overline{I}_x$ maximal. Fock and Thomas show that maximality of $I_x + \overline{I}_x$ implies the collection of degree-$1$ homogeneous polynomials appearing in elements of $I_x$ has dimension $1$ (\cite{fock2021higher}, Section 5.1). Let $f_x$ be a generator.

Then $f_x$ is a nonzero degree-$1$ homogeneous complex polynomial, and so has a unique zero in $\mathbb{CP}^1$ identified with the space of complex lines in $T^{*\mathbb{C}}_xS$. That $T^{*\mathbb{C}}_xS$ arises as a complexification endows this copy of $\mathbb{CP}^1$ with complex conjugation, a distinguished real axis, and an orientation. By possibly interchanging $I$ and $\overline{I}$ we may consistently arrange for $\text{Im}(f^{-1}_x(0)) > 0$.

This determines for any $x$ an endomorphism $J_x^*: T_x^{*\bbC}S \to T_x^{*\bbC}S$ by specifying that $J_x^*|_{f_x^{-1}(0)} = i\text{Id}$ and
$J_x^*|_{\overline{f_x^{-1}(0)}} = -i \text{Id}$. Then $J_x^*$ restricts to an endomorphism of $T_x^*M$ so that $(J_x^*)^2 = -\text{Id}$. Taking duals identifies almost complex structures on $S$ and endomorphisms $J^* \in \text{End}(T^*S) $ so $(J^*)^2 = -\text{Id}$. This associates an almost complex structure $J(I)$ to $I$, and thus a complex structure $\Sigma(I)$ to $I$ by the Newlander-Nirenberg theorem.

The almost complex structure $J(I)$ then determines a decomposition by type of $$T^{\bbC}S = T^{(1, 0)}S \oplus T^{(0,1)}S = K^{-1}_J \oplus \overline{K}^{-1}_J$$ and in turn decompositions $\mathcal{S}^j(T^{\bbC} S) = \bigoplus_{k=0}^j K^{-k}_J\overline{K}^{k-j}_J$ for $1\leq j \leq n-1$.  

We now take generators of linear terms of $I$ in neighborhoods. Denote by $\pi$ the projection $T^*S \to S$. For any $x$ in $S$, there is a neighborhood $U$ of $x$ and a function $g : \pi^{-1}(U) \to \mathbb{C}$ vanishing on the zero section so that $I_x = \langle (j_{n-1}^0(g))_x\rangle$ for all $x \in U$. The polynomial $j_{n-1}^0(g)_x$ must have nonzero degree-$1$ term $f_x$ for all $x$ in $U$ by maximiality of $I + \overline{I}$. Once again, we may arrange for $\text{Im}(f_x^{-1}(0)) > 0$ for all $x \in U$. Our choice of $J(I)$ ensures that $f_x$ vanishes on the $i$-eigenspace of $J(I)$ and so must have tensor type $(0,-1)$ at $x$. In symbols, $f_x \in T^{(0,1)}_xS$. The relation $g \equiv 0 \pmod I$ then defines a map $$\widetilde{P}_U: T^{(0,1)}U \to \bigoplus_{j=2}^{n-1} \mathcal{S}^{j}(T^{\bbC}U).$$ Repeated application of $\widetilde{P}_U$ defines a projection $$P_U: \bigoplus_{j=1}^{n-1} \Gamma(\mathcal{S}^j(T^{\bbC} U)) \to \bigoplus_{j=1}^{n-1} \Gamma(K^{-j}_J(U))$$ so that ${P}_U h \equiv h \pmod I$ for any $h \in \bigoplus_{j=1}^{n-1} \Gamma (\mathcal{S}^j(T^{\bbC}U))$ and $P_U$ restricted to $\bigoplus_{j=1}^{n-1} \Gamma(K^{-j}_J(U))$ is the identity.

Now let $g' = j_1^0(g) + P_U(j_{n-1}^0 (g) - j_1^0(g))$ be viewed as a sum of homogeneous polynomials. The above implies that $I_x = \langle g'_x \rangle$ for all $x \in U$. By construction, $g'$ has the form \begin{align}\label{invariant-normalized-form} g' = w_{0,1}+ 0+  w_{2, 0}+ ... + w_{n-1,0}\end{align} where $w_{i,j} \in \Gamma(K^{-i}_J(U) \overline{K}^{-j}_J(U))$ and $w_{0,1}$ vanishes nowhere. All generators of $I_U$ of the form \eqref{invariant-normalized-form} are of the form $hg'$ with $h$ a nowhere vanishing function on $U$, because $I_x$ has codimension $n$ for $x \in U$. So the collection of sections $\mu_{k,U} = w_{k-1, 0}/w_{0,1}$ of $\overline{K}_J(U) {K}_J^{1-k}(U)$ for $3 \leq k \leq n$ are independent of the choice of a generator $g'$ of the form of formula \eqref{invariant-normalized-form}.

If $V$ is another open set, the independence of $\mu_{3,V}, ..., \mu_{n,V}$ on the generator $g'$ of the form of formula \eqref{invariant-normalized-form} ensures that $\mu_{k,V}(x) = \mu_{k,U}(x)$ for all $x \in U \cap V$ and $3\leq k \leq n$. So for $x \in S$, defining $\mu_k(x) = \mu_{k,U}(x)$ for any open set $U$ containing $x$ produces a global family of sections $\mu_k = \mu_k(I)$ of $\overline{K}_JK_J^{1-k}$.

The almost complex structure $J(I)$ and the associated sections $\mu_3(I),..., \mu_n(I)$ completely determine $I$ and depend smoothly on $I$. Conversely, any almost complex structure $J$ and sections $\{ \mu_{k}\}_{3 \leq k \leq n}$ of $\overline{K}_JK_J^{1-k}$ produce an $n$-complex structure. Denote $\mathbb{L}^n(S) = \{(\Sigma, \mu_3, ..., \mu_n) \mid \Sigma \in \mathbb{M}(S), \mu_k \in \Gamma(\overline{K}_\Sigma {K}_\Sigma^{1-k})\}$. Then the map \begin{align*} \Phi^n: \mathbb{M}^n(S) &\to \mathbb{L}^n(S) \\ I &\mapsto (\Sigma(I), \mu_3(I), ..., \mu_n(I))\end{align*} gives coordinates on $\mathbb{M}^n(S).$

\begin{definition}\label{def-harmonic-n-complex} The coordinates given by the map $\Phi^n$ are called \textit{natural coordinates} on $\mathbb{M}^n(S)$. For an $n$-complex structure $I$, we denote the natural coordinates of $I$ by $\Sigma(I)$ and $\mu_3(I),..., \mu_n(I)$.
\end{definition}

Natural coordinates on $\dot{\mathbb{M}}^n(S)$ are given by adjoining markings. When working with natural coordinates, it is sometimes technically simpler to work with almost complex structures rather than complex structures in the first coordinate, which is equivalent by the Newlander-Nirenberg theorem.

The main advantage of these coordinates to us is the following.

\begin{proposition}\label{diff-lift-in-natural}
    If {\rm{$f \in \text{Diff}^+(S)$}}, then the action of $f^\#$ on $\mathbb{M}^n(S)$ in natural coordinates is given by $(J, \mu_3, ..., \mu_n)(f^\#) = (f^*J, f^*\mu_3, ..., f^* \mu_n)$.
\end{proposition}

The analogue of Proposition \ref{diff-lift-in-natural} for marked $n$-complex structures is $(J, \mu_3, ..., \mu_n, \phi)f^\# = (f^* I, f^* \mu_3, ..., f^* \mu_n, \phi \circ f)$ for $f \in \text{Diff}^+(S)$.

\begin{proof}
    We address the first coordinate first. The action by diffeomorphisms on almost complex structures is given by $(f^* J) (v) = (Df)^{-1} J(Df) v$, so the dual endormorphism is $(f^*J)^* \alpha = \alpha \circ (Df)^{-1} \circ J \circ Df$. So if $V^*_x$ is the $i$ eigenspace of $J^*$ at $x$, the $i$-eigenspace of $(f^*J)^*$ at $x$ is $(f^\#)^{-1}(V^*_{f^{-1}(x)})$. On the other hand, if $g: T^{*C}S \to \mathbb{C}$ is linear on fibers with $V^*_x$ its zero in $\mathbb{CP}^1$ at $x$, then the zero of $g \circ f^\#$ at $x$ is $(f^\#)^{-1}(V^*_{f^{-1}(x)})$. We conclude that $J(I f^\#) = f^*(J(I))$. 
    
    As the action of $f^\#$ on $\mathcal{J}_1^1$ coincides with the contravariant action by diffeomorphisms on vector fields, the induced actions on $K^{-j}_J$ and $\overline{K}^{-j}_J$ of $f$ for $j \in \mathbb{N}$ through jets and pullback coincide. So if there is a local expression for generators of $I = \langle g \rangle$ of the form $g = w_{0,1} + w_{2,0} + ... + w_{n-1,0},$ then $g \circ f^\# = f^* w_{0,1} + f^* w_{2,0} + ... + f^* w_{n-1,0} $ is of the form \eqref{invariant-normalized-form} with respect to $J(If^\#) = f^*(J(I))$. The conclusion follows.
\end{proof}

The following shall be our distinguished representatives of $\mathcal{N}(S)$-orbits of $n$-complex structures in Section \ref{quotient-section}.

\begin{definition}\label{harmonic-n-cx} A section $\mu_k \in \Gamma (K^{1-k}_\Sigma \overline{K}_\Sigma)$ is called a {\rm{harmonic $k$-Beltrami differential}} if $\overline{\mu}_k g_\Sigma^{k-1}$ is a holomorphic $k$-adic differential, where $g_\Sigma$ is the unique hyperbolic metric in the conformal class of $\Sigma$. 

An $n$-complex structure $I$ written in natural coordinates as $(\Sigma,\mu_3,...,\mu_n)$ is a {\rm{harmonic $n$-complex structure}} if $\mu_k(I)$ is a harmonic $k$-Beltrami differential on $\Sigma$ for $3\leq k\leq n$. We denote by $\mathcal{H}\mathbb{M}^n(S)$ the collection of harmonic $n$-complex structures. \end{definition}

Proposition \ref{diff-lift-in-natural} yields a naturality statement for the collection of harmonic $n$-complex structures, and is essential to the proofs of our main results.

\begin{corollary}\label{harmonic-natural}
    $\mathcal{H}\mathbb{M}^n(S)$ is invariant under the action of {\rm{$\text{Diff}^+(S)$}} on $\mathbb{M}^n(S)$.
\end{corollary}

\begin{proof}
    If $\delbar_\Sigma (\overline{\mu}_k g_\Sigma^{k-1}) = 0$, naturality of pullback shows that for any $f \in \text{Diff}^+(S)$ we have $0 = \delbar_{f^*\Sigma} (f^*\overline{\mu}_k f^* g_\Sigma^{k-1}) = \delbar_{f^*\Sigma} (f^*\overline{\mu}_k g_{f^*\Sigma}^{k-1})$. So if $I = (\Sigma, \mu_3,..., \mu_n)$ is harmonic, then $I f^\# = (f^*\Sigma, f^*\mu_3,...,f^*\mu_n)$ is harmonic. 
\end{proof}

 In order to analyze flows in natural coordinates later, we need a description of the first variation formula in natural coordinates. For fixed $\Sigma \in \mathbb{M}(S)$, natural coordinates with $\mathbb{M}(S)$-coordinate $\Sigma$ agree with negative-normalization centered coordinates with respect to the reference complex structure $\Sigma $ in the sense that $(\Sigma, \mu_3,..., \mu_n)$ is represented in negative-normalization centered coordinates based at $\Sigma$ by $(0, \mu_3, ..., \mu_n)$ (see the appendix of \cite{fock2021higher}). As $2$-stationary diffeomorphisms fix the collection of elements of $\mathbb{M}^n(S)$ with underlying complex structure $\Sigma(I)$, the first variation formula yields the first variation of $n$-complex structures under $k$-stationary flows for $k \geq 2$ in natural coordinates.

\begin{proposition}
Fix an $n$-complex structure $I \in \mathbb{M}^n(S)$. Under a $2$-stationary flow generated by $H_t \equiv w_k \pmod I$ with $k\geq 2$ and $w_k \in \Gamma(K^{-k}_\Sigma)$, the first variation of $(\Sigma(I), \mu_3, ..., \mu_n)$ in natural coordinates is given by $$\dot{\mu}_l = \begin{cases} \delbar w_k & \text{if $l = k+1$}, \\
 (l-k)\mu_{l-k+1} \del w_k - k w_k \del \mu_{l-k+1}  & \text{if } l > k+1, \\ 0 & l < k+1. \end{cases} $$
\end{proposition}

\section{Structure of Higher Diffeomorphism Groups}\label{section-group-structure}

The canonical vector bundle structure we find on $\mathcal{T}^n(S) = \mathbb{M}^n(S)/\mathcal{G}^n(S)$ in Section \ref{quotient-section} arises from structure present in the degree-$n$ diffeomorphism group $\mathcal{G}^n(S)$. We have now constructed enough machinery to prove the relevant results about $\mathcal{G}^n(S)$.

The outline of our approach is to decompose $\mathcal{G}^n(S)$ as a semidirect product $\text{Diff}_0(S)\ltimes \mathcal{N}(S)$, where $\mathcal{N}(S)$ is as before the group of $2$-stationary diffeomorphisms. We then analyze $\mathcal{N}(S)$ on its own and combine this with the basics from Section \ref{2-complex} to get a picture of $\mathcal{G}^n(S)$. Separating the analysis of $\mathcal{N}(S)$ from that of $\text{Diff}_0(S)$ is helpful because the two behave quite differently.

The main structural result for the $2$-stationary diffeomorphism group $\mathcal{N}(S)$ is that it has strong linearity properties: $\mathcal{N}(S)$ is a contractible nilpotent regular Fr\'echet-Lie group with factor groups consisting of smooth sections of appropriate bundles, and the group operation on factors is addition. Our main results rely heavily on two claims we prove in this section that show that the only actions of $\mathcal{N}^k(S)$ on elements of $I \in \mathbb{M}^n(S)$ are as simple as possible (Proposition \ref{auto-decomposition} and Lemma \ref{non-collapse}).

An essential component of the analysis is to show that every element of $\mathcal{N}(S)$ is represented by a Hamiltonian diffeomorphism arising from an \textit{autonomous} flow. This is quite surprising at first glance, and is not true for $\text{Diff}_0(S)$. In fact it is well-known (see for example \cite{milnor1984remarks}) that even for the circle $S^1$, there are examples of diffeomorphisms arbitrarily close to the identity that are not realizable as autonomous flows. So to frame expectations of the structure of $\mathcal{N}(S)$, a different analogue than $\text{Diff}_0(S)$ is in order: the structure we find is also present in less complicated groups of jets than $\mathcal{G}^n(S)$.

\begin{example}\label{model-jet-group}
     Consider the group $N$ of $n$-jets at $0$ of maps $f: \mathbb{R} \to \mathbb{R}$ so $f(0) = 0$ and $f'(0) = 1$. Then $N \cong \{ x + a_2 x^2 + ... + a_n x^n \in \mathbb{R}[x]\}$ with product operation composition modulo degree $\geq n+1$ terms. One sees that the subgroup $N^k < N$ consisting of polynomials of the form $x + f$ with the minimal degree of terms of $f$ equal to $k$ is normal in $N$, and the group operation on $N^k$ fixes degree $< k$ coefficients. Furthermore, the group operation on $N^k$ is addition on degree $k$ coefficients. Also, $[N, N^k] \subset N^{k+1}$, so that $N$ is nilpotent. The factor groups $N^k/N^{k+1}$ are isomorphic to $(\mathbb{R}, +)$.
     
    $N$ has the simplest possible exponential map. To see this, note that as $N$ is nilpotent the Baker-Campell-Hausdorff series of its Lie algebra $\mathfrak{n}$ converges on all of $\mathfrak{n}$. So the Baker-Campell-Hausdorff series gives a Lie group structure $\widetilde{N}$ on $\mathfrak{n}$ so that the exponential map is the identity. As $N$ is simply connected, the identity map $\mathfrak{n} \to \mathfrak{n}$ integrates to an isomorphism $N \to \widetilde{N}$.
\end{example}

Two features of $\mathcal{G}^n(S)$ are helpful to keep in mind in the following. The first is the phenomenon that for $\varphi \in \text{Ham}_c^0(T^*S)$, the $k$-jet $j_k(\varphi)$ of $\varphi$ vanishing to degree $j$ along $Z^*S$ does not directly correspond to $\varphi$ representing an element of $\mathcal{N}^{j+1}(S)$. See Section \ref{2-complex} for the $j=1$ case. No such phenomenon appears in the first variation formula, which inspires us to instead seek a parametrization in terms of Hamiltonians. 

The second feature is that the action of $\mathcal{G}^n(S)$ on $\mathbb{M}^n(S)$ is effective but not free. This would complicate an analysis of the quotient $\mathbb{M}^n(S)/\mathcal{G}^n(S)$ by more general (see for instance, \cite{fischer1984purely}) methods of understanding quotients of infinite-dimensional Lie groups that are not Banach-Lie groups. The following example can be used to construct examples of nontrivial stabilizers of higher complex structures on any surface.

\begin{example}\label{example-not-free}
    Consider the trivial $3$-complex structure $I_0$ on the disk given by $\mu_2 = \mu_3 = 0$ in negative normalization centered coordinates. For any function $\eta(x,y)$ vanishing on a neighborhood of $\partial \mathbb{D}$, the flow $\varphi_t$ generated by $H = \eta(x,y) (\partial_x^2 + \partial_y^2)$ fixes $I_0$ by Proposition \ref{first-variation-formula}. On the other hand, for a $3$-complex structure of the form $I = \langle \overline{p} + \mu_2 p  \rangle$, we have $\eta(x,y) (\partial_x^2 + \partial_y^2) \equiv- 4\eta \mu_2 p^2 \pmod I$. The first variation formula shows this does not fix $I$ whenever there is a point where neither $H$ nor $\mu_2$ vanishes. 
\end{example}

The lack of freeness of the previous example can be controlled in the sense that autonomous Hamiltonian flows always act nontrivially on some $k$-complex structure of the smallest possible $k$.

\begin{lemma}\label{action-kernel}A flow by a Hamiltonian $H_t$ that vanishes on $Z^*S$ is $k$-stationary if and only if $j_{k-1}^0(H_t)$ vanishes identically for all $t$.
\end{lemma}

\begin{proof}
    Any $H_t$ with $j_{k-1}^0(H_t) = 0$ for all $t$ acts trivially on $\mathbb{M}^k(S)$ by Proposition \ref{first-variation-formula}. Conversely, suppose that $H_t$ has $j_{k-1}^{0}(H_t)$ not identically $0$ for some $t_0$. Represent the $(k-1)$-jet of $H_{t_0}$ with homogeneous polynomials $a_1(z, v) + a_2(z,v) + ... + a_{k-1}(z,v)$. Let $1\leq j < k$ be the minimal integer so that $a_j(z, v) \not \equiv 0$, and work at a point $z_0$ so that $a_j(z_0, \cdot) \neq 0$.
    
     With respect to any almost complex structure $J_{z_0}^*$ on the fiber $T_{z_0}^*S$, the decomposition of $a_j(z_0, \cdot)$ by type has no $K_{J_{z_0}^*}^{-j}$ term if and only if $a_j(z_0, \cdot)$ vanishes on the $i$-eigenspace of $J_{z_0}^*$. As $a_j(z_0, \cdot)$ has only finitely many projective zeros, this is only possible for finitely many choices of $J_{z_0}^*$. Thus there are $n$-complex structures $I \in \mathbb{M}^n(S)$ so $a_j(z_0, \cdot)$ has nonzero $K^{-j}_{\Sigma(I)}$ part at $z_0$. For any such $I$, Corollary \ref{no-surprise-stationaries} shows $I$ is not stationary under the flow of $H_{t-t_0}$ at $t = 0$. Then, denoting the time-$t$ flow of $H_t$ by $\varphi_{t}$, in centered coordinates based at $\Sigma(I)$, we have $\frac{d}{dt}\big|_{t=t_0} \mu_{j+1} ((I(\varphi_{t_0})^{-1})\varphi_t) \neq 0.$ So the flow of $H_t$ is not $k$-stationary. \end{proof}

We now turn to examining $\mathcal{N}(S)$. The general strategy of the remainder of this subsection is to show that every element of $\mathcal{N}^k(S)$ arises as a $k$-stationary flow and leverage this to produce a decomposition of $h \in \mathcal{N}(S)$ as a product of increasingly stationary autonomous flows.

Recall the notation from Section \ref{2-complex} that $\text{Ham}_c^P(T^*S)$ is the subgroup of $\text{Ham}_c^0(T^*S)$ whose elements pointwise fix $Z^*S$. Lemma \ref{2-diffeos-just-diffeos} implies that $\varphi \in \text{Ham}_c^0(T^*S)$ is $2$-stationary if and only if it is in $\text{Ham}_c^P(T^*S)$, which is used without being explicitly mentioned throughout the following. Our first observation is that $\mathcal{N}(S)$ is connected through $2$-stationary flows.

\begin{lemma}\label{2-stationary-connected}
    If $h \in \mathcal{N}(S),$ then there is a $2$-stationary flow {\rm{$\varphi_t \in \text{Ham}_{c}^P(T^*S)$}} so that {\rm{$\varphi_0 = \text{Id}$}} and $[\varphi_1] = h$.
\end{lemma}

\begin{proof}
    Let $h \in \mathcal{N}(S)$ be represented by a symplectomorphism $\varphi$ generated by a compactly supported Hamiltonian flow $\varphi_t$ setwise fixing $Z^*S$. Let $\psi_t$ be a smooth path of Hamiltonian diffeomorphisms identical to $ (\varphi_t|_{Z^*S})^\#$ on a neighborhood of $Z^*S$ and supported on a compact set. Such paths exist by Lemma \ref{lifts-get-realized}.
    
    Then $\psi_t$ is smooth in $t$ and $\psi_0 = \psi_1 = \text{Id}$ in a neighborhood of $Z^*S$. So $(\psi_t)^{-1} \varphi_t$ is a path in $\text{Ham}_c^P(T^*S)$ between $\text{Id}$ and $(\psi_1)^{-1}\varphi_1 = \varphi$, and hence $[\psi_1^{-1}\varphi_1] = [\varphi]$. Such a path of Hamiltonian diffeomorphisms must be Hamiltonian (see Appendix \ref{appendix-community-service} or \cite{mcduff2017introduction} ch. 10).
\end{proof}

We remark that the restriction that elements of $\text{Ham}_c^0(T^*S)$ are generated by Hamiltonian flows fixing $Z^*S$ for all times $t$ was used essentially in the previous proof.

The following computation is quite elementary but used repeatedly in the following. Let $H$ be a function $T^*M \to \mathbb{R}$ with $j_{k-1}^0(H) = 0$ for some $k\geq 2$ with time-$1$ Hamiltonian flow $\varphi$. Let $I = (\Sigma, \mu_3, ..., \mu_n)$ be given. We compute $\mu_{k+1}(I\varphi)$ in natural coordinates.

Write the complexified $k$-jet of $H$ (mod $I$, degree $\geq k+1$ terms) as $w_{k} \in \Gamma(K^{-k}_{\Sigma})$. Let $\varphi_t$ be the flow of $H_t$. For any $I'$ with $\Sigma(I') = \Sigma(I)$, the first variation formula shows that $\frac{d}{dt}\mu_{k+1}(I\varphi_t) = \delbar w_k$, and that the flow of $H_t$ leaves the family of $n$-complex structures $I'$ with $\Sigma(I') = \Sigma(I)$ invariant. So \begin{align}\label{basic-auto-formula} \mu_{k+1}(I\varphi ) = \mu_{k+1}(I) +\delbar w_{k}.\end{align} In particular, if we denote the time-$1$ flow associated to a function $H$ with $j_{k-1}^0(H) = 0$ by $\varphi_H$, the map $H \mapsto \mu_{k+1}(I\varphi_H) - \mu_{k+1}(I)$ is linear in $H$.

The means by which we are able to reduce from arbitrary $k$-stationary flows to autonomous flows is an averaging procedure. This procedure produces from a $k$-stationary flow an autonomous $k$-stationary flow with the same time-$1$ action on $\mu_{k+1}$-coordinates of all $n$-complex structures as the original flow.

\begin{lemma}\label{partial-autonomy}
    If $\varphi_t$ is generated by a $k$-stationary ($k\geq 2$) Hamiltonian $H_t$ vanishing on $Z^*S$, let $\widetilde{H} = \int_0^1 H_t dt$ and $\widetilde{\varphi}$ be the time 1 flow of $\widetilde{H}$. Then for every $I \in \mathbb{M}^n(S)$, $\mu_{k+1}(I\varphi_1) = \mu_{k+1}(I\widetilde{\varphi})$.
    
    Furthermore, $\widetilde{H}$ is determined to degree $k$ in the sense that if $\widetilde{H}'$ is another $k$-stationary Hamiltonian vanishing on $Z^*S$ with time-$1$ autonomous flow $\widetilde{\varphi}'$ satisfying $\mu_{k+1}(I\varphi_1) = \mu_{k+1}(I\widetilde{\varphi}')$ for all $I \in \mathbb{M}^n(S)$, then $j_k^0(\widetilde{H}) = j_k^0(\widetilde{H}')$.
\end{lemma}

\begin{proof} We compute in natural coordinates. Let $I = (\Sigma, \mu_3, ..., \mu_n)$ be given and let $\varphi_t$ be the flow of $H_t$. Lemma \ref{action-kernel} shows that $H_t$ has vanishing $(k-1)$-jet for all times $t$. So the first variation formula shows that $\frac{d}{dt}\mu_{k+1}(I\varphi_t)$ depends only on the $k$-jet of $H_t$. Write the complexified $k$-jet of $H_t$ as $$H_t = \sum_{j=0}^k w_{j,k-j}(t) \equiv w_{k,0}(t) \pmod{I', \text{ degree $\geq k+1$ terms}} $$ where $w_{j, k-j}(t) \in \Gamma(K^{-j}_\Sigma \overline{K}^{j-k}_\Sigma)$ and $I'$ is any element of $\mathbb{M}^n(S)$ with $\Sigma(I') = \Sigma(I)$. As $H_t$ is $k$-stationary for some $k\geq 2$, we have $\frac{d}{dt} \mu_k(I\varphi_t) = \delbar w_{k,0}(t)$ for $0\leq t \leq 1$, so
\begin{align*}
    \mu_{k+1}(I\varphi_1) &= \mu_{k+1}(I) + \int_{0}^1 \delbar w_{k,0}(t) dt \\ &= \mu_{k+1}(I) + \delbar \int_0^1 w_{k,0}(t) dt.
\end{align*} Now, $\widetilde{w}_k = \int_{0}^1w_{k,0}(t)dt \in \Gamma(K^{-k}_{\Sigma})$ is equal$\pmod {I', \text{ degree $\geq k+1$ terms}}$ to the complexified $k$-jet of $\widetilde{H} = \int_{0}^1 H_t dt$ for any $I'$ with $\Sigma(I') = \Sigma(I)$. Denoting by $\widetilde{\varphi}_t$ the autonomous flow by $\widetilde{H}$, we see that as the flows of $H_t$ and $\widetilde{H}$ are $2$-stationary, $$\mu_{k+1}(I\varphi_1) = \mu_{k+1}(I) + \delbar \widetilde{w}_{k} = \mu_{k+1}(I) + \int_0^1 \frac{d}{dt} \mu_{k+1}(I\widetilde{\varphi}_t )dt = \mu_{k+1}(I \widetilde{\varphi}_1). $$

To see determination to degree $k$, suppose that $\widetilde{H}'$ generates $\widetilde{\varphi}_t'$ and $\mu_{k+1}(I\varphi_1) = \mu_{k+1}(I\widetilde{\varphi}_1')$ for all $I \in \mathbb{M}^n(S)$. Denote the time-$1$ flow of $\widetilde{H}-\widetilde{H}'$ by $\psi$. Recall that for every $n$-complex structure $I$ and autonomous Hamiltonian $H$ with $j_{k-1}^0(H) = 0$, denoting the time-$1$ flow of $H$ by $\varphi_H$, formula \eqref{basic-auto-formula} implies that $H \mapsto \mu_{k+1}(I\varphi_H) - \mu_{k+1}(I)$ is linear in $H$. So we have that $$\mu_{k+1}(I\psi) - \mu_{k+1}(I) = \mu_{k+1}(I\widetilde{\varphi}_1 ) - \mu_{k+1}(I) -\mu_{k+1}(I\widetilde{\varphi}'_1) + \mu_{k+1}(I) = 0$$ for all $I \in \mathbb{M}^n(S)$. So by Lemma \ref{action-kernel}, $j_k^0(\widetilde{H}) = j_k^0( \widetilde{H}')$.
\end{proof}

\begin{convention}
Since the action of an autonomous flow generated by a Hamiltonian $H$ on $\mathbb{M}^n(S)$ is determined by its $(n-1)$-jet $j_{n-1}^0(H)$, we shall have frequent cause when working with autonomous flows to specify Hamiltonians by their classes in $\mathcal{G}^n(S)$ by their $(n-1)$-jets. Even though the $(n-1)$-jets themselves, viewed as functions on $T^*S$, do not in general generate flows in {\rm{$\text{Ham}_c^0(T^*S)$}} as they are not compactly supported, it saves a great deal of repetitive language to make a minor conflation and talk about the ``flow of" a degree-$k$ Hamiltonian $H$. This should be taken to mean the flow of any compactly supported Hamiltonian agreeing with $H$ on a neighborhood of $Z^*S$.
\end{convention}

Our first application of the averaging procedure of Lemma \ref{partial-autonomy} is to show that every element of $\mathcal{N}^k(S)$ is realized by a $k$-stationary flow, which follows from a description of which $(k-1)$-stationary flows $\varphi_t$ have $k$-stationary time-$1$ flow $\varphi_1$.

\begin{corollary}\label{k-stationary-condition}
    For $k \geq 3$, a Hamiltonian diffeomorphism $\varphi$ generated by a $(k-1)$-stationary Hamiltonian flow $H_t$ that vanishes on $Z^*S$ is $k$-stationary if and only if the $(k-1)$-jet $j_{k-1}^0\left(\int_0^1 H_t dt\right)$ is $0$.
\end{corollary}


\begin{proof}
    Lemma \ref{partial-autonomy} shows that $\mu_k(I\varphi) = \mu_k(I\widetilde{\varphi})$ for all $I \in \mathbb{M}^n(S)$ where $\widetilde{\varphi}$ is the time-$1$ autonomous flow of $\widetilde{H} = \int_0^1 H_t dt$. Then Lemma \ref{action-kernel} applied to $\widetilde{H}$ gives the claim.
\end{proof}

\begin{lemma}\label{full-connectivity}
    For all $k\geq 2$, the $k$-stationary higher diffeomorphism group $\mathcal{N}^k(S)$ is connected. Every element of $\mathcal{N}^k(S)$ is represented by a $k$-stationary flow.
\end{lemma}

\begin{proof}
The proof is by induction. Lemma \ref{2-stationary-connected} is the $k=2$ case.

Now assume that every element of $\mathcal{N}^{k-1}(S)$ is represented by a $(k-1)$-stationary flow and let $h \in \mathcal{N}^k(S)$. Let $H_t$ generate a $(k-1)$-stationary flow $\varphi_t$ so that $[\varphi_1] = h$. Corollary \ref{k-stationary-condition} shows that the time-$1$ flows $\varphi_s$ of $H_t^s = sH_t$ are all $k$-stationary. Then $\{\varphi_s\}_{0\leq s\leq 1}$ is an isotopy from $\varphi_1$ to $\text{Id}$ consisting of $k$-stationary compactly supported Hamiltonian diffeomorphisms, and hence can be realized as a $k$-stationary flow (see e.g. \cite{mcduff2017introduction} ch. 10 or Appendix \ref{appendix-community-service}). 
\end{proof}

We have now accumulated enough information about $\mathcal{N}(S)$ to describe every element of $\mathcal{N}(S)$ as a product of well-controlled autonomous flows.

\begin{proposition}\label{auto-decomposition}
    Any $h \in \mathcal{N}(S)$ can be written uniquely as $h = h_{2} h_{3} ... h_{n-1}$ where $h_{k} \in \mathcal{N}^k(S)$ is the class in $\mathcal{G}^n(S)$ of an autonomous time-$1$ flow of a uniquely determined homogeneous degree $k$ Hamiltonian $H^{k}$. 
 \end{proposition}
 
 \begin{proof}
    Let $h \in \mathcal{N}(S)$ and let $\varphi_t$ be a 2-stationary generating flow. Lemma \ref{partial-autonomy} allows us to find a homogeneous degree-$2$ Hamiltonian $H^2$ whose autonomous time-$1$ flow $\psi$ satisfies that for any $I \in \mathbb{M}^n(S)$, $\mu_3(I\varphi)= \mu_3(I \psi)$. Then $\psi^{-1}\varphi$ is $3$-stationary. Putting $h_2 = [\psi]$, we have $h_2^{-1}h \in \mathcal{N}^3(S)$. Uniqueness in Lemma \ref{partial-autonomy} gives uniqueness.
    
    Proceeding analogously, using Lemma \ref{full-connectivity} to produce initial $k$-stationary flows, yields for each $2 \leq k \leq n-1$ a unique $h_k$ so that so that $h_{k}^{-1}...h_{3}^{-1} h_2^{-1} h$ is $k+1$-stationary. The case $k=n-1$ gives $h = h_2 h_3...h_{n-1}$, and uniqueness follows from uniqueness at each step.
\end{proof}

The bijection provided by Proposition \ref{auto-decomposition} between higher diffeomorphisms $h\in \mathcal{N}(S)$ and degree $n-1$ jets $H^2 + ... + H^{n-1}$ gives a parametrization of $\mathcal{N}(S)$ as $\mathcal{J}_{n-1}^2$. We call this parametrization \textit{inductive coordinates} for $\mathcal{N}(S)$. The terminology is due to both their iterated construction and their utility for inductive arguments.

We note that the product structure on $\mathcal{N}(S)$ in this parametrization and the product structure on $\mathcal{J}_{n-1}^2$ as an algebra do not agree. To distinguish the two appearances of jets and to emphasize that $h \in \mathcal{N}(S)$ is represented by a \textit{sequence} of autonomous flows in this parametrization, we denote the coordinates of $h \in \mathcal{N}(S)$ in inductive coordinates by $(H^2(h), ..., H^{n-1}(h))$, where $H^k(h) \in \Gamma(\mathcal{S}^k(TS))$. In inductive coordinates, $\mathcal{N}^k(S) = \{ (H^j)_{j=2}^{n-1} \mid H_j = 0 \text{ for } 2 \leq j \leq k-1\}$. 

Examination of the process that produces the decomposition in Proposition \ref{auto-decomposition} yields smoothness in an appropriate sense. For $K$ a compact submanifold (with boundary) of $T^*S$ containing the zero section in its interior, let $\text{Ham}_K^0(T^*S)$ and $\text{Symp}_K^0(T^*S)$ denote the subgroups of $\text{Ham}_c^0(T^*S)$ and $\text{Symp}_c^0(T^*S)$ consisting of elements supported on $K$.

\begin{proposition}\label{smoothness}
    Let $K\subset T^*S$ be a compact submanifold containing $Z^*S$ in its interior. The map {\rm{$\text{Ham}_K^0(T^*S) \mapsto \mathcal{J}_{n-1}^2$}} given by $\varphi \mapsto (H^2([\varphi]), ...,H^{n-1}([\varphi]))$ is smooth.
\end{proposition}

\begin{proof}
     We begin by showing the Hamiltonian function generating a flow $\psi_t \in \text{Ham}_K^0(T^*S)$ can be chosen smoothly in $\{ \psi_t \}_{0\leq t \leq 1}$. To see this, choose a Hamiltonian $G_t$ generating $\psi_t$ as follows. The flow $\psi_t$ is generated by a family of compactly supported vector fields $X_{t}$. That $\psi_t$ is a flow by symplectic diffeomorphisms is equivalent to the compactly supported forms $\lambda_t = \omega_{\text{can}}(X_t, \cdot)$ being closed. As $H_c^1(T^*S) = 0$, the forms $\lambda_t$ are exact and compactly supported potential functions can be chosen as $$G_{t}(p) = \int_{p_0}^p \lambda_t,$$ where $p_0$ is any point sufficiently far away from $K$. That $X_t$ is tangent to $Z^*S$ forces $G_t$ to be constant on $Z^*S$ (see Section \ref{2-complex}). This construction of $G_t$ is smooth in $\{ \psi_t \}_{0\leq t \leq 1}$.
    
    Subtracting appropriate constants from $G_t$ and multiplying by a cutoff function yields compactly supported Hamiltonians vanishing on $Z^*S$ and generating the same classes in $\mathcal{G}^n(S)$ as $[\psi_t]$, chosen smoothly in $\{\psi_t\}_{0\leq t\leq 1}$. Every $\varphi \in \text{Ham}_K^0(T^*S)$ is generated by a Hamiltonian flow supported in $K$ (see Appendix \ref{appendix-community-service}) and these flows can locally be chosen smoothly in $\varphi$. With this known, the proof is by checking smooth dependence in the steps taken by our construction of $H^2(\varphi),..., H^{n-1}(\varphi)$ from $\varphi$. We detail the steps below for clarity.
    
    We show that $H^k(\varphi)$ depends smoothly on a generating Hamiltonian $H_t$ for $\varphi$. Write our decomposition as $[\varphi] = h_2h_3...h_{n-1}$. As $H^k(\varphi)$ depends smoothly on a $k$-stationary flow, it suffices to show that given $H_t$ supported on $K$ with time-$1$ flow $\varphi$ so that $[\varphi] \in \mathcal{N}(S)$, a $2$-stationary Hamiltonian $H^2_t(\varphi)$ supported on $K$ generating $[\varphi]$ can be chosen smoothly in $H_t$ and that given a $k$-stationary Hamiltonian $H_t^k$ generating $h_k...h_{n-1}$ and supported on $K$, a $(k+1)$-stationary Hamiltonian $H^{k+1}_t$ generating $h_{k+1}...h_{n-1}$ and supported on $K$ can be chosen smoothly in $H^k_t$.
    
    Let $\varphi_t$ denote the flow of $H_t$. For $k =2$, let $\psi_t$ be a compactly supported Hamiltonian flow agreeing with $(\varphi_t|_{Z^*S})^\#$ on a neighborhood of $Z^*S$. Then a $2$-stationary isotopy from which $H^2_t$ is determined is given by $(\psi_t)^{-1} \varphi_t$.
    
    Now let $H^k_t$ be $k$-stationary and re-define $\varphi_t$ to be the flow of $H^k_t$. We have $[\varphi_1] = h_{k}...h_{n-1}$ in the notation of Proposition \ref{auto-decomposition}. The Hamiltonian ${H}^k$ is given by the $k$-jet $j_k^0\left( \int_0^1 H_t^k dt \right)$; denote the flow of $H^k$ by $\widetilde{\varphi}_t'$. Denote $\varphi_t' = (\widetilde{\varphi}_t')^{-1}$, and let $\psi_t$ be the smoothed concatenation of the isotopies $\varphi_t'$ and $\varphi_t$. Then $[\psi_1] = h_{k+1}...h_{n-1} \in \mathcal{N}^{k+1}(S)$, and $\psi_t$ determines a $k$-stationary Hamiltonian $G^{k+1}_t$ vanishing on $Z^*S$ smoothly in $\{\psi_t \}_{0\leq t\leq 1}$ as above. The isotopy $\{\psi_s'\}_{0\leq s \leq 1}$ where $\psi_s'$ is the time-$1$ flow of $sG_t^{k+1}$ is $(k+1)$-stationary by Corollary \ref{k-stationary-condition}, and $H^{k+1}_s$ is determined as above from $\psi_s'$. Each dependence is smooth.
\end{proof}

We now describe a central series for $\mathcal{N}(S)$. The key observation is that $\mathcal{N}^{n-1}(S)$ is in the center of $\mathcal{N}(S)$. We can see this as follows. Work in natural coordinates. If $H$ is an $(n-1)$-stationary Hamiltonian, then Lemma \ref{lemma-work-mod-I} and the first variation formula imply that the first variation of any $I \in \mathbb{M}^n(S)$ under the flow generated by $H$ depends only on $H$ and $\Sigma(I)$. More specifically, formula \eqref{basic-auto-formula} shows that the product operation on $\mathcal{N}^{n-1}(S)$ in inductive coordinates is addition and inversion is negation. 

On the other hand, if $H$ is a $k$-stationary Hamiltonian for some $k \geq 2$, then Lemma \ref{lemma-work-mod-I} and the first variation formula imply that the first variation of $\mu_{n}(I)$ under the autonomous flow of $H$ depends only on $H$ and $\Sigma(I), \mu_3(I),..., \mu_{n-k+1}(I)$. In particular, $\mu_n(I)$ has no effect on the first variation of $\mu_n(I)$ under the flow of $H$. It follows that $\mathcal{N}^{n-1}(S)$ is in the center of $\mathcal{N}(S)$.

Induction now shows that a central series of length $n-2$ for $\mathcal{N}(S)$ is given by $0 \lhd \mathcal{N}^{n-1}(S) \lhd ... \lhd \mathcal{N}^3(S) \lhd \mathcal{N}^2(S) = \mathcal{N}(S)$, so that $\mathcal{N}(S)$ is nilpotent. Proposition \ref{smoothness} implies that the group operations on $\mathcal{N}(S)$ in inductive coordinates are smooth.

We now address the structure of the whole of $\mathcal{G}^n(S)$. Let $\mathcal{D}_0(S)$ denote the subgroup of $\mathcal{G}^n(S)$ consisting of elements of the form $[\psi]$ with $\psi = f^\#$ on a neighborhood of the zero section for some $f \in \text{Diff}_0(S)$. As in Section \ref{2-complex}, every $\psi \in \text{Ham}_{c}^0(T^*S)$ determines an element of $\mathcal{D}_0(S)$ by restriction: $\varphi \mapsto [(\varphi|_{Z^*S})^\#]$. Here, we are implicitly using Lemma \ref{lifts-get-realized}. Similarly, let $\mathcal{D}^+(S)$ denote the subgroup of the $n$-standard diffeomorphism group $\mathcal{H}^{n,S}(S)$ consisting of $[f^\#]$ for $f \in \text{Diff}^+(S)$.

\begin{theorem}\label{structure-of-higher-diffeos}
    $\mathcal{G}^n(S) = \mathcal{D}_0(S) \ltimes \mathcal{N}(S)$ and $\mathcal{H}^{n,S}(S) = \mathcal{D}^+(S) \ltimes \mathcal{N}(S)$. For any compact submanifold $K \subset T^*S$ containing $Z^*S$ in its interior and smooth path {\rm{$\varphi_t \in \text{Ham}_K^0(T^*S)$}}, the map {\rm{$[0,1] \to \text{Diff}_0(S) \times \mathcal{J}^2_{n-1}$}} given by $t \mapsto (\varphi_t|_{Z^*S}, H^2((\varphi_t|_{Z^*S}^{-1})^\# \varphi_t), ..., H^{n-1}((\varphi_t|_{Z^*S}^{-1})^\#\varphi_t))$ is smooth.
\end{theorem}

Here, $H^k((\varphi_t|_{Z^*S}^{-1})^\# \varphi)$ should be taken to mean $H^k(\psi_t^{-1}\varphi_t)$ for any $\psi_t \in \text{Ham}_K^0(T^*S)$ agreeing with $\varphi_t$ on a neighborhood of $Z^*S$. Another remark is that in Theorem \ref{structure-of-higher-diffeos}, we have given $\mathcal{G}^n(S)$ the product topology of the subgroups $\mathcal{D}_0(S)$ and $\mathcal{N}(S)$, where $\mathcal{D}_0(S)$ is given the topology induced by the map $\phi \mapsto [\psi]$ where $\psi $ is any element of $\text{Ham}_c^0(T^*S)$ that agrees with $\phi^\#$ on a neighborhood of $Z^*S$, and $\mathcal{N}(S)$ is given the topology induced by inductive coordinates. The naturality of this choice of topology is addressed in Section \ref{frechet-lie-subsection}.

\begin{proof}
Lemma \ref{2-diffeos-just-diffeos} shows the intersection $\mathcal{D}_0(S) \cap \mathcal{N}(S) = \text{Id}$ is trivial as the action of $\text{Diff}_0(S)$ on $\mathbb{M}(S)$ is free. On the other hand, for any $g = [\varphi],$ we have the decomposition $g = [(\varphi|_{Z^*S})^\#] [(\varphi|_{Z^*S}^\#)^{-1} \varphi] \in \mathcal{D}_0(S) \mathcal{N}(S)$. As $\mathcal{N}(S)$ is normal in $\mathcal{G}^n(S)$, this shows that $\mathcal{G}^n(S)$ is the algebraic
semidirect product $\mathcal{D}_0(S) \ltimes \mathcal{N}(S)$. The proof that $\mathcal{H}^{n,S}(S) = \mathcal{D}^+(S) \ltimes \mathcal{N}(S)$ is nearly identical.

The map $t \mapsto (\varphi_t|_{Z^*S}, H^2((\varphi_t|_{Z^*S}^{-1})^\# \varphi_t), ..., H^{n-1}((\varphi_t|_{Z^*S}^{-1})^\#\varphi_t))$ is smooth due to Proposition \ref{smoothness} and the smoothness of $\varphi \mapsto \varphi|_{Z^*S}$. 
\end{proof}

Theorem \ref{thm-structure-summary}, which summarizes the basic structure of $\mathcal{G}^n(S)$, now follows from Theorem \ref{structure-of-higher-diffeos} and the preceding remarks.

\subsection{A Realization Lemma}

 There is one more way we need to control the lack of freeness of the action of $\mathcal{G}^n(S)$ on $\mathbb{M}^n(S)$ in our proofs of our main results. The phenomenon to understand is that an element $ h \in \mathcal{N}(S)$ can act on a specific $I \in \mathbb{M}^n(S)$ so that $\mu_j(Ih) = \mu_j(I)$ for all $j < k+1$, while in fact $h$ is not $k$-stationary (see Example \ref{example-not-free}). It is crucial to our proof of Theorem \ref{headline-result} to ensure that such examples do not provide examples of actions of $\mathcal{N}(S)$ on $n$-complex structures $I$ that are not already realized by an element of $\mathcal{N}^k(S)$. In this subsection, we prove this (Lemma \ref{non-collapse}).
 
 The proof of Lemma \ref{non-collapse} is organizationally simplified by representing elements of $\mathcal{N}^k(S)$ as a single autonomous flow rather than as a sequence of autonomous flows, which we verify is possible.

\begin{proposition}\label{exponential-coords-lemma}
    Any $h \in \mathcal{N}(S)$ can be written uniquely as the class of an autonomous time-$1$ flow of a degree $n-1$ Hamiltonian $H$ that has no degree-$1$ term and vanishes on $Z^*S$.
\end{proposition}

\begin{proof}
    Proposition \ref{auto-decomposition} yields a decomposition $h = h_2...h_{n-1}$ where $h_k$ is the class of a time-$1$ autonomous flow of a degree $k$ homogeneous Hamiltonian $H^{k}$ for $2\leq k\leq n-1$. The time-$1$ flow of $G^2 = H^2$ yields $\psi^2$ so $[\psi^2] = {h}_2$. The homogeneous Hamiltonian $H^2$ is uniquely determined by this requirement.
    
    Now suppose that ${G}^{k-1}$ is a Hamiltonian function of degree $k-1$ whose autonomous time-$1$ flow $\psi^{k-1}$ satisfies $\mu_j(I\psi^{k-1}) = \mu_j(Ih)$ for all $3\leq j \leq k$. This is of course equivalent to requiring that $(\psi^{k-1})^{-1} h \in \mathcal{N}^{k}(S)$. Then we see from Proposition \ref{auto-decomposition} that there is a unique homogeneous degree-$k$ Hamiltonian $\widetilde{H}^{k}$ with flow $\varphi_t^k$ so that $(\varphi^k_1)^{-1} (\psi^{k-1})^{-1} h \in \mathcal{N}^{k+1}(S)$, or equivalently that $\mu_j(I\psi^{k-1}_1 \varphi^k_1) = \mu_j(I h)$ for all $I \in \mathbb{M}^n(S)$ and $3\leq j \leq k+1$. Note that $\widetilde{H}^k$ may not be the same as $H^k(h)$, due to $\psi_t^{k-1}$ being generated by a single autonomous flow instead of a sequence of autonomous flows.
    
    Next, work in natural coordinates and let $I$ be given. Let $\psi_t^k$ be the flow of $G^k = {G}^{k-1} + \widetilde{H}^k$, and write $\widetilde{H}^{k} = w_{k}$ (mod $I$, degree $ \geq k+1$ terms). The first variation formula shows that for all $I, I' \in \mathbb{M}^n(S)$ with $\Sigma(I) = \Sigma(I')$, we have $\frac{d}{dt}\mu_{k+1}(I'\varphi_t^{k}) = \delbar w_{k}$, and Lemma \ref{lemma-work-mod-I} shows that $\frac{d}{dt}\mu_{k+1}(I'\psi_t^{k-1})$ depends only on $\Sigma(I'), \mu_3(I'), ..., \mu_k(I')$, and so is independent of $\mu_{k+1}(I')$.
    
    Now, for $j < k+1$, $\mu_j(I\psi^{k-1}_1) = \mu_j(I\psi_1^k)$ for all $I \in \mathbb{M}^n(S)$ as the $(k-1)$-jets of $G^k$ and $G^{k-1}$ agree. So from our tracking of dependencies of first variations, we see that for $0\leq t \leq 1$, $$\frac{d}{dt} \mu_{k+1}(I\psi_t^k) = \frac{d}{dt} \mu_{k+1}(I\psi_t^{k-1}) + \frac{d}{dt}\mu_{k+1}(I\varphi_t^k).$$ 
    
    So $\mu_{k+1}(I\psi_1^k) = \mu_{k+1}(I \psi_1^{k-1}) +  \delbar w_k$. On the other hand, $$\mu_{k+1}(Ih) = \mu_{k+1}(I\psi^{k-1}_1 \varphi^k_1) = \mu_{k+1}(I\psi_1^{k-1}) + \int_0^1 \frac{d}{dt} \mu_{k+1}(I\psi_1^{k-1} \varphi_t^k)dt = \mu_{k+1}(I\psi_1^{k-1}) + \delbar w_k.$$ We conclude that ${G}^k $ is a degree-$k$ Hamiltonian with autonomous time-$1$ flow $\psi^k_1$ satisfying $\mu_j(I\psi^k_1) = \mu_j(Ih)$ for all $3\leq j \leq k+1$. Repeating this process yields $H = {G}^{n-1}$ with time-$1$ flow $\psi$ so that $I\psi = Ih$ for all $I \in \mathbb{M}^n(S)$, and hence $[\psi] = h$. Uniqueness is given by uniqueness at each stage of our construction.
\end{proof}

 We denote by $H(h)$ the Hamiltonian associated to $h \in \mathcal{N}(S)$ by the previous proposition. We mention that these generating Hamiltonians can be used to give coordinates on $\mathcal{N}(S)$. We call these coordinates \textit{exponential coordinates}. We note that if $h = (H^2(h),...,H^{n-1}(h))$ in inductive coordinates and $h = H(h)$ in exponential coordinates, it may happen that $H (h) \neq \sum_{k=2}^{n-1} H^k(h)$, even though $H(h) \mapsto (H^2(h),..., H^{n-1}(h))$ and $(H^2(h),...,H^{n-1}(h)) \mapsto H(h)$ are smooth inverses for each other and $j_2^0(H(h)) = H^2(h)$. An analogous parametrization to Theorem \ref{structure-of-higher-diffeos} holds in exponential coordinates.
 
 In any matter, we shall only use the existence of $H(h)$ from Proposition \ref{exponential-coords-lemma} to prove our realization lemma, which we now proceed to.

\begin{lemma}[Realization Lemma]\label{non-collapse}
Suppose that $I, I' \in \mathbb{M}^n(S)$ and $Ih = I' $ for some $h \in \mathcal{N}(S)$. Let $k \geq 3$ be the least integer so that $\mu_k(I) \neq \mu_k(I')$ in natural coordinates. Then there exists $h' \in \mathcal{N}^{k-1}(S)$ so that $Ih' = I'$.
\end{lemma}

\begin{proof}
    The case $k = 3$ is vacuous, so we are free to assume $k \geq 4$. We shall show that if there is an $m$-stationary diffeomorphism $h_m \in \mathcal{N}^{m}(S)$ with $2\leq m < k-1$ so that $Ih_m = I'$, then there is an $(m+1)$-stationary diffeomorphism $h_{m+1}$ so that $Ih_{m+1} = I'$. Our hypothesis and the assumption that $k > 3$ guarantee the existence of such a diffeomorphism for $m =2$. So after showing this, induction gives the claim.
    
    So let $h_m \in \mathcal{N}^m(S)$ be a $m$-stationary diffeomorphism so $Ih_m = I'$. Write $\Sigma = \Sigma(I) = \Sigma(I')$. By Proposition \ref{exponential-coords-lemma}, we can find a $m$-stationary autonomous Hamiltonian $H(h_m)$ whose time-$1$ flow represents $h_m$. Let $\psi_t$ denote the flow of $H(h_m)$, and write $I\psi_t = I(t) = (\Sigma, \mu_3(t), ...,\mu_n(t))$. Our strategy to find $h_{m+1}$ is to examine $H(h_m) \pmod {I(t)}$, and to find an adjusted (time-dependent) Hamiltonian $G(t)$ that is equivalent to $H(h_m) \pmod {I(t)}$ for all times $t$ but only has terms of degree at least $m+1$, and show that $I(t)$ is the time-$t$ flow of $I$ under $G(t)$.

    Write $H(h_m)$ as \begin{align}\label{initial-bad-hamiltonian}H(h_m) =  H_0 + \sum_{j=1}^{m-1} w_{m-j,j} \equiv  H_0 + \sum_{l=m+1}^{n-1} P_l(t) \pmod {I(t)}\end{align} where $w_{m-j,j} \in \Gamma(K_\Sigma^{j-m}\overline{K}_\Sigma^{-j})$ and  $P_l(t) \in \Gamma(K_\Sigma^{-l})$. Here, $H_0 \in \Gamma \left(\bigoplus_{l={m+1}}^{n-1} \mathcal{S}^l(TM)\right)$ is real and has no terms of degree less than $m+1$. The lack of a $w_{m,0}$ term in formula \eqref{initial-bad-hamiltonian} is due to $H(h_m)$ being a Hamiltonian whose autonomous flow fixes $\mu_{m+1}(I)$, written in natural coordinates. The cause of the sum of $P_l(t)$ beginning at the index $l=m+1$ in \eqref{initial-bad-hamiltonian} is the lack of a $w_{m,0}$ term in $H(h_m)$.
    
    We now describe a procedure that yields the desired adjusted Hamiltonian. Define \begin{align*} G_1(t) &= H_0 + \sum_{l=m+1}^{n-1} (P_l(t) + \overline{P_l(t)}) \\ &\equiv H_0 + \sum_{l=m+1}^{n-1} P_l(t) + \sum_{l=m+2}^{n-1} Q_1^l(t) \pmod {I(t)}. \end{align*} Here, $H_0$ and $P_l(t)$ are the same sections as in formula \eqref{initial-bad-hamiltonian} and $Q_1^l(t) \in \Gamma(K^{-l}_\Sigma)$. Note the start of the sum of $\mathcal{Q}_1^l(t)$ at the index $l=m+2$.

    Now, assume that we are given $G_i(t)$ satisfying the following requirements. We require that that $G_i$ has the form $ G_i(t) = H_0 + R_i(t),$ 
    with $H_0$ the same section as above and so that $R_i(t)$ has the form $$R_i(t) = \sum_{l=m+1}^{n-1} (P_l(t) + \overline{P_l(t)}) + \sum_{l=m+2}^{m+i} (A_l(t) + \overline{A_l(t)}),$$ with $H_0$ and $P_l(t)$ the same sections as above and $A_l(t) \in \Gamma(K^{-l}_\Sigma)$. We also require that $G_i(t)\pmod{I(t)}$ admit an expression of the form
    \begin{align*}
    G_i(t) &\equiv H_0 + \sum_{l=m+1}^{n-1} P_l(t) + \sum_{l=m+i+1}^{n-1} Q_i^l(t) \pmod {I(t)},
    \end{align*} where $Q_i^l(t) \in \Gamma(K^{-l}_\Sigma).$ Note the $i$-dependence of the starting index of the sum of $Q_i^l(t)$. For $i=1$, $G_1(t)$ satisfies these requirements.
    
    Given such a $G_i(t)$, define \begin{align*}
        G_{i+1}(t) &= H_0 + R_i(t) - (Q_i^{m+i+1}(t) + \overline{Q_i^{m+i+1}(t)}). \end{align*}
    Then $G_{i+1}(t)$ also satisfies the same requirements that $G_i$ was specified to satisfy with $i+1$ in place of $i$.
    
    This construction is so that $G(t) = G_{n-m-1}(t)$ has the form
    \begin{align}
        G(t) = H_0 + \sum_{l=m+1}^{n-1} (P_l(t)  + \overline{P_l(t)} ) + \sum_{l=m+2}^{n-1} (A_l(t) + \overline{A_l(t)})
    \end{align} with the same notational requirements as above and satisfies 
    \begin{align}\label{non-collapse-verifying-requirement} G(t) &\equiv H_0 + \sum_{l=m+1}^{n-1} P_l(t) &\pmod{I(t)}  \\ &\equiv H(h_m) &\pmod{I(t)}.\label{non-collapse-cond-2} \end{align} Let $\varphi_t$ denote the time-$t$ flow of $G(t) = G_{n-m-1}(t)$. 
    
    We claim that $I \varphi_t = I(t)$ for $0 \leq t \leq 1$. That is, that the flow of $I$ under our adjusted Hamiltonian $G(t)$ agrees with the flow of $I$ under the original autonomous flow of $H(h_m)$. We shall prove by induction on $l$ that $\mu_l(I\varphi_t) = \mu_l(I(t))$ for all $3\leq l \leq n$ and $0 \leq t \leq 1$. Note that for $l \leq m$, we have $\mu_l(I\varphi_t) = \mu_lI(t)$ for $0 \leq t \leq 1$ as $h_m$ is $m$-stationary. Also, since $G_t$ is $(m+1)$-stationary by construction and $\mu_{m+1}(I(t))$ is constant, we must have $\mu_{m+1}(I\varphi_t) = \mu_{m+1}(I(t))$ for $0\leq t\leq 1$. 
    
    Now suppose that $l > m+1$ and $\mu_j(I\varphi_t) = \mu_j(I(t))$ for all $j < l$. We claim that $\frac{d}{dt} \mu_{l}(I\varphi_t) = \frac{d}{dt} \mu_l(I(t))$ for $0\leq t \leq 1$. By the first variation formula, it suffices to show that the $K^{1-j}_\Sigma$-terms of the normalized forms of $G(t) \pmod{I\varphi_t}$ and $H(h_m) \pmod {I(t)}$ agree for all $j \leq l$ and times $0\leq t\leq 1$. For any $j \leq l,$ the $K^{1-j}_\Sigma$ term of $G(t) \pmod{I\varphi_t}$ depends only upon $\mu_3(I\varphi_t),..., \mu_{l-1}(I\varphi_t)$, which agree with $\mu_3(I(t)), ..., \mu_{l-1}(I(t))$ by the induction hypothesis. Equations \eqref{non-collapse-verifying-requirement}-\eqref{non-collapse-cond-2} thus ensure that for $j \leq l$, the $K^{1-j}_\Sigma$-terms of the normalized forms of $G(t) \pmod{I\varphi_t}$ and $H(h_m) \pmod {I(t)}$ agree for all $t$. This proves that $\frac{d}{dt}\mu_l(I\varphi_t) = \frac{d}{dt}\mu_l(I(t))$ for all $0 \leq t \leq 1$, completing the induction.
    
    This shows that in particular that $I\varphi_1 = I(1) = Ih$. So $h_{m+1} = [\varphi_1]$ is the desired $(m+1)$-stationary diffeomorphism, completing the initial induction.
\end{proof}

\subsection{Other Features of Higher Diffeomorphism Groups}\label{frechet-lie-subsection}

At this point, we have accumulated everything used about the degree-$n$ diffeomorphism group $\mathcal{G}^n(S)$ in the coming sections. This subsection fills out the picture of $\mathcal{G}^n(S)$ a bit further, addressing the naturality of the topological group structure on $\mathcal{G}^n(S)$ given through inductive coordinates and first matters about $\mathcal{G}^n(S)$ as an infinite-dimensional Lie group.

The degree-$n$ diffeomorphism group $\mathcal{G}^n(S)$ obtains topologies through inductive coordinates and as a quotient group. These coincide, as is to be expected. The proof comes down to a verification that the map $\varphi \mapsto (\varphi|_{Z^*S}, H((\varphi|_{Z^*S}^{-1})^\# \varphi)$ is open.

\begin{proposition}\label{quotient-topology-is-product}
Let $\widetilde{\mathcal{G}}^n(S)$ denote the quotient of {\rm{$\text{Ham}_c^0(T^*S)$}} by the stabilizer of its action on $\mathbb{M}^n(S)$, topologized as the quotient. Then $\widetilde{\mathcal{G}}^n(S)$ is a topological group. Furthermore, map $\varphi \mapsto [\varphi]$ from {\rm{$\text{Ham}_c^0(T^*S)$}} to $ \mathcal{G}^n(S)$ with the topological group structure induced by Hamiltonian or exponential coordinates induces an isomorphism of topological groups $\widetilde{\mathcal{G}}^n(S) \to \mathcal{G}^n(S)$.
\end{proposition}

\begin{proof}
    The proposition follows from the induced map $\Phi: \widetilde{\mathcal{G}}^n(S) \to \mathcal{G}^n(S)$ being a homeomorphism. The map $\Phi$ is a bijective homomorphism, and is continuous by Proposition \ref{smoothness}. So it remains to verify that $\Phi$ is open. It suffices to show that $q: \text{Ham}_c^0(T^*S) \to \mathcal{G}^n(S)$ given by $q: \varphi \mapsto [\varphi]$ is open.
    
    Let $U$ be a nonempty neighborhood in $\text{Ham}_c^0(T^*S)$. There must exist a compact submanifold $K$ of $T^*S$ containing $Z^*S$ in its interior and $\varphi \in \text{Ham}_K^0(T^*S) \cap U$. As $U$ is open in $\text{Ham}_c^0(T^*S)$, there must be a neighborhood $U_K$ of $\varphi$ in $\text{Ham}_K^0(T^*S)$ contained in the intersection $U \cap \text{Ham}_K^0(T^*S)$. Then $W_K = \varphi^{-1}U_K$ is a neighborhood of the identity in $\text{Ham}_K^0(T^*S)$. Take a nonempty neighborhood $V_K \subset W_K$ so $V_K^2 \subset W_K$.
    
    The neighborhood $V_K$ contains autonomous flows whose generating Hamiltonians' $(n-1)$-jets along $Z^*S$ contain a neighborhood of $0$ in $\mathcal{J}_{n-1}^2$, along with elements agreeing with $\varphi^\#$ on a neighborhood of $Z^*S$ for every $\varphi$ in some neighborhood of the identity in $\text{Diff}_0(S)$. So $q(W_K)$ contains a neighborhood of the identity in $\mathcal{G}^n(S),$ and so $q(U) \supset q(U_K) = q(\varphi) q(W_K)$, which is an open neighborhood of $q(\varphi)$ as $\mathcal{G}^n(S)$ is a topological group.
\end{proof}

We next construct a regular Fr\'echet-Lie group structure on $\mathcal{G}^n(S)$ equipped with $C^\infty$ projections $q_k: \mathcal{G}^n(S) \to \mathcal{G}^k(S)$ (see \cite{omori1996infinite} for basic definitions and properties). We do this through a straightforward application of a foundational proposition on extensions of regular Fr\'echet-Lie groups:

\begin{proposition}[\cite{kobayashi1983regular}\label{awful-regular-FL-condition}, Proposition 5.2]
    Consider an exact sequence of groups $$\begin{tikzcd} 1 \arrow[r]&  N  \arrow[r]&  G   \arrow[r, "\pi"]& G/N  \arrow[r] & 1 \end{tikzcd}$$  with $N, G/N$ regular Fr\'echet-Lie groups. Suppose that there is a neighborhood $U$ of $e \in G/N$ and a map $\gamma: U \to G$ so that:
    \begin{enumerate}
        \item The map $\gamma$ is a cross-section in the sense that $\gamma(e_{G/N}) = e_G$, $\pi \circ \gamma = \text{Id}_U$, and $(g, n) \mapsto \gamma(g)n$ is a bijection from $U \times N \to \pi^{-1}(U)$,
        \item For some neighborhood $V \subset U$ containing $e \in G/N$ so that $V^2 \subset U$, $V = V^{-1}$, and $\pi^{-1}(U)$ generates $G$, the map $r_\gamma : V \times V \to N$ defined by $$r_\gamma(g, h) = \gamma(gh)^{-1} \gamma(g) \gamma(h) $$ is $C^\infty$,
        \item The conjugation map $\alpha_\gamma: V \times N \to N$ given by $\alpha_\gamma(g,m) = \gamma(g)^{-1} m \gamma(g)$ is $C^\infty$.
    \end{enumerate}
     Then there is a natural regular Fr\'echet-Lie group structure on $G$ so that $\pi: G \to G/N$ is $C^\infty$.
\end{proposition}

Where the degree of the higher diffeomorphism group of interest is ambiguous in the following, we adopt the notation that $\mathcal{N}^{k,j}(S)$ is the subgroup $\mathcal{N}^k(S) \subset \mathcal{G}^j(S)$.

We begin by assembling a regular Fr\'echet-Lie structure on $\mathcal{N}(S)$. The groups $\mathcal{N}^{k, k+1}(S)$ have regular Fr\'echet-Lie structures as additive groups of Fr\'echet spaces, and the quotient $\mathcal{N}^{k,n}(S)/\mathcal{N}^{k+1,n}(S)$ has a regular Fr\'echet-Lie structure inherited from its identification with $\mathcal{N}^{k,k+1}(S)$. So begin by considering the extension $\mathcal{N}^{n-1,n}(S) \to \mathcal{N}^{n-2,n}(S) \to \mathcal{N}^{n-2,n-1}(S)$. Define a cross-section in inductive coordinates by $\gamma: (H^{n-2}(\varphi)) \mapsto (H^{n-2}(\varphi),0)$ for $U = \mathcal{N}^{n-2,n-1}(S)$. Condition $(1)$ is immediate, and smoothness of group operations in inductive coordinates gives (2) and (3) for $V = U= \mathcal{N}^{n-2,n-1}(S)$. Proceeding inductively with the same argument produces a regular Fr\'echet-Lie group structure on $\mathcal{N}(S)$.

To extend this regular Fr\'echet-Lie structure to all of $\mathcal{G}^n(S)$, we use the extension $\mathcal{N}(S) \to \mathcal{G}^n(S) \to\text{Diff}_0(S)$ defining the algebraic semidirect product structure of Theorem \ref{structure-of-higher-diffeos}. As $S$ is compact, $\text{Diff}_0(S)$ is a regular Fr\'echet-Lie group (\cite{omori1996infinite}, for instance). Here, we take the cross-section $\gamma : \varphi \mapsto (\varphi, 0,...,0)$ in inductive coordinates for $U = \text{Diff}_0(S)$. Then (1) is clear and (2) is satisfied as $r_\gamma$ is the constant map $(g,h) \mapsto (0,...,0)$. Condition (3) holds due to Proposition \ref{smoothness} and smoothness of conjugation in $\text{Ham}_K^0(T^*S)$ for $K$ a compact submanifold of $T^*S$ containing the zero section $Z^*S$ in its interior. Summarizing, we have shown:

\begin{proposition}\label{prop-frechet-lie}
    The degree-$n$ diffeomorphism group $\mathcal{G}^n(S)$ with the topology induced from inductive coordinates admits the structure of a regular Fr\'echet-Lie group so that for $k<n$ the projections $\mathcal{G}^n(S) \to \mathcal{G}^k(S)$ are smooth.
\end{proposition}

\begin{remark}
    The group {\rm{$\text{Diff}_0(S)$}} has more structure than that of a regular Fr\'echet-Lie group: it has structures as inverse limits of Hilbert- and Banach-Lie [ILH, ILB] groups in the sense of Omori \cite{ebin1970manifold}, \cite{omori1970group}. A natural question is if ILB-Lie group structures on the symplectic diffeomorphisms supported on fixed compact sets give rise to isomorphisms in the category of ILB-Lie groups between $\widetilde{\mathcal{G}}^n(S)$ and an appropriate ILB-Lie group structure on $\mathcal{G}^n(S)$.
\end{remark}

\section{Harmonic Higher Complex Structures}\label{quotient-section}

We have now developed enough machinery to study $\mathcal{T}^n(S)$ as the quotient $\mathbb{M}^n(S)/\mathcal{G}^n(S)$. Our approach is to work in natural coordinates and show that every orbit of the action of $\mathcal{N}(S)$ on $\mathbb{M}^n(S)$ has a unique representative in $\mathcal{H}\mathbb{M}^n(S)$. As the action of $\text{Diff}_0(S)$ on $\mathbb{M}^n(S)$ in natural coordinates is by pullback, this identification of $\mathbb{M}^n(S)/\mathcal{N}(S)$ and $\mathcal{H}\mathbb{M}^n(S)$ descends to a vector bundle structure on $\mathcal{T}^n(S)$. The main result is Theorem \ref{main-bundle-thm}, which is an analogue of Labourie's conjecture for $\mathcal{T}^n(S)$. Throughout this section, the genus of $S$ is taken to be at least $2$.

We shall make use of a Hodge theorem for $k$-Beltrami differentials on $\Sigma \in \mathbb{M}(S)$ with respect to the Petersson $L^2$ product $\langle \mu_k, \nu_k \rangle = \int_{\Sigma} \mu_k \overline{\nu}_k g_\Sigma^{k-1}$. The statement for $2$-Beltrami differentials is quite well-known and widely applied. The proof for all $k \geq 2$ follows standard arguments, and is included for completeness.

\begin{lemma}\label{hodge} Let $\Sigma \in \mathbb{M}(S)$. If $\mu_k$ is a smooth $k$-Beltrami differential on $\Sigma$ for some $k \geq 2$, then $\mu_k$ can be expressed uniquely as $$\mu_k = \widetilde{\mu}_k + \delbar w_{k-1}$$ with $\widetilde{\mu}_k$ harmonic and $w_{k-1} \in \Gamma(K^{1-k}_\Sigma)$ smooth. Both $\widetilde{\mu}_k$ and $\delbar w_{k-1}$ depend smoothly and linearly on $\mu_k$.
\end{lemma}

In fact, the map $P: \mu_k \mapsto \widetilde{\mu}_k$ is the $L^2$ orthogonal projection to the space of harmonic Beltrami differentials on $\Sigma$.

\begin{proof}
	We begin by enlarging our space of sections. Denote the space of $L^2$ sections of $\overline{K}_\Sigma K_\Sigma^{1-k}$ by $\Gamma^2(\overline{K}_\Sigma K_\Sigma^{1-k})$. Let $V$ denote the space of sections of $\overline{K}_\Sigma K_\Sigma^{1-k}$ of Sobolev class $W^{1,2}$, define $E = \delbar (V)$ and let $\mathcal{H}$ be the orthogonal complement of $E$ under the $L^2$ inner product. We obtain $\Gamma^2(\overline{K}_\Sigma K_\Sigma ^{1-k}) = \overline{E} \oplus  \mathcal{H}$. Now $\delbar$ is an elliptic operator on a compact manifold, and so must be Fredholm and hence have closed range in $\Gamma^2(\overline{K}_\Sigma K_\Sigma^{1-k})$. Thus $E = \overline{E}$.
	
	Now, if $\mu_k$ is an element of $\mathcal{H}$, we have
	$$ 0 = \langle \delbar w_{k-1}, \mu_k  \rangle = \int(\delbar w_{k-1}) \overline{\mu}_k g_\Sigma^{k-1} $$ for all smooth $w_{k-1}$. Integrating by parts shows $\delbar (\overline{\mu}_k g_\Sigma^{k-1}) = 0$ weakly. Applying Weyl's lemma in coordinate patches shows $\overline{\mu}_k g_\Sigma^{k-1}$ is smooth. So $\mu_k$ is smooth and harmonic. This shows $\mathcal{H}$ consists of exactly the harmonic $k$-Beltrami differentials.
	
	Now, let $\mu_k = \widetilde{\mu}_k + \alpha$ be the orthogonal decomposition of $\mu_k$ as above. As $\alpha \in E$, we must have that $\alpha = \delbar w_{k-1}$ weakly for some $w_{k-1} \in V$. As $\alpha = \mu_k - \widetilde{\mu}_k$ is the difference of smooth tensors, it is smooth. As $\alpha = \delbar w_{k-1}$, by Weyl's lemma $w_{k-1}$ is smooth. This yields existence of the desired decomposition. Uniqueness and smooth dependence of $\widetilde{ \mu}_k$ and $\alpha$ on $\mu_k$ follow from being obtained via an orthogonal decomposition. Finally, $w_{k-1}$ is unique as if $\delbar w_{k-1}' = \delbar w_{k-1}$, then $\omega = w_{k-1}' - w_{k-1}$ is a smooth section of $K_\Sigma^{1-k}$ with $\delbar \omega = 0$ and so must be identically $0$ as the genus of $S$ is at least $2$. 
\end{proof}

We now produce harmonic representatives of $\mathcal{N}(S)$-orbits. This is the place where all the machinery we have developed comes together.

\begin{theorem}\label{harmonic-representatives}
    Let $I = (\Sigma, \mu_3, ..., \mu_n) \in \mathbb{M}^n(S)$. Then $I$ is equivalent modulo $\mathcal{G}^n(S)$ to a unique $\widetilde{I} \in \mathcal{H}\mathbb{M}^n(S)$ with $\Sigma(\widetilde{I}) = \Sigma(I)$. The dependence of $\widetilde{I}$ on $I$ is smooth.
\end{theorem}

\begin{proof}
We address the main difficulty first: uniqueness. Suppose for contradiction that $I$ is equivalent under $\mathcal{G}^n(S)$ to two distinct harmonic $n$-complex structures $I_1$ and $I_2$ with $\Sigma(I_1) = \Sigma(I_2) = \Sigma$. Then there is some $g \in \mathcal{G}^n(S)$ so $I_2 = I_1g$. Using Theorem \ref{structure-of-higher-diffeos}, write $g = [f^\#] h$ with $f \in \text{Diff}_0(S)$ and $h \in \mathcal{N}(S)$. As the action of $\text{Diff}_0(S)$ on $\mathbb{M}(S)$ is free, the hypothesis that $\Sigma(I_1) = \Sigma(I_2)$ forces $f = \text{Id}$. So there is some $h \in \mathcal{N}(S)$ so $I_2 = I_1 h$.

Now let $k \geq 3$ be the smallest $k$ so that $\mu_k(I_1) \neq \mu_k(I_2)$. Lemma \ref{non-collapse} shows that there is a $(k-1)$-stationary $h' \in \mathcal{N}^{k-1}(S)$ so that $I_1 h' = I_2$. There is then by Proposition \ref{exponential-coords-lemma} (or Proposition \ref{auto-decomposition}) a $(k-1)$-stationary Hamiltonian $H$ whose autonomous time-$1$ flow $\varphi'$ satisfies $\mu_k(I_2) = \mu_{k}(I_1h') = \mu_k(I_1\varphi')$. Writing $H \equiv w_{k-1} \pmod {I, \text{degree $\geq k$ terms}}$ with $w_{k-1} \in \Gamma(K^{1-k}_\Sigma)$, formula \eqref{basic-auto-formula} shows that $\mu_k(I_2) = \mu_k(I_1) + \delbar w_{k-1}$ with $\delbar w_{k-1} \neq 0$. This contradicts uniqueness in Lemma \ref{hodge}.

We next address existence. Lemma \ref{hodge} produces the existence of a decomposition $\mu_3 = \widetilde{\mu}_3 + \delbar w_{2}$. Take the autonomous degree-$2$ Hamiltonian $H = - w_{2} - \overline{w_{2}}$. Then $H \equiv -w_2 \pmod {I', \text{ degree $\geq 3$ terms}}$ for any $I'$ with $\Sigma(I') = \Sigma(I)$. With $h_2 \in \mathcal{N}_2$ the time-$1$ flow of $H$, we have $\mu_3(I h_2) = \widetilde{\mu}_3$. Replace $I$ with $Ih_2$. Continuing in this matter produces $h \in \mathcal{N}(S)$ so that $\mu_k(Ih)$ is harmonic for $3 \leq k \leq n$. The construction in the proof of existence of $\widetilde{I}$ depends smoothly on $I$, giving smoothness.
\end{proof}

\begin{remark} The main use of our structural results about $\mathcal{N}(S)$ in the preceding proof is to rule out the existence of actions of elements of $\mathcal{N}(S)$ not seen infinitesimally or not seen through the simplest examples. This reduces the analysis of the action of $\mathcal{N}(S)$ to an infinitesimal analysis. If it were known beforehand that $\mathcal{T}^n(S)$ is a manifold of dimension {\rm{$-\chi(S) \text{dim}(\text{PSL}(n,\bbR))$}} (c.f. \cite{fock2021higher}, Theorem 2), this can be used in place of representation of elements of $\mathcal{N}^k(S)$ by autonomous flows and the Realization Lemma in the proof above. We outline the alternative method here.

The place where a different argument is needed is in proving uniqueness of harmonic representatives of $\mathcal{N}(S)$-orbits. If there are two distinct harmonic $n$-complex structures $I$ and $I'$ in a $\mathcal{N}(S)$-orbit, then a $2$-stationary flow $\varphi_t$ connects them. Taking harmonic representatives of $I \varphi_t$ following the method of the proof of existence in Theorem \ref{harmonic-representatives} yields a nontrivial path of harmonic $n$-complex structures in the $\mathcal{N}(S)$ orbit of $I$. Uniqueness then follows from the observation that together with tangent vectors to autonomous flows, tangent vectors to this path would produce an unallowable upper bound on the dimension of $\mathcal{T}^n(S)$. \end{remark}

Theorem $\ref{harmonic-representatives}$ shows that $\mathbb{M}^n(S)/\mathcal{N(S)}$ is canonically identified with $\mathcal{H}\mathbb{M}^n(S)$. By Corollary \ref{harmonic-natural}, the induced action of $\text{Diff}_0(S) \cong \mathcal{G}^n(S)/\mathcal{N}(S)$ on $\mathcal{H}\mathbb{M}^n(S)$ is given by pullback. Recall the notation from the introduction that $\mathcal{B}^n(S) = \mathcal{H}\mathbb{M}^n(S)/\text{Diff}_0(S)$. Let $q_n: \mathcal{B}^n(S) \to \mathcal{T}(S)$ be the projection $[(\Sigma, \nu_3, ..., \nu_n) ] \mapsto [\Sigma]$.

    \begin{theorem}\label{main-bundle-thm}
    The degree-$n$ Fock-Thomas space  $(\mathcal{T}^n(S), \pi_2)$ equipped with the natural projection $\pi_2$ is identified with $(\mathcal{B}^n(S), q_n)$ as a vector bundle over $\mathcal{T}(S)$. This identification is {\rm{$\text{Mod}(S)$}}-equivariant.
    \end{theorem}
    
    \begin{proof}
        To work with the $\text{Mod}(S)$ action, we used marked coordinates and the description $\mathcal{T}^n(S) = \dot{\mathbb{M}}^n(S)/\mathcal{H}^{n,S}(S)$. The identification $\Phi^n: (\mathcal{T}^n(S), \pi_2) \to (\mathcal{B}^n(S), q_n)$ is given by $[(I,\phi)] \mapsto [(\Sigma(I), \nu_3(I), ..., \nu_n(I),\phi)]$, where $\nu_k(I)$ ($3\leq k\leq n$) are the harmonic $k$-Beltrami differentials representing the unique harmonic $n$-complex structure in the $\mathcal{N}(S)$-equivalence class of $I$.
        
        To see that this map is well-defined, let $(\Sigma, \mu_3, ..., \mu_n, \phi) \in \dot{\mathbb{M}}^n(S)$ have harmonic $\mathcal{N}(S)$-orbit representative $I = (\Sigma, \nu_3,..., \nu_n,\phi)$, and suppose that $(\Sigma, \mu_3,...,\mu_n, \phi)g = (\Sigma', \mu_3',...,\mu_n',\phi')$ with $g = [f^\#]h \in \mathcal{H}^{n,S}(S)$, and that the $\mathcal{N}(S)$ orbit of $(\Sigma', \mu_3',...\mu_n', \phi')$ has harmonic representative $(\Sigma', \nu_3',...,\nu_n',\phi')$. Write $(\Sigma, \mu_3, ..., \mu_n, \phi) = (\Sigma, \nu_3, ..., \nu_n,\phi)h'$ for some $h' \in \mathcal{N}(S)$. Then \begin{align*}
            (\Sigma', \mu_3', ..., \mu_n', \phi') &= (\Sigma, \nu_3, ..., \nu_n, \phi)h' [f^\#] h \\
            &=  (f^* \Sigma, f^*\nu_3, ..., f^*\nu_n, \phi \circ f) \widetilde{h}
        \end{align*} for some $\widetilde{h} \in \mathcal{N}(S)$ by normality of $\mathcal{N}(S)$ and Proposition \ref{diff-lift-in-natural}.
        
        So $(\Sigma', \mu_3',..., \mu_n',\phi')$ has harmonic $\mathcal{N}(S)$-orbit representative $(f^*\Sigma, f^*\mu_3, ..., f^*\nu_n, \phi\circ f)$. This forces $(\Sigma', \nu_3',...,\nu_n',\phi') = (f^*\Sigma, f^*\nu_3, ..., f^*\nu_n, \phi \circ f )$, showing well-definition.

        The map $\Phi^n$ is a bijection by Theorem \ref{harmonic-representatives}. Note that $\pi_2 = q_n \circ \Phi^n$, so that the pulled-back bundle structure on $\mathcal{T}^n(S)$ from $(\mathcal{B}^n(S), q_n)$ under $\Phi^n$ is given by $(\mathcal{T}^n(S), \pi_2)$.
        
        The final thing to verify is $\text{Mod}(S)$-equivariance. For this, let $ [f] \in \text{Mod}(S)$. For any $[(I,\phi)] \in \mathcal{T}^n(S)$, take a representative $(I, \phi)$ of $[(I, \phi)]$ and let $\widetilde{I} = (\Sigma, \nu_3,...,\nu_n,\phi)$ be the harmonic $\mathcal{N}(S)$-orbit representative of $I$. Then \begin{align*}
            \Phi^n([(I,\phi)][f]) &= \Phi^n([(\widetilde{I}, \phi)][f]) = \Phi^n([(\widetilde{I}, f^{-1} \circ \phi)]) =  [(\Sigma, \nu_3,...,\nu_n,f^{-1} \circ \phi )] \\ &= [(\Sigma, \nu_3,...,\nu_n, \phi)][f] = \Phi^n([(I,\phi)])[f].
        \end{align*}
    \end{proof}

We obtain as a corollary:

\begin{theorem}\label{projection-structure}
     The degree-$n$ Fock-Thomas space $(\mathcal{T}^n(S), \pi_k)$ equipped with the natural projection $\pi_k$ is a vector bundle over $\mathcal{T}^k(S)$ with fiber over $[I] \in \mathcal{T}^k(S)$ identified with the space of classes of tuples $[(\mu_{k+1},...,\mu_n)]$ with $\mu_j$ a harmonic $j$-Beltrami differential on a representative of $\pi_2(I)$ for $j=k+1,...,n$. The projections $\pi_k$ are {\rm{$\text{Mod}(S)$}}-equivariant.
\end{theorem}

\begin{proof}
 If $(\Sigma, \nu_3,..., \nu_n,\phi) = (\Sigma,\mu_3, ..., \mu_n,\phi)h$ is the harmonic representative of the $\mathcal{N}(S)$ orbit of $I = (\Sigma, \mu_3,...,\mu_n,\phi)$, then writing $h = h_k h'$ with $h' \in \mathcal{N}^k(S)$, for any $\varphi \in \text{Symp}_T^{0,S}(T^*S)$ so $[\varphi] = h$ in $\mathcal{H}^{n,S}(S)$, we have $(\Sigma, \mu_3, ..., \mu_k,\phi)[\varphi] = (\Sigma, \mu_3,..., \mu_k,\phi)h_k = (\Sigma, \nu_3, ..., \nu_k,\phi)$. So for any $n$-complex structure $I$ with harmonic $\mathcal{N}(S)$-orbit representative $\widetilde{I}$, the harmonic representative of $p_{n,k}(I)$ in the $\mathcal{N}(S) \subset \mathcal{G}^k(S)$ orbit is $p_{n,k}(\widetilde{I})$.
 
 Define for $2\leq k \leq n$ maps $q_{n,k}: \mathcal{B}^n(S) \to \mathcal{B}^k(S)$ by $[(\Sigma,\nu_3,..., \nu_n,\phi)] \mapsto [(\Sigma, \nu_3 ,..., \nu_k,\phi)]$. Following the notation of the proof of Theorem \ref{main-bundle-thm}, the identifications $\Phi^n: (\mathcal{T}^n(S), \pi_2) \to (\mathcal{B}^n(S), q_n)$ satisfy that $\pi_k =(\Phi^{k})^{-1} \circ q_{n,k} \circ \Phi^n$. As $\Phi^n$ and $\Phi^k$ are $\text{Mod}(S)$-equivariant vector bundle isomorphisms and $q_{n,k}$ is a surjective $\text{Mod}(S)$-equivariant bundle morphism, this gives the claim. 
\end{proof}

We remark that Theorem \ref{main-bundle-thm} endows $\mathcal{T}^n(S)$ with a manifold structure and Theorem \ref{projection-structure} confirms a strengthening of a conjecture of Thomas (\cite{thomas2020thesis}, Conjecture 18.8). 

The harmonic representatives in Theorem \ref{main-bundle-thm} can be interpreted as minima of the norm of $\mu_{k+1}(I\varphi)$ among $\mathcal{N}^k(S)$ orbits. Note that as the space of coordinates $\mu_{k+1}(I\varphi)$ under a $\mathcal{N}^{k-1}(S)$ orbit is affine, any critical point of $||\mu_{k+1}(Ih)||^2$ in a $\mathcal{N}(S)$-orbit must be a minimum.

	\begin{proposition}\label{euler-lagrange}
		For $\Sigma \in \mathbb{M}^n(S)$ and $k\geq 3$, the critical points of $||\mu_k||^2 = \int_{\Sigma} |\mu_k|^2 g^{k-1}_\Sigma $ under variations by $(k-1)$-stationary Hamiltonian flows are harmonic Beltrami differentials.
	\end{proposition}
	
	\begin{proof}
		Let $I \in \mathbb{M}^n(S)$ and write $\mu_k = \mu_k(I)$. Let $H$ be a $(k-1)$-stationary Hamiltonian with $H \equiv w_{k-1} + ... + w_{n-1} \pmod I$. The first variation formula tells us that $\frac{d}{dt} \mu_k = \delbar w_{k-1},$ hence under the flow generated by $H$ the $k$-Beltrami differential transforms as $\mu_k(t) = \mu_k + t \delbar w_2 + \delta(t)$ with $\delta = O(t^2)$. So $\mu_k$ being a critical point of $||\cdot ||^2$ in a $\mathcal{N}(S)$-orbit amounts to having for all $w_{k-1} \in \Gamma(K^{1-k}_\Sigma)$
		\begin{align*}
		0 &= \frac{d}{dt}\bigg|_{t=0} \int_{\Sigma} ||\mu_k + t \delbar w_{k-1} + \delta(t) ||^2 \\
		&=  \int_{\Sigma} \frac{d}{dt}\bigg|_{t=0}  \left(\mu_k  + t\delbar w_{k-1}+ \delta(t) \right) \left(\overline{\mu}_{k}  +t\partial \overline{w}_{k-1} + \overline{\delta(t)}  \right) g_\Sigma^{k-1} \\
		&=  \int_{\Sigma} (\overline{\mu}_k \delbar{w_{k-1}} + \mu_{k} \del \overline{w}_{k-1}) g_\Sigma^{k-1} \\
		&=-  \int_{\Sigma} \delbar (\overline \mu_k g_\Sigma^{k-1}) w_{k-1} + \del (\mu_k g_\Sigma^{k-1}) \overline{w}_{k-1} \\
		&= - \int_\Sigma \del (\mu_k g_\Sigma^{k-1}) \overline{w}_{k-1} + \overline{(\del(\mu_k g_\Sigma^{k-1}) \overline{w}_{k-1} )} \\
		&= - 2\displaystyle{\int_\Sigma} \text{Re}(\del(\mu_k g_\Sigma^{k-1}) \overline{w}_{k-1}) 
		\end{align*}
		If $\mu_k$ is harmonic, then this holds by inspection. For the other implication, if $w_{k-1}$ is taken to be supported on a uniformized coordinate patch $(U, z)$ and real in this chart, we have 
		\begin{align*}- \int_\Sigma \text{Re}(\del (\mu_k g_\Sigma^{k-1} )) \overline{w}_{k-1} &=- \int_U \text{Re}(\del (\mu_k g_\Sigma^{k-1} )) \overline{w}_{k-1} = 0,  \\  
		- \int_\Sigma \text{Re}(\del (\mu_k g_\Sigma^{k-1} ) i\overline{w}_{k-1})  &= \int_U \text{Im}(\del (\mu_k g_\Sigma^{k-1})) \overline{w}_{k-1}  = 0 .\end{align*} And since this holds for all $w_{k-1}$ supported in small coordinate balls and real in appropriate charts, we must have $\delbar(\overline{\mu}_kg_\Sigma^{k-1}) = 0$. So $\mu_k$ is harmonic. 
 	\end{proof}

\section{Relation to Hitchin Components}\label{section-formal-hitchin-link}

We now use the results of Section \ref{quotient-section} to connect $\mathcal{T}^n(S)$ and $\text{Hit}(S, \text{PSL}(n, \bbR))$ by constructing the maps of Theorem \ref{headline-result}. We also explain the results giving rise to Theorem \ref{free-from-labourie} and Propositions \ref{free-from-labourie-a}-\ref{free-from-labourie-b}. The proof of Theorem \ref{headline-result} given here is efficient in a way that obscures some of the underlying geometry of the diffeomorphisms between $\mathcal{T}^n(S)$ and $\text{Hit}(S, \text{SL}(3, \mathbb{R}))$, which we explain in Section \ref{section-3-complex} by explicitly describing how a Hitchin representation is obtained from a $3$-complex structure.

As before, every $\Sigma \in \mathbb{M}(S)$ has a unique hyperbolic metric in its conformal class, which we denote by $g_\Sigma$. The Petersson pairing on differentials induced by the hyperbolic metrics $g_\Sigma$ produces a $\text{Mod}(S)$-equivariant identification $\Phi': {\mathcal{B}}^n(S) \to {\mathcal{Q}}^n(S)$, where $\mathcal{Q}^n(S) = \mathcal{Q}(3, S) \oplus ... \oplus \mathcal{Q}(n,S)$ is the bundle over $\mathcal{T}(S)$ whose fiber over $\Sigma$ consists of classes of tuples $(q_3,...,q_n)$ with $q_k$ a holomorphic section of $K^k_\Sigma$ for $3\leq k \leq n$. Denote the identification ${\mathcal{T}}^n(S) \to \mathcal{Q}^n(S)$ arising from $\Phi'$ and Theorem \ref{main-bundle-thm} by $\Phi$.

Hitchin gives a description of the now-called Hitchin component $\text{Hit}(S, \text{PSL}(n,\bbR))$ of $\text{Rep}(S, \text{PSL}(n,\bbR))$ using Higgs bundles techniques in \cite{hitchin1992lie}. The input data for Hitchin's construction are a fixed Riemann surface $\Sigma$ and holomorphic differentials $(q_k)_{2\leq k \leq n}$ with $q_k \in \Gamma(K^k_\Sigma)$. The holonomy $\rho_{\Sigma, q_2, ..., q_n}$ of a solution to the self-duality equations for an associated Higgs bundle is associated to these invariants; for any fixed $\Sigma$, these $\rho_{\Sigma, q_2, ..., q_n}$ parametrize $\text{Hit}(S, \text{PSL}(n,\bbR))$. Hitchin shows in \cite{hitchin1987self} that such representations with $q_3 = ... = q_n = 0$ parametrize the Fuchsian representations, generalizing work of Wolf \cite{wolf1989teichmuller}.

In \cite{labourie2004anosov}, Labourie considered the map $\Psi: \mathcal{Q}^n(S) \to \text{Hit}(S,\text{PSL}(n,\bbR))$ given by $$(\Sigma, q_3, ..., q_n) \mapsto \rho_{\Sigma, 0, q_3, ..., q_n},$$ and conjectured that $\Psi$ was a diffeomorphism. This conjecture was known to be true for $n=2$ from \cite{wolf1989teichmuller} or \cite{hitchin1987self} and for $n=3$ by independent work of Labourie and Loftin \cite{labourie2007flat}, \cite{loftin2001affine}. It is true for general rank $2$ Hitchin components \cite{labourie2017cyclic}, but the analogous conjecture is false for $\text{PSL}(2, \bbR) \times \text{PSL}(2, \bbR) \times \text{PSL}(2, \bbR)$ \cite{markovic2022non}. In general $\Psi$ is surjective \cite{labourie2008cross}. 

The map $\Phi$ gives an identification of ${\mathcal{T}}^n(S)$ with the domain of Labourie's map depending on no choices, and through this a relationship with Hitchin components can be established. The general case is infinitesimal in flavor. The case for $n=3$ is also global and admits a concrete geometric description (see Section \ref{section-3-complex}).

Let $\pi_{Q}$ denote the projection $\mathcal{Q}^n(S) \to \mathcal{T}(S)$, and let $\mathcal{Q}^n_0$ be the sub-bundle of the tangent bundle to $\mathcal{Q}^n(S)$ along the zero section given by $\text{ker}(\pi_Q)$. The bundle $\mathcal{Q}_0^n$ is identified with $\mathcal{Q}^n$ by elementary means.

\begin{proof}[Proof of Theorem \ref{headline-result}]
    Labourie's map is a $\text{Mod}(S)$-equivariant diffeomorphism for $n=3$, so $\Psi \circ \Phi$ is a $\text{Mod}(S)$-equivariant diffeomorphism ${\mathcal{T}}^3(S) \to \text{Hit}(S, \text{SL}(3, \bbR)).$
    
    Let $n \geq 3$. As Hitchin's parametrization of $\text{Hit}(S, \text{PSL}(n,\bbR))$ is a diffeomorphism mapping tuples $(q_2, 0,...,0)$ to the Fuchsian locus, $$D\Psi_{\Sigma}: T_\Sigma \mathcal{T}(S) \oplus \mathcal{Q}(3,\Sigma) \oplus ... \oplus \mathcal{Q}(n, \Sigma) \to T_{\rho_{\Sigma, 0, ...,0}} \text{Hit}(S, \text{PSL}(n,\bbR))$$ has full rank and $D\Psi_\Sigma((\mathcal{Q}^n_0)_\Sigma)$ has trivial intersection with the tangent space to the Fuchsian locus at $\Sigma$. So $(D\Psi)|_{\mathcal{Q}_0^n}$ is a bundle isomorphism onto its image. Identifying ${\mathcal{T}}^n(S)$ with $\mathcal{Q}^n(S)$ then $\mathcal{Q}_0^n(S)$, then applying $D\Psi$ gives the result. 
\end{proof}

Some results obtained on $\mathcal{B}^n(S)$ and the Labourie map $\Psi$ give rise to structure on $\mathcal{T}^n(S)$, which is documented by Theorem \ref{free-from-labourie} and Propositions \ref{free-from-labourie-a}-\ref{free-from-labourie-b}. We explain the relevant results here.

Let $\mathbf{m} =(m_1,...,m_p)$ be a collection of integers each at least $2$. Let $\mathcal{B}^\mathbf{m}(S)$ be the vector bundle over $\mathcal{T}(S)$ whose fiber over $\Sigma \in \mathcal{T}(S)$ is $\mathcal{B}(m_1,\Sigma) \oplus ... \oplus \mathcal{B}(m_p,\Sigma)$. In our notation $\mathcal{B}^n(S) = \mathcal{B}^\mathbf{m}(S)$ with $\mathbf{m} = (3,4,...,n)$. Labourie constructed families of K\"ahler metrics on the bundles $\mathcal{B}^\mathbf{m}(S)$ \cite{labourie2017cyclic}, following the construction of Kim-Zhang  \cite{kim2017kahler} in the $\mathbf{m} = (3)$ case.

\begin{proposition}[\cite{labourie2017cyclic}, Propositions 1.3.1, 9.0.1]\label{kahler-metrics-existence} Let $\mathbf{m} = (m_1, ..., m_p)$ be a tuple of integers at least $2$. Then $\mathcal{B}^\mathbf{m}(S)$ admits a real $p$-dimensional family $\mathcal{K}^\mathbf{m}$ of K\"ahler metrics that are {\rm{$\text{Mod}(S)$}}-invariant and restrict to a multiple of the $L^2$ pairing on each $\mathcal{B}(m_j,\Sigma)$-subspace of the fibers.

For every $h \in \mathcal{K}^\mathbf{m}$, the zero section of $\mathcal{B}^{\mathbf{m}}(S)$ is a totally geodesic submanifold and $h$ restricted to the zero section is the Weil-Petersson metric on $\mathcal{T}(S)$. For every $h \in \mathcal{K}^\mathbf{m}$, for distinct indicies $m_{i_1}$ and $m_{i_2}$ the subspaces $\mathcal{B}(m_{i_1}, \Sigma)$ and $\mathcal{B}(m_{i_2}, \Sigma)$ to tangent spaces of the fibers are orthogonal.
\end{proposition}

The case $\mathbf{m} = (3, 4,...,n)$ yields Theorem \ref{free-from-labourie} except for the real-analyticity of the K\"ahler metrics there, which is proved in Appendix \ref{appendix-kahler-metrics-analytic}.

The construction of K\"ahler metrics on $\text{Hit}(S, G)$ for $G$ of rank $2$ in \cite{labourie2017cyclic} is accomplished by using the Petersson pairing on fibers to (non-holomorphically) identify $\mathcal{B}^{n}(S)$ with $\mathcal{Q}^n(S)$, then applying appropriate Labourie maps $\Psi$. The K\"ahler metrics and compatible complex structures on $\text{Hit}(S,G)$ are then obtained by pullback. The construction of the diffeomorphisms of Theorem \ref{headline-result} then immediately implies Proposition \ref{free-from-labourie-a}.

Finally, in \cite{labourie2018variations}, Labourie and Wentworth show that for any holomorphic $k$-adic differential $q_k$ on $\Sigma \in \mathcal{T}(S)$, the pressure metric $P$ at $\rho_{\Sigma, 0,...,0}$ can be evaluated on pairs of differentials to be $$P(D \Psi_\Sigma(q_k), D\Psi_\Sigma(q_k)) = C(k) \int_\Sigma \frac{ q_k \overline{q}_k}{g_\Sigma^{k-1}}, $$ where $C(k)$ is an explicit constant depending only on the tensor type of $q_k$ and the underlying topological surface $S$. It is also shown in \cite{labourie2018variations} that for $k \neq j$, any holomorphic differentials $q_k \in \Gamma(K^k_\Sigma)$ and $q_j \in \Gamma(K^j_\Sigma)$ are orthogonal with respect to the pressure metric: $P(D\Psi_\Sigma(q_k), D\Psi_\Sigma(q_j)) = 0$. This yields Proposition \ref{free-from-labourie-b}.

\section{3-Complex Structures and Affine Spheres}\label{section-3-complex}

We now give an explicit and geometric description of the Hitchin representation associated to $[I] \in \mathcal{T}^3(S)$. Following Labourie's parametrization \cite{labourie2007flat} (see also \cite{loftin2001affine}) of $\text{Hit}(S, \text{SL}(3,\bbR))$, we explain a geometric interpretation of the map used to identify $\mathcal{T}^3(S)$ and $\text{Hit}(S, \text{SL}(3, \bbR))$. We begin with a condensed overview of relevant facts about affine spheres in $\mathbb{R}^3$. An introductory text on affine surfaces that contains the basic theory we cover is \cite{nomizu1994affine}, and \cite{loftin2010survey} is a survey on affine spheres.

The study of affine surfaces is a rich field concerned with the structure induced on surfaces $S$ from immersions $f: S \to \mathbb{R}^3$ that is invariant under the special affine linear group $\text{SAff}(\mathbb{R}^3)$ generated by $\text{SL}(3, \bbR)$ and translations of $\mathbb{R}^3$.

Denote by $\overline{\nabla }$ the standard connection on $\mathbb{R}^3$ and the standard volume form on $\mathbb{R}^3$ by $\lambda$. From an immersion $f: S \to \mathbb{R}^3$ and transverse vector field $\xi$ one can produce a torsion-free connection $\nabla$, symmetric $(0,2)$-tensor $h$, a $(1,1)$-tensor $Sh$, a $1$-form $\tau$, and a nowhere vanishing $2$-form $\theta$ on $S$ by defining \begin{align}
    \overline{\nabla}_X Df(Y) &= Df(\nabla_XY) + h(X,Y)\xi, \label{affine-connection-metric} \\
    \overline \nabla_X \xi &= - Df(Sh(X)) + \tau(X)\xi, \label{affine-shape} \\
    \theta(X,Y) &= \lambda(Df (X),Df( Y), \xi).
\end{align}

Here, $\nabla$ is called the \textit{affine connection}, $h$ the \textit{affine fundamental form}, $Sh$ the \textit{affine shape operator}, $\tau$ the \textit{transverse connection form}, and $\theta$ the \textit{induced volume form}. These depend upon the choice of transverse vector field $\xi$.

In the following we restrict to strictly convex immersions $f$ so that $h$ is definite. In this setting, there is a unique transversal vector field $\xi$, called the \textit{affine normal}, specified by the requirements that $h$ be a Riemannian metric, $\nabla \theta = 0$, and $\theta = d\text{Vol}_h$. The condition $\nabla \theta = 0$ is equivalent to $\tau = 0$. The affine normal is equivariant under $\text{SAff}(\mathbb{R}^3)$, and the corresponding affine shape operator, fundamental form, and connection specified by the affine normal are all well-defined objects on $S$.

The tuple $(\nabla, h, Sh)$ obtained from a strictly convex immersion $f: S \to \mathbb{R}^3$ is called a \textit{Blaschke structure} on $S$, the surface $f(S)$ is called a \textit{Blaschke surface}, and $f$ is called a \textit{Blaschke immersion}. The metric $h$ given by the affine second fundamental form is called the \textit{Blaschke metric} and $\nabla$ the \textit{Blaschke connection}. Standard arguments adapted from Riemannian geometry show that the condition that $(\nabla, h,Sh)$ arise from an immersion $f: S \to \mathbb{R}^3$ produce the following constraints:
\begin{align}
    R^\nabla(Z, Y)Z &= h(Y, Z) Sh(X) - h(X,Z) Sh(Y),\label{first-integrability} \\
    (\nabla_X h)(Y, Z) &= (\nabla_Y h)(X, Z), \\
    (\nabla_X Sh)Y &= (\nabla_Y Sh)X , \\
    h(X, Sh(Y)) &= h(Sh(X), Y).\label{last-integrability}
\end{align}

The restrictions we have found so far on Blaschke structures provide all the necessary integrability conditions to develop an affine surface, in the following sense.

\begin{theorem}[\cite{nomizu1994affine}, Theorem 8.1]\label{sphere-development}
    Let $M$ be a simply connected oriented surface, $\nabla$ a torsion-free connection on $M$, $h$ a Riemannian metric on $M$, and $Sh$ a $(1,1)$-tensor on $M$. Suppose that $(\nabla, h, Sh)$ satisfy \eqref{first-integrability}-\eqref{last-integrability} and that $\nabla d\text{Vol}_h = 0$. Then there exists a Blaschke immersion $f: M \to \mathbb{R}^3$ so that the Blaschke structure induced on $M$ is given by $(\nabla, h, Sh)$. The Blaschke immersion $f$ is unique up to the action of {\rm{$\text{SAff}(\bbR^3)$}} on the target.
    \end{theorem}

    A Blaschke surface is called an \textit{affine sphere} if $Sh = k\text{Id}$ for some $k \in \mathbb{R}$. An affine sphere is \textit{proper} if $k \neq 0$ and \textit{improper} if $k = 0$. A Blaschke surface is a proper affine sphere if and only if all of its affine normals meet at one point. Such a point is called the \textit{center} of the affine sphere. Our interest is in hyperbolic affine spheres, which are the affine spheres with $k = -1$.
    
    For a hyperbolic affine sphere $S$ with Blaschke data $(\nabla, h, -\text{Id})$, consider the difference tensor $A = \nabla^h - \nabla$, where $\nabla^h$ is the Levi-Civita connection of $h$. A computation shows that there is a unique holomorphic cubic differential $\varphi$ on $S$ with respect to the conformal structure of $h$ so that the cubic tensor $C$ defined by $C(X,Y,Z) = h(A(X)Y,Z)$ is given by $C(X,Y,Z) = \text{Re}(\varphi(X,Y,Z))$. The tensor $\varphi$ is called the \textit{Pick differential} of $f$.
    
    On the other hand, let $S$ be a closed surface of genus at least $2$ and let $\varphi$ be a holomorphic cubic differential on $\Sigma \in \mathbb{M}(S)$. If $h$ is a metric in the conformal class of $\Sigma$, let $A_\varphi^h$ denote the cubic tensor defined by $h(A_\varphi^h(X)Y, Z) = \text{Re}(\varphi(X,Y,Z)).$ A torsion-free connection can be built from $h$ and $A_\varphi^h$ by $\nabla^{h,\varphi} = \nabla^h + A_\varphi^h$. Lifting $h$ and $\nabla^{h, \varphi}$ to the universal cover of $S$ produces a metric $\widetilde{h}$ and a torsion-free connection $\widetilde{\nabla}^{h, \varphi}$ on the disk $\mathbb{D}$ that are both invariant under the action of covering transformations.
    
    Through reducing the integrability conditions of Theorem \ref{sphere-development} to the Wang equation \cite{changping1991some} and following analysis analogous to Wang, Labourie and Loftin showed the following.
    
    \begin{theorem}[\cite{changping1991some}, \cite{labourie2004anosov}, \cite{loftin2001affine}]\label{differentials-and-spheres}
        Let $\Sigma \in \mathbb{M}(S)$ and let $\varphi$ be a holomorphic cubic differential on $\Sigma$. In the conformal class of $\Sigma$, there is a unique metric $h$ so that {\rm{$(\widetilde{\nabla}^{h, \varphi}, \widetilde{h}, -\text{Id})$}} is a Blaschke structure on $\mathbb{D}$.
    \end{theorem}
    
    From an affine sphere $f$ with Blaschke data $(\widetilde{\nabla}^{h, \varphi}, \widetilde{h}, -\text{Id})$, a representation $\rho: \pi_1(S) \to \text{SL}(3, \mathbb{R})$ can be produced as follows. We are free to require $f$ have its center be located at $0$. Any $\gamma \in \pi_1(S)$ acts on $\mathbb{D}$ by a covering transformation whose induced action on $\widetilde{\nabla}^{h, \varphi}$ and $\widetilde{h}$ leaves both fixed. So the map $z \mapsto f(\gamma z)$ is also an affine sphere with Blaschke data $(\widetilde{\nabla}^{h, \varphi}, \widetilde{h}, -\text{Id})$. The uniqueness assertion of Theorem \ref{sphere-development} shows that there is some $A_\gamma \in \text{SAff}(\mathbb{R}^3)$ so that $f (\gamma z) = A_\gamma(f(z))$ for all $z \in \mathbb{D}$. The center of $A_\gamma \circ f$ is $0$ as the image of $f$ is invariant under precomposition by covering transformations, so $A_\gamma \in \text{SL}(3, \bbR)$. The map $\rho: \gamma \to A_\gamma$ produces a representation $\pi_1(S) \to \text{SL}(3, \mathbb{R})$ so that $f$ is $\rho$-equivariant.
    
    By considering the lift of the affine sphere $f$ to a $\rho$-equivariant minimal surface in $\text{SL}(3, \mathbb{R})/\text{SO}(3)$, Labourie shows in \cite{labourie2007flat} that the representations constructed through affine spheres and through Hitchin's parametrization have a simple relationship.
    
    \begin{theorem}[\cite{labourie2007flat}]
        The representation $\rho$ associated to $(\Sigma, \varphi)$ via Theorem \ref{differentials-and-spheres} is given in Hitchin's parametrization based at $\Sigma$ by $\rho_{0, \varphi/12}$.
    \end{theorem}

    In light of the preceding, the Hitchin representation $\rho \in \text{Hit}(S, \text{SL}(3, \mathbb{R}))$ produced from $[I] \in \mathcal{T}^3(S)$ in the proof of Theorem \ref{headline-result} is as follows.

    From a $3$-complex structure $I$ representing $[I]$, the linear order terms in generators distinguish a complex structure $\Sigma$ as in Section \ref{section-natural-coords}. Taking local normalized representatives modulo $I$ whose degree-$2$ terms are purely of type $(-2, 0)$ gives rise to a $(-2,1)$-tensor $\mu_3$ on $\Sigma$. The pairing on tensors induced by the hyperbolic metric $g_\Sigma$ in the conformal class of $\Sigma$ associates to $\mu_3$ a type $(3,0)$ tensor $\varphi_I = \overline{\mu}_3 g_\Sigma^{2}/12$.
    
    For any metric $h$ in the conformal class of $\Sigma$, a connection $\nabla^{h,\varphi_I}$ can be formed as above. As affine spheres have holomorphic Pick differential, for almost all $3$-complex structures $I$ there are no metrics $h$ in the conformal class of $\Sigma(I)$ so that $(\widetilde{\nabla}^{\varphi_I, h}, \widetilde{h}, -\text{Id})$ is a Blaschke structure. 
    
    So we seek $3$-complex structures equivalent to $I$ under higher diffeomorphisms that do develop to affine spheres. Theorem \ref{main-bundle-thm} shows there is a unique higher complex structure $\widetilde{I}$ in the $\mathcal{N}(S)$-orbit of $I$ so that there is a metric $h$ in the conformal class of $\Sigma$ so $({\widetilde{\nabla}}^{{\varphi_{\widetilde{I}}}, h}, \widetilde{h}, -\text{Id})$ is a Blaschke structure on $\mathbb{D}$. Such a metric $h$ is unique in the conformal class of $\Sigma$.
    
    Let $f$ be an affine sphere associated to $({\widetilde{\nabla}}^{{\varphi_{\widetilde{I}}}, h}, \widetilde{h}, -\text{Id})$ with center $0$. The Hitchin representation $\rho$ associated to $I$ is so that $\rho(\gamma)f(z) = f(\gamma z)$ for all $z \in \mathbb{D}$.

\appendix

\section{Real-Analyticity of Some K\"ahler Metrics}\label{appendix-kahler-metrics-analytic}

Our assertion in Theorem \ref{free-from-labourie} that the K\"ahler metrics on $\mathcal{T}^n(S)$ considered there are real-analytic K\"ahler relies on the fact that the K\"ahler metrics constructed in \cite{labourie2017cyclic} and \cite{kim2017kahler} on bundles over $\mathcal{T}(S)$ of differentials are real-analytic. We prove this here. A corollary of this is that the K\"ahler metrics on Hitchin components of rank $2$ constructed by Labourie and Kim-Zhang are real-analytic K\"ahler. Throughout the following, let $S$ be a closed, oriented surface of genus $g \geq 2$ and fixed topological type.

Let $\mathbf{m} = (m_1,...,m_p)$ be a tuple of integers all at least $2$. Denote by $\mathcal{B}^\mathbf{m}(S)$ the bundle over $\mathcal{T}(S)$ with fiber over $\Sigma$ the classes $\mathcal{B}(m_1, \Sigma) \oplus ... \oplus \mathcal{B}(m_p, \Sigma)$ where $\mathcal{B}(k, \Sigma)$ is the space of classes of $k$-Beltrami differentials on $\Sigma
$. Let $\mathcal{Q}^\mathbf{m}(S)$ be the bundle over $\mathcal{T}(S)$ whose fiber over $\Sigma$ is the space of classes of differentials $\mathcal{Q}(m_1, \Sigma) \oplus ... \oplus \mathcal{Q}(m_p, \Sigma)$ where $\mathcal{Q}(k,\Sigma)$ is the space of classes of holomorphic $k$-adic differentials on $\Sigma$. Let $\mathcal{B}_k(S)$ and $\mathcal{Q}_k(S)$ denote the sub-bundles of $\mathcal{B}^\mathbf{m}(S)$ and $\mathcal{Q}^\mathbf{m}(S)$ corresponding to a single index of $\mathbf{m}$, and consisting of differentials of a fixed degree $k$. The integration pairing $(q,\mu) \mapsto \int_\Sigma q \mu $ identifies $\mathcal{B}_k(S)$ with the dual of $\mathcal{Q}_k(S)$, and $\mathcal{B}^\mathbf{m}(S)$ so inherits a holomorphic vector bundle structure. 

There are pairings on $\mathcal{Q}^\mathbf{m}(S)$ and $\mathcal{B}^\mathbf{m}(S)$ specified by \begin{align*} h_k(q_1, q_2) &= \int_{\Sigma} \frac{q_1 \overline{q}_2}{g_\Sigma^{k-1}} \qquad \qquad (q_1, q_2 \in \mathcal{Q}(k,\Sigma) ), \\ h'_k(\mu_1, \mu_2) & = \int_\Sigma \mu_1 \overline{\mu}_2 g_\Sigma^{k-1}  \qquad (\mu_1, \mu_2 \in \mathcal{B}(k,\Sigma)), \end{align*} and the requirement that differentials corresponding to different elements $m_p$ of the tuple $\mathbf{m}$ are orthogonal. Here, $g_\Sigma$ is the unique hyperbolic metric in the conformal class of $\Sigma$. Let $g_{\text{WP}}$ be the Weil-Petersson metric on $\mathcal{T}(S)$. The metrics defined by Labourie in \cite{labourie2017cyclic} are given by $$\epsilon \del \delbar h + \pi^* g_{\text{WP}}$$ where $\epsilon > 0$, $\delbar$ and $\del$ are the Dobeault operators on $\mathcal{B}^\mathbf{m}(S)$, the bilinear form $h$ on fibers is a linear combination with positive real coefficients of the pairings $h_{m_p}'$, and $\pi$ is the holomorphic projection $\mathcal{B}^\mathbf{m}(S) \to \mathcal{T}(S)$. So our goal follows from:

\begin{lemma}\label{real-analytic-lemma}
    The pairings $h_k'$ are real-analytic on $\mathcal{B}_k(S)$.
\end{lemma}

This statement is shown in \cite{wolpert2017equiboundedness} for $h_2'$. The coordinate system on $\mathcal{T}(S)$ that seems best adapted for this question is the Bers embedding \cite{bers1961correction}, which represents $\mathcal{T}(S)$ as a bounded domain in $\mathbb{C}^{3g-3}$. We review the Bers embedding and the construction of holomorphic frames for $\mathcal{Q}_n(S)$ in the approach of Bers \cite{bers1961holomorphic}, since Lemma \ref{real-analytic-lemma} follows from their description.

Fix a base marked Riemann surface $\Sigma$ and uniformizing Fuchsian group $G$ for $\Sigma$. The Riemann sphere is divided into the disk $\mathbb{D}$ and the complementary region. To any Beltrami differential $\mu$ on $\mathbb{D}$, an extension of $\mu$ to $\mathbb{C}$ is made by setting $\mu = 0$ off $\mathbb{D}$. A unique quasiconformal map $f^{\mu} : \mathbb{S}^2 \to \mathbb{S}^2$ exists that solves the Beltrami equation $f^\mu_{\overline{z}} = \mu f^{\mu}_z $ and is normalized so $f^{\mu}(0) = 0, f^\mu(1) = 1,$ and $f^\mu(\infty) = \infty$. The map $f^\mu$ is then conformal on the complement of the closed disk, and computing the Schwarzian derivative $\phi^\mu$ of $f^\mu$ on the region $|z| > 1$ yields an embedding of $\mathcal{T}(S)$ into the space of holomorphic quadratic differentials on $\overline{\Sigma}$. This is known as the Bers embedding.

When working with the Bers embedding, we use $\tau$ to represent an element of $\mathcal{T}(S)$ and $\Sigma_\tau$ a corresponding marked Riemann surface. Bers shows (\cite{bers1961holomorphic} Theorem I, also see \cite{bers1960simultaneous}) that M\"obius transformations $A_i(\tau, z), B_i(\tau,z)$ ($1\leq i \leq g$) that give a standard generating list for quasi-Fuchsian groups $G(\tau)$ uniformizing $\Sigma_\tau$ can be chosen holomorphically in $\tau$. Here, a \textit{standard generating list} is a tuple $(A_1, ..., A_g, B_1, ..., B_g)$ generating $\pi_1(S)$ and satisfying the single relation $\prod_{j=1}^g[A_j,B_j] = \text{Id}.$ Bounded domains $D(\tau)$ of discontinuity for these representations can also be chosen analytically in $\tau$, so that the space of pairs $\hat{\mathcal{T}}(S) = \{(\tau, z) \mid \tau \in \mathcal{T}(S), z \in D(\tau) \}$ is a bounded domain in $\mathbb{C}^{3g-2}$.

Let $W_k(\tau)$ be the space of holomorphic functions $\varphi(z)$ on $D(\tau)$ so that $\varphi(z) dz^k$ is a holomorphic $k$-adic differential on $D(\tau)$ invariant under $G(\tau)$, and $W_k$ the space of holomorphic functions $\varphi(\tau,z)$ on $\hat{\mathcal{T}}(S)$ so that $\varphi(\tau,\cdot) \in W_k(\tau)$ for all $\tau \in \mathcal{T}(S)$. A remarkable theorem of Bers \cite{bers1961holomorphic} that produces local holomorphic frames of $\mathcal{Q}_k(S)$ is that every element of $W_k(\tau)$ is the restriction of an element of $W_k$.

\begin{proof}[Proof of Lemma \ref{real-analytic-lemma}]
Let $\mu$ be a local holomorphic coordinate on $\mathcal{T}(S)$ given by harmonic Beltrami differentials on $\Sigma$ (see e.g. \cite{ahlfors1961some}). Then in this set-up, $D(\mu) = f^\mu(\mathbb{D})$ and for $A \in G$, the corresponding element $A^\mu$ of $G(\mu)$ is defined by $A^\mu \circ f^{\mu} = f^{\mu} \circ A$. Let $R \subset \mathbb{D}$ be a fundamental domain for the action of $G$ on $\mathbb{D}$. One sees that $R_\mu = f^{\mu}(R)$ is a fundamental domain for the action of $G(\mu)$ on $D(\mu)$.

After possibly restricting the coordinate neighborhood, let $Q_1(\mu, z),..., Q_{N}(\mu,z)\in W_k$ give a local holomorphic frame for $\mathcal{Q}_k(S)$. That is, for any fixed $\mu$ in our coordinate chart, $\{Q_i(\mu, z) dz^k \}$ is a basis for $\mathcal{Q}(k,\mu)$. Then as the integration pairing induces the holomorphic structure on $\mathcal{B}_k(S)$, we have a dual holomorphic frame for $\mathcal{B}_k(S)$ given by $$B_i(\mu,z) = \left( \int_{R_\mu} \frac{Q_i(\mu, z) \overline{Q_i(\mu, z)}}{g_\mu^{k-1}} \right)^{-1} \frac{\overline{Q_i(\mu, z)}}{g_\mu^{k-1}}$$ where $g_\mu$ is the hyperbolic metric for $\Sigma_\mu$ pulled back to $D(\mu)$.

Let $\mu = \mu(t)$ be real-analytic in $t$. By pulling back integration to $R$ via $f^{\mu}$, we have for all $i, j$ that with $f^{\mu*}$ the pullback on appropriate tensors, \begin{align*}
    h_k'(B_i(\mu,z),& B_j(\mu,z)) \\ &= \left(\int_R \frac{|f^{\mu *} Q_i(\mu,z)|^2}{f^{\mu*} g_\mu^{k-1}}  \right)^{-1} \left( \int_R \frac{|f^{\mu*} Q_j(\mu,z)|^2 }{f^{\mu*} g_\mu^{k-1}} \right)^{-1} \int_R \frac{f^{\mu *} Q_i(\mu,z) \overline{f^{\mu*} Q_j(\mu,z) }}{f^{\mu*} g_\mu^{k-1}}.
\end{align*}
Each $Q_k(\mu,z)$ varies analytically in $(\mu,z)$ by construction. Work of Ahlfors and Bers \cite{ahlfors1960riemann} shows that $f^{\mu}$ varies real-analytically in an appropriate Banach space, and work of Wolpert \cite{wolpert1990bers} shows that $f^{\mu*}g_\mu$ varies real-analytically in $\mu$. The result follows.
\end{proof}

\section{Groups of Hamiltonian and Symplectic Diffeomorphisms}\label{appendix-community-service}

In this appendix, we describe some features of the full degree-$n$ diffeomorphism group $\mathcal{H}^n(S)$ as a topological space, and investigate the relationship between $\mathcal{H}^n(S)$ and $\mathcal{G}^n(S)$. The result of broader interest is that $\text{HMod}^n(T^*S)$ is discrete. The following also ensures that the basic definitions used for higher degree diffeomorphism groups are reasonable, and answers some foundational questions on the framework used to define higher degree complex structures.

For an instance of a potential pathology addressed, $\text{Ham}_c^0(T^*S)$ was defined in terms of the existence of a Hamiltonian flow that fixes the zero section setwise. If there were Hamiltonian diffeomorphisms arbitrarily close to the identity that setwise fix $Z^*S$ but are only generated by flows that do not always fix $Z^*S$, this definition would be unnatural and somewhat pathological, but at the same time essential for effective use of the first variation formula. It is a matter of good housekeeping to ensure this does not occur (Propositions \ref{appendix-reduction}, \ref{hams-open}). The main result is:

\begin{theorem}\label{things-not-totally-awful} Let $\mathcal{G}^n(S)$ be the degree-$n$ diffeomorphism group and $\mathcal{H}^n(S)$ the quotient of {\rm{$\text{Symp}_T^0(T^*S)$}} by the kernel of its action on $\dot{\mathbb{M}}^n(S)$. Then:
\begin{enumerate}[label=(\alph*).] \item The full degree-$n$ diffeomorphism group $\mathcal{H}^n(S)$ is a topological group with the topology it inherits as the quotient of {\rm{$\text{Symp}_{T}^0(T^*S)$}} by the stabilizer of its action on $\dot{\mathbb{M}}^n(S)$,
\item The quotient {\rm{$\text{HMod}^n(S) = \mathcal{H}^n(S)/\mathcal{G}^n(S)$}} is discrete,
\item The identity component of $\mathcal{H}^n(S)$ is open and equal to $\mathcal{G}^n(S)$.
\end{enumerate}
\end{theorem}

Along the way, we mention some relevant facts about $\text{Ham}_c(T^*S)$ and $\text{Symp}_c(T^*S)$ we make use of throughout the paper. Many of the basic facts we address about these groups can be found in \cite{mcduff2017introduction}, especially Chapter $10$. 

In the following, we shall repeatedly have cause to consider compact submanifolds $K \subset T^*S$ containing $Z^*S$ in their interior. We shall call such a submanifold $K$ an \textit{admissible submanifold} of $T^*S$. For admissible $K$, denote by $\text{Ham}_K^0(T^*S)$ the group of Hamiltonian diffeomorphisms of $T^*S$ generated by Hamiltonian flows setwise fixing the zero section and supported on $K$, and $\text{Symp}_K^0(T^*S)$ the group of symplectomorphisms supported on $K$ and setwise fixing $Z^*S$. Our first goal is the following.

\begin{proposition}\label{appendix-reduction}
Let $K \subset T^*S$ be admissible. Then {\rm{$\text{Ham}_K^0(T^*S)$}} is the identity component of {\rm{$\text{Symp}_K^0(T^*S)$}} and open in {\rm{$\text{Symp}_K^0(T^*S)$}}. \end{proposition}

One way that Proposition \ref{appendix-reduction} rules out potential pathology is that it follows that $\text{Ham}_K^0(T^*S)$ is the identity component of the group of all Hamiltonian diffeomorphisms of $T^*S$ supported on $K$ and setwise fixing $Z^*S$, without the restriction to flows that setwise fix $Z^*S$.

The first step we take towards Proposition \ref{appendix-reduction} is to show that $\text{Symp}_K^0(M)$ is locally path-connected and has an open identity component. The local path-connectivity of $\text{Symp}(M, \omega)$ for closed symplectic manifolds is due to Weinstein and is a consequence of the Lagrangian neighborhood theorem \cite{weinstein1971symplectic}. We explain how local path-connectivity in our case follows from celebrated work of Ebin and Marsden \cite{ebin1970groups}.

Ebin and Marsden show that for $M$ a compact manifold with boundary and $K$ a closed submanifold of $M$, \begin{align*} \text{Diff}_{\partial M}^K(M) &= \{\varphi \in \text{Diff}(M) \mid \varphi(K) \subset K, \, \varphi|_{\partial M} = \text{Id} \}, \\ 
\text{Diff}_{\partial M}^{K,P}(M) &= \{\varphi \in \text{Diff}(M) \mid \varphi|_{K} = \text{Id}, \, \varphi|_{\partial M} = \text{Id} \}
\end{align*} are inverse limit of Hilbert [ILH] subgroups of $\text{Diff}(M)$ in the sense of Omori, and in particular are locally path-connected. Furthermore, Ebin and Marsden show: 

\begin{theorem}[\cite{ebin1970groups}, Theorem 8.3]\label{exact-boundary-fixed}
    Let $M$ be a manifold with boundary, and $\omega$ be an exact symplectic form on $M$. Then $$\text{{\rm{Symp}}}_{\partial M}(M, \omega) = \{ \varphi \in \text{{\rm{Diff}}}(M) \mid \varphi^* \omega = \omega, \varphi|_{\partial M } = \text{{\rm{Id}}} \}$$ is an ILH subgroup of {\rm{$\text{Diff}(M)$}}.
\end{theorem} Exactness of $\omega$ is essential here: the result is false if the hypothesis is omitted.

Now, let $K \subset T^*S$ be admissible, let $L \supset K$ be a compact submanifold of $T^*S$ containing $K$ in its interior, and let $K'=L - \text{interior}(K)$. The structure we shall use in the following that follows from the results of Ebin-Marsden is that $$\text{Symp}_K^0(T^*S) = \text{Diff}_{\partial L}^{Z^*S}(L) \cap \text{Symp}_{\partial L}(L, \omega_{\text{can}}) \cap \text{Diff}^{K', P}_{\partial L}(L) $$ has an open identity component that is path-connected through smooth paths.

We next turn to the group of Hamiltonian diffeomorphisms as a subgroup of the symplectic diffeomorphisms. In general, for a symplectic manifold $(M, \omega)$ how $\text{Ham}(M)$ sits inside of the identity component $\text{Symp}_0(M)$ is a difficult question and is the subject of recent research (for instance \cite{lalonde1997flux}, \cite{ono2006floer}, \cite{buhovsky2015towards}). However, in the case of cotangent bundles $(T^*M, \omega_{\text{can}})$, it is well-known that that $\text{Ham}_c(T^*M) = \text{Symp}_{c,0}(T^*M)$. To see this, as $(T^*M, \omega_{\text{can}})$ is exact there is an exact sequence $$\begin{tikzcd} 0 \arrow[r] & \text{Ham}_c(T^*M) \arrow[r] & \text{Symp}_{c,0}(T^*M) \arrow[r] & H^1_c(T^*M) \arrow[r] & 0, \end{tikzcd}$$where the first map is inclusion and the second map is given by $\varphi \mapsto [\lambda_{\text{taut}} - \varphi^* \lambda_{\text{taut}}]$. By Poincar\'e duality, $(H^1_c(T^*M))^* \cong H^{2\dim(M) -1}(T^*M) \cong \{0\}$, so that $\text{Ham}_c(T^*M) = \text{Symp}_{c,0}(T^*M)$. A corollary of this is that any isotopy of compactly supported Hamiltonian diffeomorphisms of $T^*M$ is a Hamiltonian flow.

We have now accumulated everything necessary to prove Proposition \ref{appendix-reduction}.

\begin{proof}[Proof of \ref{appendix-reduction}]
Let $\varphi_t$ be a path in $\text{Symp}_K^0(T^*S)$ beginning at $\varphi_0 = \text{Id}$. Then $\varphi_t$ is a path in $\text{Symp}_{c,0}(T^*M)$ and is hence a path in $\text{Ham}_K^0(T^*S)$. This path is represented by a Hamiltonian flow supported on $K$ that setwise fixes $Z^*S$. Path-connectivity of $\text{Symp}_K^0(T^*S)$ through smooth paths gives that $\text{Ham}_K^0(T^*S)$ is the identity component of $\text{Symp}_K^0(T^*S)$, which is open.
\end{proof}

The following lemma and proposition will show that the restriction to Hamiltonian diffeomorphisms generated by flows fixing $Z^*S$ provides the entire identity component of $\text{Symp}_c^0(T^*S)$.

\begin{lemma}\label{symp-id-open}
The identity component {\rm{$(\text{Symp}_c^0(T^*S))_0$}} of {\rm{$\text{Symp}_c^0(T^*S)$}} is open.
\end{lemma}
\begin{proof}
For a fixed admissible $K$, the intersection $(\text{Symp}_c^0(T^*S))_0 \cap \text{Symp}_K^0(T^*S)$ contains the identity component of $\text{Symp}_K^0(T^*S)$. This is because $\text{Symp}_K^0(T^*S)$ is locally path-connected and a path in $\text{Symp}_K^0(T^*S)$ is also contained in $(\text{Symp}_c^0(T^*S))_0$.

So $(\text{Symp}_c^0(T^*S))_0 \cap \text{Symp}_K^0(T^*S)$ is a subgroup containing a neighborhood of the identity in $\text{Symp}_K^0(T^*S)$, and hence open as $\text{Symp}_K^0(T^*S)$ is a topological group. As $\text{Symp}_c^0(T^*S)$ is topologized as the direct limit of $\text{Symp}_K^0(T^*S)$ under inclusion, this gives the claim.
\end{proof}

\begin{proposition}\label{hams-open}
{\rm{$\text{Ham}_c^0(T^*S) = (\text{Symp}_c^0(T^*S))_0$}}. In particular, {\rm{$(\text{Symp}_c^0(T^*S))_0$}} is a subgroup of {\rm{$\text{Symp}_c^0(T^*S)$}}.
\end{proposition}
\begin{proof}
    We first claim that $\text{Ham}_c^0(T^*S)$ is open in $\text{Symp}_c^0(T^*S)$. Once again, we use the definition of the direct limit topology. So fix $K \subset T^*S$ admissible. As $\text{Ham}_c^0(T^*S)$ contains $\text{Ham}_K^0(T^*S)$, Proposition \ref{appendix-reduction} shows that $\text{Ham}_c^0(T^*S) \cap \text{Symp}_K^0(T^*S)$ is a subgroup containing a neighborhood of the identity, hence open. So $\text{Ham}_c^0(T^*S)$ is open in $\text{Symp}_c^0(T^*S).$
    
    The claim would immediately follow if it were known that $\text{Symp}_c^0(T^*S)$ were a topological group (c.f. \cite{tatsuuma1998group}, Theorem 6.1). In place of knowing this, we show that $\text{Symp}_c^0(T^*S) - \text{Ham}_c^0(T^*S)$ is open directly. So let $K \subset T^*S$ be admissible. Then as $\text{Ham}_c^0(T^*S) \cap \text{Symp}_K^0(T^*S)$ is an open subgroup, the complement $(\text{Symp}_c^0(T^*S) - \text{Ham}_c^0(T^*S)) \cap\text{Symp}_K^0(T^*S)$ is also open. By the definition of the direct limit topology, $\text{Symp}_c^0(T^*S) - \text{Ham}_c^0(T^*S)$ is open in $\text{Symp}_c^0(T^*S)$. So $\text{Ham}_c^0(T^*S)$ is a connected component of $\text{Symp}_c^0(T^*S)$.
\end{proof}

The following is an immediate corollary of Lemma \ref{symp-id-open} and the definition of the topologies on $\text{Ham}_T^0(T^*S)$ and $\text{Symp}_T^0(T^*S)$ (see Section \ref{2-complex}).

\begin{corollary}
{\rm{$\text{Ham}_T^0(T^*S)$}} is the identity component of the group of tame symplectic diffeomorphisms {\rm{$\text{Symp}_T^0(T^*S)$}} and open in {\rm{$\text{Symp}_T^0(T^*S)$}}.
\end{corollary}

We now turn towards a proof of Proposition \ref{things-not-totally-awful}. We begin by introducing relevant notation. For $K$ admissible, let $\text{Symp}_{T,K}^0(T^*S)$ be the subgroup of $\text{Symp}_T^0(S)$ consisting of $\varphi$ so that there is a diffeomorphism $f \in \text{Diff}^+(S)$ so that $\varphi(\alpha) = f^\#(\alpha)$ for all $\alpha \notin K$. Similarly, define $\text{Ham}_{T,K}^0(T^*S)$ to be the subgroup of $\text{Ham}_T^0(T^*S)$ of diffeomorphisms agreeing with $f^\#$ for some $f \in \text{Diff}_0(S)$ outside of $K$. For all admissible $K$, let $\mathcal{H}^n_K(S)$ denote the quotient of $ \text{Symp}_{T,K}^0(T^*S)$ by the stabilizer of its action on $\dot{\mathbb{M}}^n(S)$, and let $q_K$ denote the quotient map. Let $q_T: \text{Symp}_T^0(S) \to \mathcal{H}^n(S)$ denote the quotient.

\begin{lemma}
    For $K$ admissible, {\rm{$\text{Ham}_{T,K}^0(S)$}} and {\rm{$\text{Symp}_{T,K}^0(S)$}} are topological groups.
\end{lemma}

\begin{proof}

As the action of $\text{Diff}_0(S)$ on $\text{Ham}_K^0(T^*S)$ by conjugation by lifts $f^\#$ of diffeomorphisms $f \in \text{Diff}_0(S)$ is continuous, and the analogous action of $\text{Diff}^+(S)$ on $\text{Symp}_{K}^0(T^*S)$ is continuous, the semidirect products $\text{Ham}_{T,K}^0(S) = \text{Diff}_0(S) \ltimes \text{Ham}_K^0(T^*S)$ and $\text{Symp}_{T,K}^0(T^*S) = \text{Diff}^+(S) \ltimes \text{Symp}_K^0(T^*S)$ are topological groups.
\end{proof}

The principal difficulty of the remainder of the appendix is to show that $\mathcal{H}^n(S)$ inherits a topological group structure from its definition as a quotient of $\text{Symp}_T^0(T^*S)$ (c.f. \cite{tatsuuma1998group} Theorem 6.1). An outline of our approach is to confirm that the group operations on $\mathcal{H}^n(S)$ are continuous by locating components of $\mathcal{H}^n(S)$ inside $\mathcal{H}_K^n(S)$ and reducing to the fact that $\mathcal{H}_K^n(S)$ is a topological group for $K$ admissible. Our structural results for $\mathcal{G}^n(S)$ from Section \ref{section-group-structure} are essential for this. The way that we involve our analysis of $\mathcal{G}^n(S)$ in Section \ref{section-group-structure} in the proof is following lemma.

\begin{lemma}\label{lemma-appendix-leverage}
    Let $\mathcal{G}^n_K(S) \subset \mathcal{H}_K^n(S)$ be the quotient of {\rm{$\text{Ham}_{T,K}^0(S)$}} by its stabilizer on $\dot{\mathbb{M}}^n(S)$ and $\widetilde{\mathcal{G}}^n(S)$ the quotient of {\rm{$\text{Ham}_{T}^0(S)$}} by its stabilizer on $\dot{\mathbb{M}}^n(S)$. Then $\mathcal{G}^n_K(S)$ is the identity component of $\mathcal{H}^n_K(S)$ for all admissible $K$. Furthermore, if $K' \supset K$ is admissible the map $\mathcal{G}^n_K(S) \to \mathcal{G}^n_{K'}(S)$ induced by the inclusion {\rm{$\text{Ham}_{T,K}^0(T^*S) \to \text{Ham}_{T,K'}^0(T^*S)$}} is an isomorphism of topological groups.

    Additionally, $\widetilde{\mathcal{G}}^n(S)$ is the identity component of $\mathcal{H}^n(S)$ and open in $\mathcal{H}^n(S)$, and the maps $\mathcal{G}^n_K(S) \to \widetilde{\mathcal{G}}^n(S)$ induced by the inclusions {\rm{$\text{Ham}_{T,K}(S) \to \text{Ham}_T^0(T^*S)$}} are all isomorphisms of topological groups.
\end{lemma}

\begin{proof}
For all admissible $K$, Lemma \ref{appendix-reduction} and the fact that quotients in the category of topological groups are open shows that $\mathcal{G}^n_K(S)$ is a connected open subgroup of $\mathcal{H}^n_K(S)$ containing the identity, hence the identity component of $\mathcal{H}^n_K(S)$. The proof of Proposition \ref{quotient-topology-is-product} shows that the maps $\mathcal{G}_K^n(S) \to \mathcal{G}^n_{K'}(S)$ and $\mathcal{G}^n_K(S) \to \widetilde{\mathcal{G}}^n(S)$ induced by appropriate inclusions are all homeomorphisms.

So what remains to be shown is that $\widetilde{\mathcal{G}}^n(S)$ is the identity component of $\mathcal{H}^n(S)$ and open in $\mathcal{H}^n(S)$. The proof is reminiscent of Lemma \ref{hams-open}. We must show that $q_T^{-1}(\widetilde{\mathcal{G}}^n(S))$ and $q_T^{-1}((\widetilde{\mathcal{H}}^n(S) - \widetilde{\mathcal{G}}^n(S))$ are both open, which is verified by showing their intersections with $\text{Symp}_K^0(T^*S)$ are open for all admissible $K$.

If $K$ is admissible, then $q_T^{-1}( \widetilde{\mathcal{G}}^n(S)) \cap \text{Symp}_K^0(T^*S) = q_K^{-1}(\mathcal{G}^n_K(S))$, which is open as $\mathcal{G}_K^n(S)$ is open. Also, $q_T^{-1}( \mathcal{H}^n(S) - \widetilde{\mathcal{G}}^n(S)) \cap \text{Symp}_K^0(T^*S) = q_K^{-1}(\mathcal{H}^{n}_K(S) - \mathcal{G}^n_K(S))$, which is also open as open subgroups of topological groups are closed. This proves that $\widetilde{\mathcal{G}}^n(S)$ is the identity component in $\mathcal{H}^n(S)$ and open in $\mathcal{H}^n(S)$, as desired.
\end{proof}

In the proof of the following lemma, we transfer our knowledge of $\mathcal{G}^n(S)$ to arbitrary connected components of $\mathcal{H}^n(S)$.

\begin{lemma}\label{components-lemma}
    Let $\varphi$ represent an element $[\varphi] \in \mathcal{H}^n(S)$. Let $K$ be an admissible submanifold of $T^*S$ containing the support of $\varphi$, and let $U_K$ be the connected component of $\mathcal{H}^n_K(S)$ containing $[\varphi]$. Viewed as a subset of $\mathcal{H}^n(S)$, the set $U_K$ is the connected component of $\mathcal{H}^n(S)$ containing $[\varphi]$. The inclusion $\mathcal{H}^n_K(S) \to \mathcal{H}^n(S)$ restricts to a homeomorphism of $U_k$ onto its image.
\end{lemma}

\begin{proof}
    Let $\varphi$ and $K$ be as given. Let $U_K$ be the component of $\mathcal{H}^n_K(S)$ containing $[\varphi]$ and consider $U_K$ as a subgroup of $\mathcal{H}^n(S)$. We claim that $U_k$ is an open connected component of $\mathcal{H}^n(S)$. The criterion used is the same as in Lemma \ref{lemma-appendix-leverage} and Lemma \ref{hams-open}. We begin by showing that $U_K \subset \mathcal{H}^n(S)$ is open.
    
    As $\mathcal{H}^n_K(S)$ is a topological group, Lemma \ref{lemma-appendix-leverage} shows that $U_K$ is the open connected component of $\mathcal{H}^n_K(S)$ given by $[\varphi] \mathcal{G}^n_K(S)$. For any admissible $K' \supset K$, we have that $[\varphi]\mathcal{G}_{K'}^n(S)$ is the connected component of $\mathcal{H}^n_{K'}(S)$ containing $U_K$. Using Lemma \ref{lemma-appendix-leverage} to identify $\mathcal{G}^n_{K'}$ and the inclusion of $\mathcal{G}_K^n(S)$ into $\mathcal{H}^n_{K'}(S)$, we see that the image of $U_K$ in $\mathcal{H}^n_{K'}(S)$ under inclusion is open. In fact, it is a connected component of $\mathcal{H}^n_{K'}(S)$. So we see that $q_T^{-1}(U_K) \cap \text{Symp}_{T,K'}^0(T^*S) = q_{K'}^{-1}(U_K)$ is open. As $K' \supset K$ was arbitrary, $q^{-1}_T(U_K)$ is open in $\text{Symp}_{T}^0(T^*S)$ and so $U_K$ is open in $\mathcal{H}^n(S)$.
    
    We now claim that $\mathcal{H}^n(S) - U_K$ is open in $\mathcal{H}^n(S)$. To see this, let $[\psi] \in \mathcal{H}^n(S) - U_K$ be represented by $\psi \in \text{Symp}_{T,L}^0(T^*S)$ for some admissible $L$. Let $L'$ be admissible and contain $K$ and $L$, and let $V$ be the connected component of $\mathcal{H}^n_{L}(S)$ containing $[\psi]$. Since $[\psi]$ is not in the connected component $U_K$ of $\mathcal{H}^n_L(S)$, we must have $V \cap U_K = \emptyset$. Then by the argument used to show that $U_K$ is open, $V$ is open in $\mathcal{H}^n(S)$. As $V$ contains $[\psi]$ and has empty intersection with $U_K$, we see that $\mathcal{H}^n(S) - U_K$ is open. We conclude that $U_K$ is a connected component of $\mathcal{H}^n(S)$.
    
    We finish the proof by showing that the inclusion of $U_K$ into $\mathcal{H}^n(S)$ is a homeomorphism with its image. First, let $U \subset U_K$ be open in $\mathcal{H}^n(S)$. Then $q_{T}^{-1}(U)$ is open in $\text{Symp}_{T}^0(T^*S)$, so $q_{K}^{-1}(U) = \text{Symp}_{T,K}^0(T^*S) \cap q_{T}^{-1}(U)$ is open and we conclude that $U$ is open in $\mathcal{H}^n_K(S)$.
    
    On the other hand, let $U \subset U_K$ be open in $\mathcal{H}^n_K(S)$. Then for any admissible $K' \supset K$, we have that $U = [\varphi]([\varphi]^{-1}U)$, where $[\varphi]: \mathcal{H}^n_K(S) \to \mathcal{H}^n_K(S)$ is a homeomorphism and $[\varphi]^{-1}U \subset \mathcal{G}^n_{K'}(S)$. Now, $U$ is open in $\mathcal{H}^n_{K'}(S)$. This is because $[\varphi]^{-1}(U)$ is open in $\mathcal{G}^n_K(S)$ by hypothesis and the fact that $\mathcal{H}^n_K(S)$ is a topological group, and the inclusion $\mathcal{G}^n_K(S) \to \mathcal{G}^n_{K'}(S)$ is a homeomorphism by Lemma \ref{lemma-appendix-leverage}. So we have $q_T^{-1}(U) \cap \text{Symp}_{T,K'}^0(T^*S) = q_{K'}^{-1}(U)$ is open. We conclude that $U$ is open in $\mathcal{H}^n(S)$, as desired.
\end{proof}

We are at last ready to prove Theorem \ref{things-not-totally-awful}.

\begin{proof}[Proof of \ref{things-not-totally-awful}]
We begin by showing statement (a): that $\mathcal{H}^n(S)$ is a topological group. Let $[\varphi], [\psi] \in \mathcal{H}^n(S)$ have representatives $\psi$ and $\varphi$. Let $K$ be an admissible submanifold so that $\psi, \varphi \in \text{Symp}_{T,K}^0(T^*S)$. Let $U_K, V_K,$ and $W_K$ be the connected components of $\mathcal{H}^n_K(S)$ containing $[\varphi], [\psi]$, and $[\varphi][\psi]$ respectively.

As $\mathcal{H}^n_K(S)$ is a topological group, the restriction of the product to $U_K \times V_K \to W_K$ is continuous. Lemma \ref{components-lemma} shows that the inclusions of $U_K, V_K$ and $W_K$ into $\mathcal{H}^n(S)$ are homeomorphisms, so that the product on the open set $U_K \times V_K \subset \mathcal{H}^n(S)\times \mathcal{H}^n(S)$ is continuous. We conclude from the local characterization of continuity that the product $\mathcal{H}^n(S) \times \mathcal{H}^n(S) \to \mathcal{H}^n(S)$ is continuous. The proof of continuity of inversion is similar.

Lemma \ref{lemma-appendix-leverage} shows statement (b). Statement (c) follows immediately from (a) and (b). 
\end{proof}

\bibliographystyle{plain}
\bibliography{refs}

\begin{thebibliography}{10}

\bibitem{ahlfors1961some}
Lars Ahlfors.
\newblock Some remarks on {T}eichm\"{u}ller's space of {R}iemann surfaces.
\newblock {\em Ann. of Math.}, pages 171--191, 1961.

\bibitem{ahlfors1960riemann}
Lars Ahlfors and Lipman Bers.
\newblock Riemann's mapping theorem for variable metrics.
\newblock {\em Ann. of Math.}, pages 385--404, 1960.

\bibitem{bers1960simultaneous}
Lipman Bers.
\newblock Simultaneous uniformization.
\newblock {\em Bull. Amer. Math. Soc.}, pages 94--97, 1960.

\bibitem{bers1961correction}
Lipman Bers.
\newblock Correction to ``{S}paces of {R}iemann surfaces as bounded domains''.
\newblock {\em Bull. Amer. Math. Soc.}, pages 465--466, 1961.

\bibitem{bers1961holomorphic}
Lipman Bers.
\newblock Holomorphic differentials as functions of moduli.
\newblock {\em Bull. Amer. Math. Soc.}, pages 206--210, 1961.

\bibitem{bridgeman2015pressure}
Martin Bridgeman, Richard Canary, Fran\c{c}ois Labourie, and Andres Sambarino.
\newblock The pressure metric for {A}nosov representations.
\newblock {\em Geom. Funct. Anal.}, pages 1089--1179, 2015.

\bibitem{buhovsky2015towards}
Lev Buhovsky.
\newblock Towards the {$C^0$} flux conjecture.
\newblock {\em Algebr. Geom. Topol.}, pages 3493--3508, 2014.

\bibitem{choi1993convex}
Suhyoung Choi and William Goldman.
\newblock Convex real projective structures on closed surfaces are closed.
\newblock {\em Proc. Amer. Math. Soc.}, pages 657--661, 1993.

\bibitem{d-Hoker2015higher}
Eric D'Hoker and Duong~H. Phong.
\newblock Higher order deformations of complex structures.
\newblock {\em SIGMA Symmetry Integrability Geom. Methods Appl.}, 11:Paper 047,
  14, 2015.

\bibitem{earle1969fibre}
Clifford~J. Earle and James Eells.
\newblock A fibre bundle description of {T}eichm\"{u}ller theory.
\newblock {\em J. Differential Geometry}, pages 19--43, 1969.

\bibitem{ebin1970manifold}
David Ebin.
\newblock The manifold of {R}iemannian metrics.
\newblock In {\em Global {A}nalysis ({P}roc. {S}ympos. {P}ure {M}ath., {V}ol.
  {XV}, {B}erkeley, {C}alif., 1968)}, pages 11--40. Amer. Math. Soc.,
  Providence, R.I., 1970.

\bibitem{ebin1970groups}
David Ebin and Jerrold Marsden.
\newblock Groups of diffeomorphisms and the motion of an incompressible fluid.
\newblock {\em Ann. of Math.}, pages 102--163, 1970.

\bibitem{feix2001hyperkahler}
Birte Feix.
\newblock Hyperk\"{a}hler metrics on cotangent bundles.
\newblock {\em J. Reine Angew. Math.}, pages 33--46, 2001.

\bibitem{fischer1984purely}
Arthur~E. Fischer and Anthony~J. Tromba.
\newblock On a purely ``{R}iemannian'' proof of the structure and dimension of
  the unramified moduli space of a compact {R}iemann surface.
\newblock {\em Math. Ann.}, pages 311--345, 1984.

\bibitem{fock2021higher}
Vladimir Fock and Alexander Thomas.
\newblock Higher complex structures.
\newblock {\em Int. Math. Res. Not. IMRN}, pages 15873--15893, 2021.

\bibitem{goldman1990convex}
William Goldman.
\newblock Convex real projective structures on compact surfaces.
\newblock {\em J. Differential Geom.}, pages 791--845, 1990.

\bibitem{guichard2008convex}
Olivier Guichard and Anna Wienhard.
\newblock Convex foliated projective structures and the {H}itchin component for
  {${\rm PSL}_4({\bf R})$}.
\newblock {\em Duke Math. J.}, pages 381--445, 2008.

\bibitem{guichard2012anosov}
Olivier Guichard and Anna Wienhard.
\newblock Anosov representations: domains of discontinuity and applications.
\newblock {\em Invent. Math.}, pages 357--438, 2012.

\bibitem{hitchin1987self}
Nigel Hitchin.
\newblock The self-duality equations on a {R}iemann surface.
\newblock {\em Proc. London Math. Soc.}, pages 59--126, 1987.

\bibitem{hitchin1992lie}
Nigel Hitchin.
\newblock Lie groups and {T}eichm\"{u}ller space.
\newblock {\em Topology}, pages 449--473, 1992.

\bibitem{kaledin1997hyperkaehler}
D.~Kaledin.
\newblock A canonical hyperk\"{a}hler metric on the total space of a cotangent
  bundle.
\newblock In {\em Quaternionic structures in mathematics and physics ({R}ome,
  1999)}, pages 195--230. Univ. Studi Roma ``La Sapienza'', Rome, 1999.

\bibitem{kim2017kahler}
Inkang Kim and Genkai Zhang.
\newblock K\"{a}hler metric on the space of convex real projective structures
  on surface.
\newblock {\em J. Differential Geom.}, pages 127--137, 2017.

\bibitem{kolar2013natural}
Ivan Kol\'{a}\v{r}, Peter~W. Michor, and Jan Slov\'{a}k.
\newblock {\em Natural operations in differential geometry}.
\newblock Springer-Verlag, Berlin, 1993.

\bibitem{labourie2004anosov}
Fran\c{c}ois Labourie.
\newblock Anosov flows, surface groups and curves in projective space.
\newblock {\em Invent. Math.}, pages 51--114, 2006.

\bibitem{labourie2007flat}
Fran\c{c}ois Labourie.
\newblock Flat projective structures on surfaces and cubic holomorphic
  differentials.
\newblock {\em Pure Appl. Math. Q.}, 2007.

\bibitem{labourie2008cross}
Fran\c{c}ois Labourie.
\newblock Cross ratios, {A}nosov representations and the energy functional on
  {T}eichm\"{u}ller space.
\newblock {\em Ann. Sci. \'{E}c. Norm. Sup\'{e}r.}, pages 437--469, 2008.

\bibitem{labourie2017cyclic}
Fran\c{c}ois Labourie.
\newblock Cyclic surfaces and {H}itchin components in rank 2.
\newblock {\em Ann. of Math.}, pages 1--58, 2017.

\bibitem{labourie2018variations}
Fran\c{c}ois Labourie and Richard Wentworth.
\newblock Variations along the {F}uchsian locus.
\newblock {\em Ann. Sci. \'{E}c. Norm. Sup\'{e}r.}, pages 487--547, 2018.

\bibitem{lalonde1997flux}
Fran\c{c}ois Lalonde, Dusa McDuff, and Leonid Polterovich.
\newblock On the flux conjectures.
\newblock In {\em Geometry, topology, and dynamics ({M}ontreal, {PQ}, 1995)},
  volume~15 of {\em CRM Proc. Lecture Notes}, pages 69--85. Amer. Math. Soc.,
  Providence, RI, 1998.

\bibitem{loftin2001affine}
John Loftin.
\newblock Affine spheres and convex {$\Bbb{RP}^n$}-manifolds.
\newblock {\em Amer. J. Math.}, pages 255--274, 2001.

\bibitem{loftin2010survey}
John Loftin.
\newblock Survey on affine spheres.
\newblock In {\em Handbook of geometric analysis, {N}o. 2}, volume~13 of {\em
  Adv. Lect. Math. (ALM)}, pages 161--191. Int. Press, Somerville, MA, 2010.

\bibitem{markovic2022non}
Vladimir Markovi\'{c}.
\newblock Non-uniqueness of minimal surfaces in a product of closed {R}iemann
  surfaces.
\newblock {\em Geom. Funct. Anal.}, pages 31--52, 2022.

\bibitem{mcduff2017introduction}
Dusa McDuff and Dietmar Salamon.
\newblock {\em Introduction to symplectic topology}.
\newblock Oxford Graduate Texts in Mathematics. Oxford University Press,
  Oxford, second edition, 1998.

\bibitem{milnor1984remarks}
J.~Milnor.
\newblock Remarks on infinite-dimensional {L}ie groups.
\newblock In {\em Relativity, groups and topology, {II} ({L}es {H}ouches,
  1983)}, pages 1007--1057. North-Holland, Amsterdam, 1984.

\bibitem{nomizu1994affine}
Katsumi Nomizu and Takeshi Sasaki.
\newblock {\em Affine differential geometry}, volume 111 of {\em Cambridge
  Tracts in Mathematics}.
\newblock Cambridge University Press, Cambridge, 1994.
\newblock Geometry of affine immersions.

\bibitem{omori1970group}
Hideki Omori.
\newblock On the group of diffeomorphisms on a compact manifold.
\newblock In {\em Global {A}nalysis ({P}roc. {S}ympos. {P}ure {M}ath., {V}ol.
  {XV}, {B}erkeley, {C}alif., 1968)}, pages 167--183. Amer. Math. Soc.,
  Providence, R.I., 1970.

\bibitem{omori1996infinite}
Hideki Omori.
\newblock {\em Infinite-dimensional {L}ie groups}, volume 158 of {\em
  Translations of Mathematical Monographs}.
\newblock American Mathematical Society, Providence, RI, 1997.
\newblock Translated from the 1979 Japanese original and revised by the author.

\bibitem{kobayashi1983regular}
Hideki Omori, Yoshiaki Maeda, Akira Yoshioka, and Osamu Kobayashi.
\newblock On regular {F}r\'{e}chet-{L}ie groups. {V}. {S}everal basic
  properties.
\newblock {\em Tokyo J. Math.}, pages 39--64, 1983.

\bibitem{ono2006floer}
Kaoru Ono.
\newblock Floer-{N}ovikov cohomology and the flux conjecture.
\newblock {\em Geom. Funct. Anal.}, pages 981--1020, 2006.

\bibitem{tatsuuma1998group}
Nobuhiko Tatsuuma, Hiroaki Shimomura, and Takeshi Hirai.
\newblock On group topologies and unitary representations of inductive limits
  of topological groups and the case of the group of diffeomorphisms.
\newblock {\em J. Math. Kyoto Univ.}, pages 551--578, 1998.

\bibitem{thomas2020higher}
Alexander Thomas.
\newblock Higher complex structures and flat connections.
\newblock 2020.
\newblock \href{https://arxiv.org/abs/2005.14445}{arxiv:2005.14445}.

\bibitem{thomas2020thesis}
Alexander Thomas.
\newblock Higher complex structures and higher {Teichm\"uller} theory.
\newblock 2020.
\newblock \href{https://arxiv.org/abs/2007.00382}{arxiv:2007.00382}.

\bibitem{thomas2021differential}
Alexander Thomas.
\newblock Differential operators on surfaces and rational {WKB} method.
\newblock 2021.
\newblock \href{https://arxiv.org/abs/2111.07946}{arxiv:2111.07946}.

\bibitem{changping1991some}
Chang~Ping Wang.
\newblock Some examples of complete hyperbolic affine {$2$}-spheres in {${\bf
  R}^3$}.
\newblock In {\em Global differential geometry and global analysis ({B}erlin,
  1990)}, volume 1481 of {\em Lecture Notes in Math.}, pages 271--280.
  Springer, Berlin, 1991.

\bibitem{weinstein1971symplectic}
Alan Weinstein.
\newblock Symplectic manifolds and their {L}agrangian submanifolds.
\newblock {\em Advances in Math.}, pages 329--346, 1971.

\bibitem{wolf1989teichmuller}
Michael Wolf.
\newblock The {T}eichm\"{u}ller theory of harmonic maps.
\newblock {\em J. Differential Geom.}, pages 449--479, 1989.

\bibitem{wolpert1986chern}
Scott Wolpert.
\newblock Chern forms and the {R}iemann tensor for the moduli space of curves.
\newblock {\em Invent. Math.}, pages 119--145, 1986.

\bibitem{wolpert1990bers}
Scott Wolpert.
\newblock The {B}ers embeddings and the {W}eil-{P}etersson metric.
\newblock {\em Duke Math. J.}, pages 497--508, 1990.

\bibitem{wolpert2017equiboundedness}
Scott~A. Wolpert.
\newblock Equiboundedness of the {W}eil-{P}etersson metric.
\newblock {\em Trans. Amer. Math. Soc.}, pages 5871--5887, 2017.

\end{thebibliography}

\end{document}